\documentclass[reqno,a4paper]{amsart}
\usepackage[english,activeacute]{babel}
\usepackage{amssymb,amsmath,amsthm,amsfonts,mathrsfs,hyperref,mathtools}
\usepackage[font=footnotesize]{caption}
\usepackage{tikz}
\usetikzlibrary{intersections,shapes,trees,arrows,spy,positioning,decorations,arrows.meta,decorations.pathreplacing,patterns,calc}
\usepackage{tikz-cd}
\usepackage{tkz-euclide}
\usepackage{pgfplots}
\usepackage{xcolor}
\usepackage[numbers,square,sort]{natbib}
\usepackage{tabularx}
\usepackage{enumitem}  %Proper enumerate 
\usepackage[T1]{fontenc}
\usepackage[foot]{amsaddr}
\usepackage{url}  

\usetikzlibrary{decorations.markings}

\theoremstyle{plain}
\newtheorem{theorem}{Theorem}[section]
\newtheorem{lem}[theorem]{Lemma}
\newtheorem{lemma}[theorem]{Lemma}
\newtheorem{proposition}[theorem]{Proposition}
\newtheorem{corollary}[theorem]{Corollary}
\theoremstyle{definition}
\newtheorem{algorithm}[theorem]{Algorithm}
\newtheorem{definition}[theorem]{Definition}

\theoremstyle{remark}
\newtheorem{remark}[theorem]{Remark}

\numberwithin{equation}{section}

\newcommand{\noo}[1]{{]{#1}]}}

\newcommand{\Nb}  {{\mathbb N}}
\newcommand{\Rb}  {{\mathbb R}}

\newcommand{\Pb}  {{\mathbb P}}

\newcommand{\Tb}  {{\mathbb T}}

\newcommand{\Cs} {{\mathcal C}}
\newcommand{\Ds} {{\mathcal D}}
\newcommand{\Es} {{\mathcal E}}

\newcommand{\Gs} {{\mathcal G}}

\newcommand{\Ls} {{\mathcal L}}
\newcommand{\Ms} {{\mathcal M}}

\newcommand{\Ps} {{\mathcal P}}

\newcommand{\Rs} {{\mathcal R}}
\newcommand{\Ss} {{\mathcal S}}

\newcommand{\Vs} {{\mathcal V}}
\newcommand{\Us} {{\mathcal U}}

\newcommand{\conv}[1] {\text{conv}\left({#1}\right)}

\newcommand{\dd}{\, \mathrm{d}}

\renewcommand{\phi}{\varphi}

\renewcommand{\P}{\mathbb P}

\newcommand{\E}{\mathbb E}
\newcommand{\R}{\mathbb R}

\newcommand{\N}{\mathbb N}

\newcommand{\ind}{1\!\kern-1pt \mathrm{I}}
\newcommand{\rsto}{]\!\kern-1.8pt ]}
\newcommand{\lsto}{[\!\kern-1.7pt [}

\newcommand{\ccS}{\mathscr{S}}

\newcommand\F{\mbox{I\kern-2pt F}}
\newcommand\1{{\bf 1}}

\newcommand\cC{{\mathcal C}}
\newcommand\cG{{\mathcal G}}

\newcommand\cM{{\mathcal M}}

\newcommand\cD{{\mathcal D}}
\newcommand\cR{{\mathcal R}}

\newcommand\cS{{\mathcal S}}

\newcommand\dps{{\mathrm{D}+}}
\newcommand\dns{{\mathrm{D}-}}
\newcommand\rps{{\mathrm{R}+}}
\newcommand\rns{{\mathrm{R}-}}

\newcommand{\bindist}[2]{\textrm{Bin}({#1},{#2})}
\newcommand{\hypdist}[3]{\textrm{Hyp}({#1},{#2},{#3})}

\pgfdeclarepatternformonly{soft crosshatch}{\pgfqpoint{-1pt}{-1pt}}{\pgfqpoint{4pt}{4pt}}{\pgfqpoint{3pt}{3pt}}%
{
	\pgfsetstrokeopacity{1}
	\pgfsetlinewidth{0.4pt}
	\pgfpathmoveto{\pgfqpoint{3.1pt}{0pt}}
	\pgfpathlineto{\pgfqpoint{0pt}{3.1pt}}
	\pgfpathmoveto{\pgfqpoint{0pt}{0pt}}
	\pgfpathlineto{\pgfqpoint{3.1pt}{3.1pt}}
	\pgfusepath{stroke}
}

\usetikzlibrary{decorations.markings}
\tikzset
{marking1/.style=
	{decoration=
		{markings,
			mark= between positions 0.03 and 0.97 step 5 mm with {\arrow[line width=0.5pt]{>}}
		},
		postaction=decorate
	}
}

\title{General selection models: Bernstein duality and minimal ancestral structures}

\author{F. Cordero$^{1}$}

\address{$^1$Faculty of Technology, Bielefeld University, Box 100131, 33501 Bielefeld, Germany}
\email{fcordero@techfak.uni-bielefeld.de}

\author{S. Hummel$^1$}
\email{shummel@techfak.uni-bielefeld.de}
\email{emmanuel.schertzer@univie.ac.at}
\author{E. Schertzer$^{2}$}

\address{$^2$Faculty of Mathematics, University of Vienna, Oskar-Morgenstern-Platz 1, 1090 Wien, Austria}

\date{\today}%
\begin{document}
\maketitle
\begin{abstract}
	$\Lambda$-Wright--Fisher processes provide a robust framework to describe the type-frequency evolution of an infinite neutral population. We add a polynomial drift to the corresponding stochastic differential equation to incorporate frequency-dependent selection. A decomposition of the drift allows us to approximate the solution of the stochastic differential equation by a sequence of Moran models. The genealogical structure underlying the Moran model leads in the large population limit to a generalisation of the ancestral selection graph of Krone and Neuhauser. Building on this object, we construct a continuous-time Markov chain and relate it to the forward process via a new form of duality, which we call Bernstein duality. We adapt classical methods based on the moment duality to determine the time to absorption and criteria for the accessibility of the boundaries; this extends a recent result by Gonz{\'a}lez Casanova and Span{\`o}. An intriguing feature of the construction is that the same forward process is compatible with multiple backward models. In this context we introduce suitable notions for minimality among the ancestral processes and characterise the corresponding parameter sets. In this way we recover classic ancestral structures as minimal ones.
\end{abstract}

%\begin{abstract}
%	$\Lambda$-Wright--Fisher processes model the type-frequency evolution in large populations. We augment the corresponding stochastic differential equation by a general polynomial drift vanishing at the boundary and show that its solution can be approximated by several Moran models with frequency-dependent selection. The model inherent genealogical structure in the large population limit is a generalisation of the ancestral selection graph of Krone and Neuhauser. We introduce a Markovian backward process that tracks the conditional type distribution in a sample given partial observation of its ancestry. Forward and backward process are related via a new form of duality, which we call Bernstein duality. We show that many methods based on the classic moment duality extend to this new relation; for example, we determine the time to absorption and criteria for the accessibility of the boundaries, this extends a recent result by Gonz{\'a}lez Casanova and Span{\`o}. An intriguing feature of the construction is that the same forward process is compatible with multiple backward models. In this context we introduce suitable notions for minimality among the ancestral processes and characterise the corresponding parameter sets. In this way we recover classic ancestral structures as minimal ones.
%\end{abstract}
\bigskip { \footnotesize
\noindent{\slshape\bfseries MSC 2020.} Primary:\, 60K35, 92D15  \ Secondary:\, 60G99, 60J25, 60J27

\medskip 
\noindent{\slshape\bfseries Keywords.} $\Lambda$-Wright--Fisher~process, duality, frequency-dependent selection, branching-coalescing system, ancestral selection graph, absorption probability, coming down from infinity}

\setcounter{tocdepth}{1}
\tableofcontents
\section{Introduction}\label{s1}
{\bf Forward and backward in neutral models.} There are essentially two approaches to study a population genetic model: prospective and retrospective. Most prospective processes track the type frequencies of an evolving population. Retrospective processes are related to potential genealogies of present populations and they give rise to branching-coalescing structures. Both approaches are usually related via duality~\citep{Jaku}.

\smallskip 

In a neutral infinite population setting, the type-frequency evolution forward in time is typically described via a driftless stochastic differential equation (SDE); whereas a coalescent process describes the genealogy of a sample from the population. In this context prominent examples are the neutral \emph{Wright--Fisher diffusion} and \emph{Kingman's coalescent}~\citep{kingman1982coalescent}. The neutral two-type Wright--Fisher diffusion $(X_t:\,t\geq 0)$ describes the frequency of one type forward in time and it satisfies the SDE
\begin{equation*}
\dd X_t \ =\ \sqrt{X_t(1-X_t)}\, \dd W_t,
\label{eq:SDEoriginal00}
\end{equation*}
with an initial frequency $X_0=x\in[0,1]$, where $W=(W_t:\,t\geq 0)$ is a standard Brownian motion. The line-counting process of Kingman's coalescent $(L_t:\,t\geq 0)$ describes backward in time the number of ancestors of a sample containing $L_0=n\in \N$ individuals. It is well-known that the formal link between the two processes is a \textit{moment duality}
\begin{equation}\label{eq:duality-lambda}   
\E_x\left[X_t^n\right] \ = \ \E_{n}\left[ x^{L_t} \right], \qquad \forall x\in[0,1],\, \forall n\in \N.
\end{equation}

The graphical representation of finite population models as interactive particle systems provides a way to embed forward and backward process into the same picture. For example, in a \emph{neutral Moran model} with population size~$N$, reproduction events that involve two individuals, the parent and the replaced individual, turn into a binary coalescence event in the backward picture, i.e. two individuals sharing a common ancestor~\citep{etheridge2011some}. This leads to a natural coupling (a pathwise duality) between the type-frequency process and the \emph{$N$-Kingman coalescent}, which traces back the genealogy of the entire population. The duality relation~\eqref{eq:duality-lambda} can be recovered by considering an appropriate large population limit.

\smallskip

For a large class of models, the forward evolution converges to the Wright--Fisher diffusion, and the corresponding genealogy to Kingman's coalescent~\citep{MM00, MS01}. However, these approximations are inappropriate if the variance of an individual's offspring size is asymptotically infinite. Backward in time this leads to consider coalescents with multiple mergers of ancestral lines. They are called~\emph{$\Lambda$-coalescents} and were independently introduced by \citep{DK99}, \citep{Pit99}, and \citep{Sa99} (see \citep{B09} for a review). Forward in time, this leads to jumps in the type-frequency paths. In the two-type case the corresponding type-frequency process is the \emph{neutral $\Lambda$-Wright--Fisher process} $(X_t:\,t\geq 0)$ that evolves according to the SDE
\begin{equation}\label{eq:SDEoriginal01}
\dd X_t\ =\ \sqrt{\Lambda(\{0\})X_t(1-X_t)}\,\dd W_t+\! \! \! \! \int\limits_{(0,1]\times [0,1]}r\big(\1_{\{u\leq X_{t-}\}}(1-X_{t-})-\1_{\{u>X_{t-}\}} X_{t-} \big)\tilde{N}(\dd t, \dd r,\dd u),
\end{equation}
with $X_0=x\in[0,1]$, where $W$ is as before and $\tilde{N}=\tilde{N}(\dd t,\dd r,\dd u)$ is an independent compensated Poisson measure on $[0,\infty)\times(0,1]\times[0,1]$ with intensity $\dd t\times r^{-2}\Lambda(\dd r)\times \dd u$, and~$\Lambda$ is a finite measure on~$[0,1]$ (see \cite{BLG03} for details). Once again a moment duality relates the type-frequency process with the line-counting process of the associated $\Lambda$-coalescent~\citep{etheridge2011some}.

\smallskip

{\bf Forward and backward in models with selection.} Understanding the interplay of selection and random reproduction is a major challenge in population genetics. In an infinite population with two types, the effect of selection on the type-frequency evolution is usually captured by a drift $d(y)=y(1-y) s(y)$, $y\in[0,1]$, for some function $s$; the factor $y(1-y)$ represents the fraction of individuals of one type being replaced by the offspring of the other type in a small time horizon; the function $s$ encodes the direction and strength of selection, and it can depend on the current type composition. The resulting process $X\coloneqq(X_t:\,t\geq 0)$ is called the \emph{$\Lambda$-Wright--Fisher process with frequency-dependent selection} and it satisfies
\begin{equation}\begin{aligned}
\dd X_t\ &= \ d(X_t)\dd t\ + \  \sqrt{\Lambda(\{0\})X_t(1-X_t)}\, \dd W_t\\
&\ \ +\! \! \! \! \int\limits_{(0,1]\times [0,1]}r\Big(\1_{\{u\leq X_{t-}\}}(1-X_{t-})-\1_{\{u>X_{t-}\}} X_{t-} \Big)\tilde{N}(\dd t, \dd r,\dd u), \quad X_0=x\in[0,1],
\end{aligned}\label{eq:SDEoriginal02}
\end{equation}
where $W$ and $\tilde{N}$ are as above. The potential genealogies in the $\Lambda$-Wright--Fisher process with \emph{genic} selection, i.e. with $s$ constant, can be described by means of the $\Lambda$-ancestral selection graph ($\Lambda$-ASG) \citep{BLW16} (see \cite{etheridge2009,EGT10} for an alternative approach). It was originally constructed by \citeauthor{KroNe97} \citep{KroNe97,NeKro97} for the Wright--Fisher diffusion, i.e. $\Lambda=\delta_0$. The $\Lambda$-ASG is a branching-coalescing process; the coalescence mechanism is as in the $\Lambda$-coalescent; genic selection induces binary branchings at constant rate per ancestral line. \citet{gonzalezcasanova2018} generalise the $\Lambda$-ASG to the case of $s$ being a power series with negative non-decreasing coefficients. In this setting $X$ is in moment duality to the line-counting process of the $\Lambda$-ASG~\citep{foucart2013impact,gonzalezcasanova2018}. Even though classic methods of diffusion theory do not apply to the forward process, it can still be analysed using methods based on the moment duality. This led, for example, to statements about the accessibility of the boundaries \citep[Thm.~4.6]{gonzalezcasanova2018} and the time to absorption~\citep[Thm.~4.3]{bah2015} of~$X$.

\smallskip

Beyond the setting considered in \citep{gonzalezcasanova2018} the picture is more involved. A good example of this is the Wright--Fisher diffusion with \emph{balancing selection}, i.e. $\Lambda=\delta_0$ and $s(y)=1-2y$. The corresponding ASG is described by \citet{Ne99}. However, its line-counting process does not satisfy a moment duality with the type-frequency process. For the mutation-selection equation with pairwise interaction, i.e. $\Lambda\equiv 0$ and~$s$ linear, \citep{BCH2018interaction} establish a more elaborate duality to a process that keeps track of the \emph{entire} ancestral structure. The duality allows them to explain how the long-term behaviour of the forward process relates to the law of potential genealogies. But the dual process and the duality function are more complicated, and the techniques known from moment dualities can not be applied directly. Ideally, one would like to generalise the moment duality to a relation that holds for a large class of processes, but  still allows the application of the same techniques.

\smallskip
 
Despite the central role of frequency-dependent selection in ecology and evolution~\citep{ayala1974frequency}, a framework to treat models with general selection term is to the best of our knowledge still missing. The present article is a first step to fill this gap. We consider the SDE~\eqref{eq:SDEoriginal02} with~$d$ being a general polynomial vanishing at the boundary, i.e. $d(y)=y(1-y)s(y)$ for some polynomial $s$. We address the following questions.
\begin{enumerate}
	\item[(Q1)] Does the~$\Lambda$-Wright--Fisher process following the SDE~\eqref{eq:SDEoriginal02} admit a natural ancestral structure?
	\item[(Q2)] Is there a generalisation of the moment duality~\eqref{eq:duality-lambda}?
\end{enumerate}
If~$X$ could be approximated via a sequence of Moran models, then there is an intuitive answer to (Q1). Since in a Moran model, selection and neutral reproduction act at the level of individuals, we can embed the forward model and its ancestral structures into the same graphical representation. The limit of the ancestral structures, as population size tends to infinity, yields a natural ancestral structure for~$X$ and generalises the ASG to this setting. This leads to the following reformulation of~(Q1).

\begin{enumerate}
	\item [(Q1')] Is there a family of Moran models that converges to the solution of the SDE~\eqref{eq:SDEoriginal02}? 
\end{enumerate}
Indeed, we show that for any polynomial $d$ vanishing at the boundary, there are multiple ways to approximate~\eqref{eq:SDEoriginal02} via Moran models. One important (and puzzling) consequence is that different ASGs can be associated to the same forward process~$X$. This plurality of ASGs is addressed in more detail in (Q3) and (Q4) below.

\smallskip 

We now turn to (Q2). Consider the ASG of a sample of the population taken at time~$t_0$. Let~$L_u$ be the number of lines in the ASG at time~$u$ before~$t_0$. Assume that at time~$t$ before $t_0$ the frequency of type~$a$ is~$x$. It will become apparent that conditional on $(L_u:\,u\in[0,t])$, the probability that all individuals in the sample are type~$a$ is a polynomial of degree~$L_t$ in~$x$. Write~$V_t$ for the coefficient vector of this polynomial in the Bernstein basis of degree~$L_t$. We call $V\coloneqq(V_t:\, t\geq0)$ the \emph{Bernstein coefficient process}. It is an autonomous continuous-time Markov chain valued in $\cup_{n\in\N}\R^{n+1}$ with transition rates stated in Definition~\ref{def:bcp}. The duality relation~\eqref{eq:duality-lambda} generalises into what we call \emph{Bernstein duality}, that is 
\begin{equation}\label{eq:bernstein--0}
\E_x\big[X_t^n\big] \ = \ \E_{e_{n+1}}\left[\sum_{\ell=0}^{L_t} V_{t}(\ell)\binom{L_t}{\ell}x^\ell(1-x)^{L_t-\ell}\right],\qquad \forall x\in [0,1], \, \forall n\in \N
\end{equation}
where $V_0=e_{n+1}\coloneqq(0,\ldots,0,1)^T$, i.e. the $(n+1)$-st unit vector. The formulation of our duality raises the following question.
\begin{enumerate}
	\item [(Q3)] Is the Bernstein duality useful to analyse the behaviour of the process $X$? 
\end{enumerate}
We show that many methods that were successfully utilised in a moment duality extend to our setting. For example, the Bernstein duality allows us to relate the absorption (fixation/extinction) probabilities and the time to absorption of~$X$ to properties of~$V$. More precisely, we show that the fixation probabilities relate to the invariant measure of~$V$ (when it exists), and the time to fixation relates to the mean asymptotic behaviour of~$V$ as the number of lines initially present in the ASG tends to infinity. 

\smallskip

Different ASGs associated to the same model can behave significantly differently. For example, depending on the ASG, the corresponding line-counting process~$L$ can be either transient or recurrent. This makes it natural to ask:

\begin{enumerate}
	\item [(Q4)] Are some ASGs more favourable? Is there an optimal ASG? Is there a unique optimal ASG?
\end{enumerate}

In order to answer (Q4) we introduce the notion of \emph{minimal ASG}. Loosely speaking, minimal ASGs minimise the number of potential ancestors. By restricting to minimal ASGs, one recovers classical ancestral processes from the literature~\cite{NeKro97,KroNe97,Ne99}.
\smallskip

\bigskip

{\bf Outline.} The article is organised as follows. Section~\ref{s2} summarises the paper and contains all our main results. The proofs and more in-depth analyses are shifted to the subsequent sections. Section~\ref{s3} contains the proof of the convergence of appropriate Moran models to the~$\Lambda$-Wright--Fisher process. A detailed discussion of the ancestral process and the proofs of its properties can be found in Section~\ref{s4}. In particular, it contains the proof of the Bernstein duality. The process counting the potential ancestors is analysed in Section~\ref{s5}. Section~\ref{s6} is devoted to applications of the new processes and of the duality. In Section \ref{s7}, we treat the problem of minimality among potential genealogies in two ways. One approach seeks to avoid superfluous branches, and another one minimises the effective branching rate.

%%%%%%%%%%%%%%%%%%%%%%%%%%%%%%%%%%%%%%%%%%%%%%%%%%%%%%%%%%%%%%%%%%%%%%%%%%%%%%%%%%%%%%%%%%%%%%%%%%%%%%%%%%%%%%%%%%%%%%%%%%%%%%%%%%%%%%%%%%%%%%%%%%%%%%%%%%%%%
%%%%%%%%%%%%%%%%%%%%%%%%%%%%%%%%%%%%%%%%%%%%%%%%%%%%%%%%%%%%%%%%Section2%%%%%%%%%%%%%%%%%%%%%%%%%%%%%%%%%%%%%%%%%%%%%%%%%%%%%%%%%%%%%%%%%%%%%%%%%%%%%%%%%%%%%
%%%%%%%%%%%%%%%%%%%%%%%%%%%%%%%%%%%%%%%%%%%%%%%%%%%%%%%%%%%%%%%%%%%%%%%%%%%%%%%%%%%%%%%%%%%%%%%%%%%%%%%%%%%%%%%%%%%%%%%%%%%%%%%%%%%%%%%%%%%%%%%%%%%%%%%%%%%%%
\section{Summary of main results}\label{s2} 
\textbf{Notation.}
The positive and non-negative integers are denoted by $\N$ and $\N_0\coloneqq\N\cup\{0\}$, respectively. The non-negative real numbers are denoted by~$\R_+$. For $n\in \N$, we define \[[n]\coloneqq\{1,\ldots,n\},\quad [n]_0\coloneqq[n]\cup\{0\},\quad \text{and} \quad \noo{n}\coloneqq[n]\setminus\{1\}.\]  
Write $\xrightarrow[]{(d)}$ for convergence in distribution of random variables and $\xRightarrow[]{(d)}$ for convergence in distribution of c\`{a}dl\`{a}g process, where the space of c{\`a}dl{\`a}g functions is endowed with the Skorokhod topology. For any Borel set $S\subset\Rb$, $\Ms_f^*(S)$ is the set of non-zero finite measures on $S$. 
\smallskip

For $n,k,j\in \N_0$ with $n\geq k\vee j$, we write $K\sim \hypdist{n}{k}{j}$ if~$K$ is a hypergeometric random variable with parameters $n,k$, and $j$, i.e. \[\P(K=i)= \frac{\binom{k}{i} \binom{n-k}{j-i} }{\binom{n}{j}},\qquad i\in \{ 0\vee (j+k-n),\ldots, k\wedge j\}.\]
For $n\in \N_0$ and $x\in [0,1]$, we write $B\sim \bindist{n}{x}$ if~$B$ is a binomial random variable with parameters $n$ and $x$, i.e. $\P(B=i)= \binom{n}{i} x^i (1-x)^{n-i}$ for $i\in [n]_0$. 
\smallskip

For $n\in\Nb$ and $i\in[n]_0$, $b_{i,n}$ denotes the $i$-th polynomial in the Bernstein basis of the polynomials of degree at most $n$, i.e. $b_{i,n}(x) \coloneqq \binom{n}{i} x^{i}(1-x)^{n-i}$, $x\in\R$. In addition, we set
$B_{n}(x) \coloneqq \left(b_{i,n}(x)\right)_{i=0}^{n}$, $x\in\R$. The degree of a polynomial $f$ is denoted by $\deg(f)$. For any polynomial $f$ with $\deg(f)\leq n$, the $n$-Bernstein coefficient vector of $f$ is the unique vector $(v_{i})_{i=0}^n\in\Rb^{n+1}$ such that \[f(x)=\sum_{i=0}^n v_i b_{i,n}(x),\quad \textrm{for all $x\in \Rb$}.\] For $u\coloneqq(u_i)_{i\in[n]_0}, v\coloneqq(v_i)_{i\in[n]_0}\in\Rb^{n+1}$, the inner product between $u$ and $v$ is defined via $\langle u,v\rangle \coloneqq \sum_{i=0}^{n}u_i v_i$.

%%%%%%%%%%%%%%%%%%%%%%%%%%%%%%%%%%%%%%%%%%%%%%%%%%%%%%%%%%%%%%%%Subsection2.1%%%%%%%%%%%%%%%%%%%%%%%%%%%%%%%%%

\subsection{Moran models with frequency-dependent selection and large neutral offspring}\label{s2.1}
Consider the SDE~\eqref{eq:SDEoriginal02} with a polynomial $d$ vanishing at the boundary. From now on we fix $m\in \N\setminus\{1\}$ such that $\deg(d)\leq m$. Our first aim is to approximate the solution to the SDE by a sequence of Moran models. This answers~(Q1').

\smallskip

Our Moran model describes a haploid population of~$N$ individuals of type either~$a$ or~$A$. The population evolves in continuous-time and is subject to selection and neutral reproduction. The basic principle in a selective event is that an individual independently gathers a group of \textit{potential parents}. The total group size is at most $m$. The group composition determines a criterion according to which one of the potential parents is chosen to produce a single offspring that replaces the individual that initiated the selective event. We refer to such an event as an \emph{interaction}. In addition, each individual independently reproduces in a neutral way and its offspring replaces a fraction of the population. 

\smallskip
Now we spell out the details. For each $\ell\in \, ]m]$, each individual is independently selected at rate $\beta_\ell$ to initiate an \emph{$\ell$-interaction}. This means that $\ell-1$ individuals are chosen uniformly at random without replacement among the $N-1$ remaining ones. Together with the selected individual they form a group of~$\ell$ potential parents. Assume $j$ of the~$\ell$ potential parents are of type $a$. Then with probability $p_{j,\ell}$ (resp. $1-p_{j,\ell}$), a potential parent of type~$a$ (type~$A$) produces a single offspring that replaces the selected individual. The reproducing individual is chosen uniformly at random among the potential parents of type~$a$ (resp.~$A$). We do not allow for mutations and so it is natural to set $p_{0,\ell}=0$ and $p_{\ell,\ell}=1$. Thus, the selection mechanism is driven by a vector of rates $\beta=(\beta_{\ell})_{\ell\in]m]}\in\Rb_+^{m-1}$ and an array of probabilities  
\[p\coloneqq(p_{j,\ell})_{\ell\in \, ]m],j\in[\ell]_0}\in\Ps_m\coloneqq \prod_{\ell=2}^{m}\{0\}\times[0,1]^{\ell-1}\times\{1\}.\]
Note that to ease the notation, the index of $\beta$ starts at~$2$. Fig.~\ref{fig:selandneutral}~(left) illustrates a~$4$-interaction. In such a graphical representation, it will be convenient to think of types along the arising structures as colours. This is why in this context, we occasionally refer to types as colours, and it is the reason we call~$p\in\Ps_m$ a \emph{colouring rule}. The way we model selection is based on an idea already present in~\citep{gonzalezcasanova2018}.

\begin{remark}\label{rem:determiisticvoting(intro)}
	Typical choices for a colouring rule are
\vspace{.5em}	\begin{enumerate}
		\begin{minipage}{0.52\linewidth}
			\item $p_{i,\ell}=\1_{\{i\geq \ell/2 \}}$ - majority rule (e.g.~\citep{etheridge2017}),
			\item $p_{i,\ell}=\1_{\{i\leq\ell/2\}}+\1_{\{i=\ell\}}$ - minority rule, 
		\end{minipage}~~
		\begin{minipage}{0.47\linewidth}
			\item $p_{i,\ell}=\1_{\{i=\ell \}}$ -~$A$ always wins (e.g.~\citep{KroNe97}),
			\item $p_{i,\ell}=i/\ell$ - uniform rule,
		\end{minipage}
	\end{enumerate}
	where $\ell\geq 2$ and $i\in[\ell]_0$. If $p_{i,\ell}\in \{0,1\}$ for all $\ell\in\, ]m],\, i\in [\ell]_0$, the type of the descendent is a deterministic function of the types of the potential parents. We call such colouring rules~\emph{deterministic}. In particular, rules (1)--(3) are deterministic. 
\end{remark}
\begin{figure}[t!]
	\begin{minipage}{0.45\textwidth}
		\centering
		\scalebox{0.9}{\begin{tikzpicture}
			%horizontal line
			\draw[opacity=1, line width=.5mm, dotted] (0,0) -- (2,0);
			\draw[opacity=1, line width=.5mm] (0,.5) -- (2,0.5);
			\draw[opacity=1, line width=.5mm, dotted] (0,1) -- (2,1);
			\draw[opacity=1, line width=.5mm, dotted] (0,1.5) -- (2,1.5);
			\draw[opacity=1, line width=.5mm] (0,2) -- (2,2);
			\draw[opacity=1, line width=.5mm, dotted] (0,2.5) -- (1,2.5);
			\draw[opacity=1, line width=.5mm] (1,2.5) -- (2,2.5);
			
			% Large selection event
			\fill[white, opacity=1, draw=black, fill=black] (1.1,2.4) rectangle (0.9,2.6);
			\fill[white, opacity=1, draw=black] (1.1,2.1) rectangle (0.9,1.9);
			\fill[white, opacity=1, draw=black] (1.1,1.6) rectangle (0.9,1.4);
			%\fill[white, opacity=1, draw=black] (1.1,1.1) rectangle (0.9,0.9);
			\fill[white, opacity=1, draw=black] (1.1,-0.1) rectangle (0.9,0.1);
			%\fill[white, opacity=1, draw=black] (1.1,-0.1) rectangle (0.9,0.1);
			\end{tikzpicture}}
	\end{minipage} \begin{minipage}{0.45\textwidth}
		\centering
		\scalebox{0.9}{\begin{tikzpicture}
			%horizontal line
			\draw[opacity=1, line width=.5mm, dotted] (0,0) -- (2,0);
			\draw[opacity=1, line width=.5mm] (0,.5) -- (1,0.5);
			\draw[opacity=1, line width=.5mm, dotted] (2,.5) -- (1,0.5);
			\draw[opacity=1, line width=.5mm, dotted] (0,1) -- (2,1);
			\draw[opacity=1, line width=.5mm] (0,1.5) -- (1,1.5);
			\draw[opacity=1, line width=.5mm, dotted] (1,1.5) -- (2,1.5);
			\draw[opacity=1, line width=.5mm] (0,2) -- (2,2);
			%\draw[opacity=1, line width=.5mm, dotted] (1,2) -- (2,2);
			\draw[opacity=1, line width=.5mm, dotted] (0,2.5) -- (1,2.5);
			\draw[opacity=1, line width=.5mm, dotted] (1,2.5) -- (2,2.5);
			
			% Large selection event
			%\draw (1,2) circle (1.2mm)  [fill=gray];
			\draw (1,1.5) circle (1.2mm)  [fill=white!100];
			%\draw (1,1) circle ()1.2mm)  [fill=gray];
			\draw (1,0.5) circle (1.2mm)  [fill=white!100];
			\draw (1,0) circle (1.2mm)  [fill=white!100];
			\draw[opacity=1] (1,2.5) circle (1.2mm)  [fill=black];
			\end{tikzpicture}    }
	\end{minipage}
	\caption[Selective ]{Interaction (left) and neutral reproduction (right) in a Moran model forward in time. Solid (resp. dotted) lines correspond to type~$a$ (resp. type~$A$). Time runs form left to right. Left: the type $A$ individual marked by a black square initiates a $4$-interaction; the other potential parents are marked by white squares; a type~$a$ replaces the selected individual (which occurs with probability~$p_{1,4}$). Right: The individual with a black circle initiates a $3$-reproduction event; all individuals marked with a white circle are replaced by its offspring. The figure does not illustrate the genealogical structure of the Moran model.}
	\label{fig:selandneutral}
\end{figure}

Neutral reproduction is driven by a measure $\mu\in\Ms_f^*([N-1]_0)$. For each $r\in[N-1]$, each individual independently gives birth to $r$ individuals at rate $\mu(\{r\})$ if $r\neq 1$ and at rate $\mu(\{1\})+\mu(\{0\})/2$ if $r=1$. The offspring inherits the parent's type and replaces $r$ individuals that are chosen uniformly at random without replacement from the population before the reproduction event, excluding the parent. We call this an \emph{$r$-reproduction}. Fig.~\ref{fig:selandneutral}~(right) illustrates a~$3$-reproduction. By construction, $\ell$-interactions and $r$-reproductions keep the population size constant. We refer to the just-described model as the \emph{$(\beta,p,\mu)$-Moran model}.

\smallskip

The description of the large population limit of a family of Moran models requires some notation. Define the operator $T^N:\Ms_f^*([N-1]_0)\to\Ms_1([0,1])$ via
\[T^N\mu\coloneqq\frac{1}{M_\mu}\left(\delta_0 \mu(\{0\})+\sum_{k=1}^{N-1}\delta_{\frac{k}{N}}\,\mu(\{k\})k^2 \right),\]
where $M_\mu\coloneqq\mu(\{0\})+\sum_{k=1}^{N-1}\mu(\{k\})k^2$, $\delta_{y}$ is the Dirac mass at $y$, and $\Ms_1([0,1])$ is the set of probability measures on $[0,1]$. We consider a so-called weak selection framework, i.e. selection scales inversely with the population size. Furthermore, time is sped up by the population size.
\begin{theorem}[Large population limit]\label{teo:convergence}
	Let $m\in\Nb\setminus\{1\}$, $(\beta,p)\in \R_+^{m-1}\times \Ps_m$, and $\Lambda\in\Ms_f^*([0,1])$. Define $d_{\beta,p}:\R\to\R$ via  \begin{equation}
	d_{\beta,p}(x)\coloneqq\sum_{\ell=2}^m\beta_\ell\sum_{i=0}^\ell b_{i,\ell}(x) \Big(p_{i,\ell} -\frac{i}{\ell}\Big).\label{eq:driftbetap}
	\end{equation}
	For each~$N\in \N$ with $N\geq m$, let $\beta^{(N)}\in \R_+^{m-1}$ and $\mu_N\in\Ms_f^*([N-1]_0)$. Let $(X^{(N)}_t:\,t\geq 0)$ be the type-$a$ frequency process in a $(\beta^{(N)},p,\mu_N)$-Moran model of size $N$. Assume that
	\begin{enumerate}
		\item $N\beta^{(N)}\xrightarrow[N\to\infty]{}\beta,$
		\item $\mu_{N}(\{0\})\xrightarrow[N\to\infty]{}\Lambda(\{0\}),\quad M_{\mu_N}\xrightarrow[N\to\infty]{}\Lambda([0,1]),\quad\textrm{and}\quad T^N\mu_N\xrightarrow[N\to\infty]{(d)}\Lambda/\Lambda([0,1]).$
	\end{enumerate}
	 If in addition, $X_0^{(N)}\xrightarrow[N\to\infty]{(d)} x\in[0,1]$, then $\left(X^{(N)}_{Nt}:\,t\geq 0\right)\xRightarrow[N\to\infty]{(d)} X\coloneqq(X_t:\,t\geq 0)$, where $X$ is the pathwise unique strong solution of the SDE 
	\begin{equation}\begin{aligned}
	\dd X_t&=d_{\beta,p}(X_t) \dd t  + \sqrt{\Lambda(\{0\})X_t(1-X_t)}\dd W_t\\
	&\quad +\int\limits_{(0,1]\times [0,1]}r\Big(\1_{\{u\leq X_{t-}\}}(1-X_{t-})-\1_{\{u>X_{t-}\}} X_{t-} \Big)\tilde{N}(\dd t, \dd r,\dd u),\quad X_0=x,
	\end{aligned}\label{eq:SDEoriginalBer} 
	\end{equation}
	where $(W_t:\,t\geq 0)$ is a standard Brownian motion and $\tilde{N}(\dd t,\dd r,\dd u)$ is an independent compensated Poisson measure on $[0,\infty)\times(0,1]\times[0,1]$ with intensity $\dd t\times r^{-2}\Lambda(\dd r)\times \dd u$.
\end{theorem}
\begin{remark}
	Note that condition (2) in Theorem \ref{teo:convergence} is equivalent to  
	\[\mu_{N}(\{0\})\xrightarrow[N\to\infty]{}\Lambda(\{0\})\quad\textrm{and}\quad N^2\sum_{k=1}^{N-1} f\left(\frac{k}{N}\right)\mu_N(\{k\})\xrightarrow[N\to\infty]{} \int\limits_{(0,1]}f(r)\frac{\Lambda(\dd r)}{r^2}, \]
	for every $f\in\Cs([0,1])$ such that $x\in[0,1]\mapsto f(x)/x^2\in\Cs([0,1])$ (cf. \cite[Condition (4.6)]{EGT10}).
\end{remark}
Theorem~\ref{teo:convergence} provides conditions under which a given sequence of Moran models converges to the SDE~\eqref{eq:SDEoriginalBer} with $\Lambda\in \Ms_f^*([0,1])$ and drift~\eqref{eq:driftbetap} parametrised by $(\beta,p)\in \R_+^{m-1}\times \Ps_m$. The next result provides an explicit choice of parameters satisfying the conditions of Theorem \ref{teo:convergence}.
\begin{corollary}\label{Natsequence}
	Let $(\beta,p,\Lambda)\in \R_{+}^{m-1}\times\Ps_m\times\Ms_f^*([0,1])$. Define $\beta^{(N)}\in \R_+^{m-1}$ and $\mu_N\in\Ms_f^*([N-1]_0)$ via
	$$\beta^{(N)}:=\frac{\beta}{N}\quad\textrm{and}\quad\mu_N:=\frac{1}{N^2}\left(N(N-1)\Lambda(\{0\})\,\delta_0+\sum_{k=1}^{N-1}\binom{N}{k+1}\lambda_{N,k+1}^0\,\delta_k\right),$$
	where $\lambda_{n,k}^0:=\int_{(0,1]} r^{k-2} (1-r)^{n-k} \Lambda(\dd r)$, $n\geq k\geq 2$. Let $X^{(N)}:=(X^{(N)}_t:\,t\geq 0)$ be the type-$a$ frequency process in a $(\beta^{(N)},p,\mu_N)$-Moran model of size $N$. If $X_0^{(N)}\xrightarrow[N\to\infty]{(d)} x\in[0,1]$, then $$(X^{(N)}_{Nt}:\,t\geq 0)\xRightarrow[N\to\infty]{(d)} X:=(X_t:\,t\geq 0),$$
	where $X$ is the pathwise unique strong solution of \eqref{eq:SDEoriginalBer}.
\end{corollary}
Note that in a Moran model of size~$N$ in Corollary~\ref{Natsequence}, the total rate of an $r$-reproduction is the same as the rate at which $r+1$ lines merge into one in a $\Lambda$-coalescent of size~$N$ with time slowed down by $1/N$. Furthermore, the colouring rule is the same in the Moran model and in its large population limit. The proof of Theorem~\ref{teo:convergence} and Corollary~\ref{Natsequence} as well as some other complementary results are given in Section~\ref{sect:EUS}. 

\subsection{Selection decomposition} One way to approximate the solution of~\eqref{eq:SDEoriginalBer} by a sequence of Moran models is given by Corollary~\ref{Natsequence}. Moreover, the SDE~\eqref{eq:SDEoriginal02} belongs to the class of SDEs~\eqref{eq:SDEoriginalBer} whenever $d=d_{\beta,p}$ for some $(\beta,p)\in \R_+^{m-1}\times \Ps_m$. Hence, to answer (Q1'), it remains to identify this class of drifts. Note that for any selection mechanism $(\beta,p)\in \R_+^{m-1}\times \Ps_m$, $d_{\beta,p}$ is a polynomial with degree at most~$m$ vanishing at the boundary. This motivates the following definition.

\begin{definition}[Selection decomposition] \label{def:SelectionDecomposition}
A \emph{selection decomposition} of a polynomial~$d$ vanishing at the boundary with $\deg(d)\leq m$ is a pair $(\beta,p)\in \R_+^{m-1}\times \Ps_m$ such that $d_{\beta,p}=d$. We denote by $\ccS_d$ the set of selection decompositions of $d$. Similarly, we say that $(\beta,p)\in \R_+^{m-1}\times \Ps_m$ is a selection decomposition of a vector $v\in\Rb^{m+1}$ if the $m$-Bernstein coefficient vector of $d_{\beta,p}$ is $v$.

\end{definition}
Note that for any $(\beta,p)\in \R_+^{m-1}\times \Ps_m$, $d_{\beta,p}(0)=d_{\beta,p}(1)=0$. According to the following result, for any polynomial vanishing at~$0$ and~$1$, the set of selection decompositions is infinite.
\begin{theorem}\label{thm:infiniteed}
	For any polynomial $d$ such that $d(0)=d(1)=0$, the set $\ccS_d$ is infinite.
\end{theorem}
In particular, there are infinitely many (substantially different) ways to approximate the solution of~\eqref{eq:SDEoriginal02} via Moran models and this answers (Q1'). The proof for Theorem~\ref{thm:infiniteed} is provided in Section \ref{sect:mBernDeco}. An explicit description of the set of selection decompositions is provided in Corollary~\ref{coro:geometriccharacterisationced}.

\begin{remark}[No selection]
In our context no selection means $d\equiv 0$. Note that $(\beta,\hat{p})\in \R_+^{m-1}\times \Ps_m$ is a selection decomposition of~$0$ if $\beta=(0,\ldots,0)$ or $\hat{p}$ is the uniform rule, i.e. $\hat{p}_{i,\ell}=i/\ell$.
\end{remark}

%%%%%%%%%%%%%%%%%%%%%%%%%%%%%%%%%%%%%%%%%%%%%%%%%%%%%%%%%%%%%%%%Subsection2.2%%%%%%%%%%%%%%%%%%%%%%%%%%%%%%%%%

\subsection{Ancestry in finite populations}
In this section we describe a natural ancestral process for the Moran model introduced in Section \ref{s2.1}. In particular, this answers~(Q1) at the level of finite populations. 

\smallskip 
\begin{figure}[b!]
	\centering
	\scalebox{0.7}{\begin{tikzpicture}
		%Frame
		\draw[densely dotted] (0,-0.5) --(0,4.5);
		\draw[densely dotted] (9.5,-0.5) --(9.5,4.5);    
		\node [right] at (0,-0.5) {$-t$};
		\node [right] at (0,4.5) {$\, \, t$};
		\node [right] at (9.5,4.5) {$0$};
		\node [right] at (9.5,-0.5) {$0$};
		\draw[-{angle 60[scale=5]}] (2.25,-.6) -- (7.25,-.6) node[text=black, pos=.5, yshift=6pt]{};
		\draw[-{angle 60[scale=5]}] (7.25,4.6) -- (2.25,4.6) node[text=black, pos=.5, yshift=6pt]{};

		%horizontal lines
		\draw[opacity=0.4, line width=.1mm] (0,0) -- (9.5,0);
		\draw[opacity=0.4, line width=.1mm] (0,0.5) -- (9.5,0.5);
		\draw[opacity=0.4, line width=.1mm] (0,1) -- (9.5,1);
		\draw[opacity=0.4, line width=.1mm] (0,1.5) -- (9.5,1.5);
		\draw[opacity=0.4, line width=.1mm] (0,2) -- (9.5,2);
		\draw[opacity=0.4, line width=.1mm] (0,2.5) -- (9.5,2.5);
		\draw[opacity=0.4, line width=.1mm] (0,3) -- (9.5,3);
		\draw[opacity=0.4, line width=.1mm] (0,3.5) -- (9.5,3.5);
		\draw[opacity=0.4, line width=.1mm] (0,4) -- (9.5,4);

		% asg Horizontal
		\draw[opacity=1, line width=.5mm] (2,3.5) -- (9.5,3.5);
		
		\draw[opacity=1, line width=.5mm] (2,3) -- (8,3);
		
		\draw[opacity=1, line width=.5mm] (0,4) -- (6,4);
		\draw[opacity=1, line width=.5mm] (2,2.5) -- (6,2.5);
		\draw[opacity=1, line width=.5mm] (0,2) -- (6,2);
		\draw[opacity=1, line width=.5mm] (4,1) -- (6,1);
		\draw[opacity=1, line width=.5mm] (0,0) -- (6,0);
		
		\draw[opacity=1, line width=.5mm] (0,1.5) -- (2,1.5);

		% asg vertical
		\draw[opacity=1, line width=.5mm] (8,3) -- (8,3.5);
		
		\draw[opacity=1, line width=.5mm] (6,3.597) -- (6,4);
		\draw[line width=.5mm ] (6,3.375) arc(-100:100:.125) ;
		\draw[opacity=1, line width=.5mm] (6,3.4) -- (6,1.597);
		\draw[line width=.5mm ] (6,1.375) arc(-100:100:.125) ;
		\draw[opacity=1, line width=.5mm] (6,1.4) -- (6,0.597);
		\draw[line width=.5mm ] (6,0.375) arc(-100:100:.125) ;
		\draw[opacity=1, line width=.5mm] (6,0.4) -- (6,0);
		
		\draw[opacity=1, line width=.5mm] (4,2.5) -- (4,2.097);
		\draw[line width=.5mm ] (4,1.875) arc(-100:100:.125) ;
		\draw[opacity=1, line width=.5mm] (4,1.9) -- (4,1.597);
		\draw[line width=.5mm ] (4,1.375) arc(-100:100:.125) ;
		\draw[opacity=1, line width=.5mm] (4,1.4) -- (4,1);
		
		\draw[opacity=1, line width=.5mm] (2,3.5) -- (2,2.097);
		\draw[line width=.5mm ] (2,1.875) arc(-100:100:.125) ;
		\draw[opacity=1, line width=.5mm] (2,1.9) -- (2,1.5);
		%		\draw[line width=.5mm ] (2,0.875) arc(-100:100:.125) ;
		%		\draw[opacity=1, line width=.5mm] (2,0.9) -- (2,0.5);

		% Classical selection
		\fill[white, opacity=1, draw=black] (8.1,2.9) rectangle (7.9,3.1);
		\fill[white, opacity=1, draw=black, fill=black] (8.1,3.4) rectangle (7.9,3.6);
		%	\draw (8,3.5)[gray] circle (0.5mm)  [fill=gray];

		% Large selection event
		\fill[white, opacity=1, draw=black] (6.1,4.1) rectangle (5.9,3.9);
		\fill[white, opacity=1, draw=black, fill=black] (6.1,3.1) rectangle (5.9,2.9);
		\fill[white, opacity=1, draw=black] (6.1,2.6) rectangle (5.9,2.4);
		\fill[white, opacity=1, draw=black] (6.1,2.1) rectangle (5.9,1.9);
		\fill[white, opacity=1, draw=black] (6.1,1.1) rectangle (5.9,0.9);
		\fill[white, opacity=1, draw=black] (6.1,-0.1) rectangle (5.9,0.1);
		%\draw (6,3)[gray] circle (0.5mm)  [fill=gray];
		
		% Large neutral offspring
		\draw (2,3) circle (1.2mm)  [fill=white!100];
		\draw (2,2.5) circle (1.2mm)  [fill=white!100];
		\draw (2,0.5) circle (1.2mm)  [fill=white!100];
		\draw (2,3.5) circle (1.2mm)  [fill=white!100];
		\draw[opacity=1] (2,1.5) circle (1.2mm)  [fill=black!100];
		
		% Classical coalsecence
		\draw (4,1) circle (1.2mm)  [fill=white];
		\draw[opacity=1] (4,2.5) circle (1.2mm)  [fill=black!100];

		% Coalsecence (not in ASG)
		\draw (7,2) circle (1.2mm)  [fill=white];
		\draw (7,1) circle (1.2mm)  [fill=black];
		\draw[opacity=1] (7,2.5) circle (1.2mm)  [fill=white];
		
		% Large selection event
		\fill[white, opacity=1, draw=black] (9.1,4.1) rectangle (8.9,3.9);
		\fill[white, opacity=1, draw=black] (9.1,1.6) rectangle (8.9,1.4);
		\fill[white, opacity=1, draw=black, fill=black] (9.1,-0.1) rectangle (8.9,0.1);
		
		\end{tikzpicture}    }
	\caption[Graphical representation of the Moran interacting particle system and its embedded ASG]{A realisation of the Moran interacting particle system (thin lines) for a population of size~$N=9$ and the embedded Moran-ASG (bold lines) for a sample of size~$1$. Time runs forward in the Moran model ($\rightarrow$) and backward in the ASG ($\leftarrow$). Backward time~$t$ corresponds to forward time~$-t$. Circles represent (neutral) reproduction events. Squares represent (selective) replacement events. Backward in time the lineages involved in an $r$-reproduction event merge into a single lineage (black circle). In contrast, in an $\ell$-interaction the single lineage (black square) branches into~$\ell$ lineages. }
	\label{fig:representASG}
\end{figure}

The Moran model admits a natural graphical representation as an interactive particle system, see Fig.~\ref{fig:representASG}. Here, individuals are represented by pieces of horizontal lines. Forward time runs from left to right. Squares indicate which lines are involved in an $\ell$-interaction: a black square marks the individual that initiated the selection, and white squares mark the other potential parents. Circles indicate which lines are involved in an $r$-reproduction: a black circle marks the individual that reproduces and white circles mark the individuals that are replaced by its offspring. These graphical elements arise in the picture according to the arrival times of independent Poisson processes (the rates can be worked out from the definition of the Moran model).
 
\smallskip 

So far this procedure provides a construction of an untyped particle system. Assume that we have constructed the particle picture in $[-t,0]$. Given an initial type configuration (that is, a type assigned to each line at forward time~$-t$), types propagate forward in time along the untyped particle system according to the (random) colouring procedure described in Section~\ref{s2.1}. 

\smallskip 

In this setting genealogical structures are extracted from the particle picture as follows. Start with a sample of~$n$ individuals chosen at time $t=0$ and trace back the set of their potential ancestors by reading the graphical picture from right to left. Suppose there are currently~$n$ potential ancestors in the graph. Neutral reproductions have the following effect backward in time. If $k$ lineages of potential ancestors simultaneously encounter circles, they merge into one and take the place of the individual marked with a black circle. In particular, the number of potential ancestors decreases to $n-k+1$. The effect of an $\ell$-interaction backward in time is as follows. If a lineage encounters a black square, we add all lines marked with a white square (in this case~$\ell-1$) to the set of potential ancestors. In particular, if~$\hat{\ell}\leq \ell-1$ white squares are currently outside the set of potential ancestors, the number increases to~$n+\hat{\ell}$.

\smallskip 

The structure arising up to backward time $t$ under this procedure endowed with the colouring rule is a generalisation of the \emph{ancestral selection graph}~(ASG) of~\citet{KroNe97}, see Fig.~\ref{fig:representASG}. We refer to it as the \emph{Moran-ASG}. Note that its distribution depends on the parameters $(\beta,p,\mu)$ of the underlying Moran model and the size of the sample~$n$. To determine the types of the individuals in the sample at time $0$, we assign types to the lines in the ASG at backward time $t$ in an exchangeable manner according to the initial type distribution, and then propagate them forward along the lines up to time~$0$ while adhering to the colouring rule~$p$, see Fig.~\ref{fig:ASGcoloured}. 

\begin{figure}[t]
	\centering
	\scalebox{0.7}{\begin{tikzpicture}
		%Frame
		\draw[densely dotted] (0,-0.5) --(0,4.5);
		\draw[densely dotted] (9.5,-0.5) --(9.5,4.5);    
		\node [right] at (0,4.5) {$t$};
		\node [right] at (9.5,4.5) {$0$};

		% asg Horizontal
		\draw[opacity=1, line width=.5mm] (2,3.5) -- (8,3.5);
		\draw[opacity=1, line width=.5mm] (8,3.5) -- (9.5,3.5);
		
		\draw[opacity=1, line width=.5mm] (2,3) -- (6,3);
		\draw[opacity=1, line width=.5mm, dotted] (6,3) -- (8,3);
		
		\draw[opacity=1, line width=.5mm] (0,4) -- (6,4);
		\draw[opacity=1, line width=.5mm] (2,2.5) -- (6,2.5);
		\draw[opacity=1, line width=.5mm, dotted] (0,2) -- (6,2);
		\draw[opacity=1, line width=.5mm] (4,1) -- (6,1);
		\draw[opacity=1, line width=.5mm, dotted] (0,0) -- (6,0);
		
		\draw[opacity=1, line width=.5mm] (0,1.5) -- (2,1.5);
		
		%		%asg coloring of the leaves
		
		% asg vertical
		\draw[opacity=1, line width=.5mm] (8,3) -- (8,3.5);
		
		\draw[opacity=1, line width=.5mm] (6,3.597) -- (6,4);
		\draw[line width=.5mm ] (6,3.375) arc(-100:100:.125) ;
		\draw[opacity=1, line width=.5mm] (6,3.4) -- (6,0.0);
		
		\draw[opacity=1, line width=.5mm] (4,2.5) -- (4,2.097);
		\draw[line width=.5mm ] (4,1.875) arc(-100:100:.125) ;
		\draw[opacity=1, line width=.5mm] (4,1.9) -- (4,1);
		
		\draw[opacity=1, line width=.5mm] (2,3.5) -- (2,2.097);
		\draw[line width=.5mm ] (2,1.875) arc(-100:100:.125) ;
		\draw[opacity=1, line width=.5mm] (2,1.9) -- (2,1.5);

		% Classical selection
		\fill[white, opacity=1, draw=black] (8.1,2.9) rectangle (7.9,3.1);
		\fill[white, opacity=1, draw=black, fill=black] (8.1,3.4) rectangle (7.9,3.6);

		% Large selection event
		\fill[white, opacity=1, draw=black] (6.1,4.1) rectangle (5.9,3.9);
		\fill[white, opacity=1, draw=black, fill=black] (6.1,3.1) rectangle (5.9,2.9);
		\fill[white, opacity=1, draw=black] (6.1,2.6) rectangle (5.9,2.4);
		\fill[white, opacity=1, draw=black] (6.1,2.1) rectangle (5.9,1.9);
		\fill[white, opacity=1, draw=black] (6.1,1.1) rectangle (5.9,0.9);
		\fill[white, opacity=1, draw=black] (6.1,-0.1) rectangle (5.9,0.1);
		%\draw (6,3)[gray] circle (0.5mm)  [fill=gray];
		
		% Large neutral offspring
		\draw (2,3) circle (1.2mm)  [fill=white!100];
		\draw (2,2.5) circle (1.2mm)  [fill=white!100];
		%\draw (2,0.5) circle (1.2mm)  [fill=white!100];
		\draw (2,3.5) circle (1.2mm)  [fill=white!100];
		\draw[opacity=1] (2,1.5) circle (1.2mm)  [fill=black!100];
		
		% Classical coalsecence
		\draw (4,1) circle (1.2mm)  [fill=white];
		\draw[opacity=1] (4,2.5) circle (1.2mm)  [fill=black!100];
		
		\end{tikzpicture}    }
	\caption{Colouring of the ASG of Fig \ref{fig:representASG} under the minority rule, i.e. $p_{i,\ell}=\1_{\{i\leq \ell/2  \}}$ for $i\in [\ell-1]$. This means that at every interactive event, if not all potential parents (black and white squares) have the same colour, the colour of the replaced individual (black square) is determined by the minority colour among the potential parents, type $a$ wins in case of equality. At every reproduction event, the offspring (white circles) inherit the colour of the parent (black circle). Solid (resp. dotted) lines are type $a$ (resp. $A$). }
	\label{fig:ASGcoloured} 
\end{figure}
\subsection{Ancestral selection graph}\label{s2.2}
Under the conditions stated in Theorem~\ref{teo:convergence}, a sequence of Moran models converges to a $\Lambda$-Wright--Fisher process as~$N\to\infty$ and with time appropriately rescaled. Considering the corresponding limit for the Moran-ASG leads to a natural candidate for an ancestral process of the $\Lambda$-Wright--Fisher process. More precisely, let~$d$ be a polynomial with $d(0)=d(1)=0$. Fix~$(\beta,p)\in\ccS_d$ and $\Lambda\in\Ms_f^*([0,1])$. Consider the Moran-ASGs corresponding to the sequence of~$(\beta^{(N)},p,\mu_N)$-Moran models defined in Corollary~\ref{Natsequence}, time is sped up by~$N$. First, note that the probability that an interactive event involves more than one line already present in the ASG is small (it vanishes as $N\to\infty$). Moreover, for sufficiently large~$N$, (1) each line branches into~$\ell$ lines at a rate which is close to~$\beta_\ell$, and (2) a given group of~$k$ lines merges into one at a rate close to 
\[\lambda_{n,k}\coloneqq\int_{[0,1]} r^{k-2} (1-r)^{n-k} \Lambda(\dd r)=\Lambda(\{0\})\1_{\{k=2\}}+\lambda_{n,k}^0,\quad n\geq k\geq 2.\]
These observations motivate the following definition. 
\begin{definition}[ASG]\label{def:asg} Fix $(\beta,p)\in \R_+^{m-1}\times \Ps_m$, $\Lambda\in\Ms_f^*([0,1])$. The $(\beta,p,\Lambda)$-ASG in $[0,t]$ starting from a sample of size $n$ is the pair $(\Gs_t,p)$, where $\Gs_t$ is the branching-coalescing particle system in~$[0,t]$ that starts with $n\in\Nb$ particles at time $0$ and that evolves as follows. 
	\begin{itemize}
		\item For $\ell\in\,]m]$, each particle branches at rate $\beta_\ell$ into $\ell$ particles.
		\item If the current number of particles is $j\geq2$, then for $k\in\,]j]$, every $k$-tuple of particles coalesce into a single particle at rate $\lambda_{j,k}$.
	\end{itemize}
We refer to the particles in~$\Gs_t$ that are present at time~$0$ as \emph{roots} and to the particles at time~$t$ as \emph{leaves}. We write $L_t$ for the number of leaves of~$\Gs_t$ and call $L:=(L_t:\, t\geq 0)$ the \emph{leaf process}.
\end{definition}
A more precise construction of the branching-coalescing system is provided in Definition~\ref{ASGprecise}. We write~$\Gs_t^n$ instead of $\Gs_t$ (resp. $L_t^n$ instead of $L_t$) whenever we want to stress the number of roots~$n$. Processes reminiscent to~$L$ also go by the name block-counting process of the ASG~\citep{gonzalezcasanova2018}, line-counting process of the ASG~\citep{LKBW15}, or process of potential ancestral lineages~\citep[Sect. 5.4]{etheridge2011some}.

\smallskip

The leaf process $L$ is a continuous-time Markov chain on~$\N$, and it has the following transition rates.
\begin{itemize}
	\item For $n\in \Nb$ and every $\ell\in \,]m]$, $n\to n+\ell-1$ at rate $n\beta_\ell$.
	\item For $n\in \Nb\setminus\{1\}$ and every $k\in\,]n]$, $n\to n-k+1$ at rate $\binom{n}{k}\lambda_{n,k}.$
\end{itemize}

For the class of drifts considered in~\citep{gonzalezcasanova2018}, the $\Lambda$-Wright--Fisher processes is in moment duality with the line-counting process of an ASG. In this setting the line-counting process plays a crucial role in the analysis of the $\Lambda$-Wright--Fisher process. In our framework this quantity will also be important.

\begin{remark}
	Consider a population evolving according to the SDE~\eqref{eq:SDEoriginal02} with polynomial drift $d$ vanishing at the boundary and $\Lambda\in \Ms_f^*([0,1])$. Then, for any $(\beta,p)\in \ccS_d$, the $(\beta,p,\Lambda)$-ASG with $n$ roots provides a natural potential genealogy for a sample of size~$n$ from the population. Since $\ccS_d$ is infinite, infinitely many different potential genealogies are associated to the same forward model.
\end{remark}

\subsection{Ancestral selection polynomial}\label{sec:asp}

The ASG introduced in the previous section is a rather cumbersome object. We now introduce a more tractable process. Recall that the types at the roots of an ASG represent the types in a sample from the population. They are a (random) function of the ASG and the types at the leaves. More precisely, the types in the sample are determined after propagating the types at the leaves along the lines of the ASG according to the colouring rule. Instead of keeping track of the ASG, we keep track of the conditional probability of a type composition in the sample, given partial observation of the ancestry. This motivates the following definition.

\begin{definition}[Ancestral selection polynomial] \label{def:asp} The $(\beta,p,\Lambda)$-\emph{ancestral selection polynomial (ASP)} is the random function $x\in[0,1]\mapsto P_t(x)$, where $P_t(x)$ is the conditional probability given $(L_{u}: u\in[0,t])$, that all roots of $\Gs_t$ are of type~$a$ if each leaf of $\Gs_t$ is of type~$a$ with probability $x$ (resp. of type~$A$ with probability $1-x$). It is assumed that the initial type assignment is independent for each leaf and $\cG_t$ is typed from the leaves to the roots using the colouring rule $p$. 
\end{definition}
For a more precise formulation, we refer to Definition~\ref{def-asp}.
The name ancestral selection polynomial stems from the fact that $P_t$ can be written as
$$P_t(x)=\sum_{i=0}^{L_t} V_t(i) \,b_{i,L_t}(x)=\langle B_{L_t}(x),V_t\rangle,\qquad x\in[0,1],$$
where $V_t(i)$ is the conditional probability given $(L_{u}: u\in[0,t])$, that all the roots of $\Gs_t$ are of type~$a$ if $i$ (resp. $L_t-i$) leaves are of type~$a$ (resp. $A$). Thus, $P_t$ is indeed a random polynomial, its degree is at most~$L_t$, and $V_t\coloneqq (V_t(i))_{i\in[L_t]_0}$ is its (random) $L_t$-Bernstein coefficient vector. 

\smallskip

It remains to see how a transition of~$L$ affects the ASP. First assume that $L_{t-}=n$ and $L_{t}=n+\ell-1$, i.e.~a branching event of size~$\ell$ occurs at time~$t$. Forward in time this means $\ell$ lines merge into one, see also Fig.~\ref{fig:graphs} (left). If~$i$ of the~$L_{t}$ leaves are of type~$a$ and $K_i$ denotes the number of lines with type~$a$ among the lines merging at time~$t$, then $K_i\sim \hypdist{L_{t}}{i}{\ell}$. In addition, given $K_i$, the offspring in the interactive event has type~$a$ (resp. $A$) with probability $p_{K_i,\ell}$ (resp. $1-p_{K_i,\ell}$), so at time~$t-$ there are $i-{K_i}+1$ (resp. $i-{K_i}$) lines of type~$a$. 
\smallskip

Next, assume that $L_{t-}=n$ and that $L_{t}=n-k+1$, i.e. $k$ lines coalesce at time~$t$. Forward in time this means that one leaf produces~$k-1$ lines at time~$t$, see also Fig.~\ref{fig:graphs} (right). If there are~$i$ leaves of type~$a$ at time~$t$, the leaf producing offspring has type~$a$ (resp.~$A$) with probability $i/(n-k+1)$ (resp. $1-i/(n-k+1)$), the number of lines having type~$a$ at time~$t-$ is then $i+k-1$ (resp.~$i$). 
\smallskip
\begin{figure}[t!]
	\begin{minipage}{.4\textwidth}
		\centering
		\scalebox{0.7}{\begin{tikzpicture}
			\draw[dotted] (-.5,-0.3) -- (-.5,4.7);
			\node[above] at (-0.5,4.7) {$t$};
			\node[right] at (-.5,-0.2) {$t-$};
			% asg Horizontal
			\draw[opacity=1, line width=.5mm] (2,3.5) -- (5.5,3.5);
			\draw[opacity=1, line width=.5mm] (2,3) -- (5.5,3);
			\draw[opacity=1, line width=.5mm] (-1,4) -- (4,4);
			\draw[opacity=1, line width=.5mm] (2,2.5) -- (4,2.5);
			\draw[opacity=1, line width=.5mm] (-1,2) -- (4,2);
			\draw[opacity=1, line width=.5mm] (-1,1.5) -- (2,1.5);
			\draw[opacity=1, line width=.5mm] (-1,1) -- (5.5,1);
			\draw[opacity=1, line width=.5mm] (-1,.5) -- (-.5,0.5);
			\draw[opacity=1, line width=.5mm] (-1,0) -- (-.5,0);
			
			% asg vertical
			\draw[opacity=1, line width=.5mm] (4,3.597) -- (4,4);
			\draw[line width=.5mm ] (4,3.375) arc(-100:100:.125) ;
			\draw[opacity=1, line width=.5mm] (4,3.4) -- (4,2);
			\draw[opacity=1, line width=.5mm] (-1,4.5) -- (-.5,4.5);
			\draw[opacity=1, line width=.5mm] (2,3.5) -- (2,2.097);
			\draw[line width=.5mm ] (2,1.875) arc(-100:100:.125) ;
			\draw[opacity=1, line width=.5mm] (2,1.9) -- (2,1.5);
			
			%Last selection event
			\fill[black, opacity=1, draw=black] (4.1,3.1) rectangle (3.9,2.9);
			
			% Large neutral offspring
			\draw[opacity=1] (2,1.5) circle (1.2mm)  [fill=black];
			
			\draw[opacity=1, line width=.5mm] (-.5,3.90) -- (-.5,2.1);
			\draw[line width=.5mm ] (-.5,3.875) arc(-100:100:.125) ;
			\draw[line width=.5mm ] (-.5,1.875) arc(-100:100:.125) ;
			\draw[opacity=1, line width=.5mm] (-.5,1.9) -- (-.5,1.6);
			\draw[line width=.5mm ] (-.5,1.375) arc(-100:100:.125) ;
			\draw[opacity=1, line width=.5mm] (-.5,4.097) -- (-.5,4.5);
			\draw[opacity=1, line width=.5mm] (-.5,1.4) -- (-.5,0);
			
			\fill[black, opacity=1, draw=black] (-.6,.9) rectangle (-.4,1.1);
			
			\end{tikzpicture}}
	\end{minipage}\hspace{1cm}
	\begin{minipage}{.4\textwidth}
		\centering
		\scalebox{0.7}{\begin{tikzpicture}
			
			\draw[dotted] (-.5,-0.3) -- (-.5,4.7);
\node[above] at (-0.5,4.7) {$t$};
\node[right] at (-.5,-0.2) {$t-$};
			% asg Horizontal
			\draw[opacity=1, line width=.5mm] (2,3.5) -- (5.5,3.5);
			\draw[opacity=1, line width=.5mm] (2,3) -- (5.5,3);
			\draw[opacity=1, line width=.5mm] (-.5,4) -- (4,4);
			\draw[opacity=1, line width=.5mm] (2,2.5) -- (4,2.5);
			\draw[opacity=1, line width=.5mm] (-1,2) -- (4,2);
			\draw[opacity=1, line width=.5mm] (-.5,1.5) -- (2,1.5);
			\draw[opacity=1, line width=.5mm] (-.5,1) -- (5.5,1);
			
			% asg vertical
			\draw[opacity=1, line width=.5mm] (4,3.597) -- (4,4);
			\draw[line width=.5mm ] (4,3.375) arc(-100:100:.125) ;
			\draw[opacity=1, line width=.5mm] (4,3.4) -- (4,2);
			\draw[opacity=1, line width=.5mm] (2,3.5) -- (2,2.097);
			\draw[line width=.5mm ] (2,1.875) arc(-100:100:.125) ;
			\draw[opacity=1, line width=.5mm] (2,1.9) -- (2,1.5);
			
			%Last selection event
			\fill[black, opacity=1, draw=black] (4.1,3.1) rectangle (3.9,2.9);
			
			% Large neutral offspring
			\draw[opacity=1] (2,1.5) circle (1.2mm)  [fill=black];
			\draw[opacity=1, line width=.5mm] (-1,.5) -- (-.5,.5);

			\draw[opacity=1, line width=.5mm] (-.5,4) -- (-.5,2.1);
			\draw[line width=.5mm ] (-.5,1.875) arc(-100:100:.125) ;
			\draw[opacity=1, line width=.5mm] (-.5,1.9) -- (-.5,.5);
			
			\draw (-.5,.5) circle (1.2mm)  [fill=black];
			\end{tikzpicture}    }
	\end{minipage}
	\caption{Left: At an interactive event, a branching point is grafted at a leaf chosen uniformly at random at time~$t-$. In this example it is a $4$-interaction, $L_{t}=7$, and $L_{t-}=4$. Right: In a coalescence event, a subset of leaves at time~$t-$ collapses into a single line. Here, $3$ leaves merge, $L_{t}=2$, and $L_{t-}=4$.} 
	\label{fig:graphs}
\end{figure}

The following definition captures the just-described effect of branchings and coalescences to the Bernstein coefficient vector of the ASP. 
\begin{definition}[Selection and coalescence matrices]\label{def:SelCoagMatrices} Fix $m\in \N\setminus\{1\}$ and $p\in \Ps_m$. Define the following linear operators. 
	\begin{enumerate}
		\item For every $\ell\in \, ]m],\, n\in \N$, let $S^{n,\ell}: \R^{n+1}\to \R^{n+\ell}$ with
		\begin{equation*} 
			S^{n,\ell} v \ \coloneqq \  \left(\E\big[ \ p_{K_i,\ell} \ v_{i+1-K_i} \ + \ (1- p_{K_i,\ell}) \ v_{i-K_i} \big] \right)_{i=0}^{n+\ell-1},
		\end{equation*}
		where $K_i\sim\hypdist{n+\ell-1}{i}{\ell}$ and $v=(v_{i})_{i=0}^n\in \R^{n+1}$.
		\item For $n\in \N$ with $n\geq 2$ and $k\in \, ]n]$, let $C^{n,k}:\R^{n+1}\to \R^{n-k+2}$ with
		\begin{equation*}
			C^{n,k} v \ \coloneqq \  \left(\frac{i}{n-k+1}v_{i+k-1} +\left(1- \frac{i}{n-k+1}\right)v_{i} \right)_{i=0}^{n-k+1},
		\end{equation*}
	  where $v=(v_{i})_{i=0}^n\in \R^{n+1}$.
	\end{enumerate}
\end{definition}
The operators in Definition~\ref{def:SelCoagMatrices} define the transitions of a Markov process that codes the evolution of the Bernstein coefficient vector of the ASP.

\begin{definition}[Bernstein coefficient process]\label{def:bcp}
	The \emph{Bernstein coefficient process} is the Markov process $V\coloneqq (V_t:\, t\geq 0)$ on $\Rb^\infty\coloneqq\cup_{n\in\N} \R^{n+1}$ with the following transition rates.
	\begin{enumerate}
		\item For $v\in\Rb^{n+1}$ and for every $\ell\in \, ]m]$, 
		\begin{align*}
		v\to S^{n,\ell}v&\quad \text{at rate}\quad n\beta_\ell.\qquad \ 
		\end{align*}
		\item For $v\in\Rb^{n+1}$ and for every $k\in \, ]n]$,
		\begin{align*}
		v\to C^{n,k}v&\quad \text{at rate}\quad \binom{n}{k}\lambda_{n,k}.
		\end{align*}
	\end{enumerate}
Furthermore, we set $L_t\coloneqq \dim(V_t)-1$.
\end{definition}

The leaf process~$L$ from Definition~\ref{def:asg} is equal in law to $(\dim(V_t)-1:\, t\geq 0)$. This legitimises the abuse of notation in the use of~$L$ for both processes. The next result formalises the connection between the ASP and the Bernstein coefficient process.

\begin{proposition}\label{ASPvsVCP}
	Let $n\in \N$. Consider the ASP~$(P_t(x):\,t\geq 0)$ with initial condition $x\mapsto x^n$ and the Bernstein coefficient process $(V_t:\, t\geq 0)$ with initial condition $V_0 =e_{n+1}$, where $e_{n+1}$ is the $(n+1)$-st unit vector.
	For all $x\in[0,1]$, \[(P_t(x):\, t\geq 0)\ \overset{(d)}{=}\ \left(\langle B_{L_t}(x), V_t\rangle :\,t\geq0\right).\]
\end{proposition}
%%%%%%%%%%%%%%%%%%%%%%%%%%%%%%%%%%%%%%%%%%%%%%%%%%%%%%%%%%%%%%%%Subsection2.3%%%%%%%%%%%%%%%%%%%%%%%%%%%%%%%%%
\subsection{Bernstein duality}\label{s2.3}
In this section we generalise the notion of a moment duality to what we call \emph{Bernstein duality}, which addresses question (Q2). 
We start explaining the main idea behind this type of duality.

\smallskip

Let $X$ be the solution of the SDE~\eqref{eq:SDEoriginal02} with $X_0=x$, where~$d$ is polynomial vanishing at the boundary and $\Lambda\in\Ms_f^*([0,1])$. Consider the evolution of the underlying population model up to (forward) time~$t$, and there independently sample~$n$ individuals. The conditional probability given~$X_t$, that they are all of type~$a$ is~$X_{t}^n$. Now, to approach the problem from a backward perspective, consider~$(\beta,p)\in \ccS_d$ and run the corresponding branching-coalescing system~$\Gs$ starting with the~$n$ sampled individuals up to (forward) time $0$. Assign type~$a$ (resp.~$A$) independently to the leaves of $\Gs_t$ with probability $x$ (resp. $1-x$). The conditional probability given $(L_u:u\in[0,t])$, that the $n$ sampled individuals have type~$a$ is by definition $P_{t}(x)$. After averaging over all possible observations, this intuitive argument suggests that $\E_x[X_t^n]=\E_{x^n}[P_t(x)].$ The next theorem makes this heuristic precise.

\begin{theorem}[Bernstein duality]\label{thm:bernstein-duality}
	The processes $(X_t:\,t\geq0)$ and $(V_t:\,t\geq0)$ are dual with respect to the duality function $(x,v)\mapsto \langle B_{\dim(v)-1}(x),v\rangle$, i.e. for all $t>0$,
	\begin{equation}\label{eq:generalduality-selection}  
	\E_x\left[\langle B_{n}(X_t), v\rangle \right] \ = \ \E_{v}\left[ \langle B_{L_t}(x), V_t\rangle \right],\qquad \forall x\in[0,1],\, \forall n\in \N,\, \forall v\in \R^{n+1}.
	\end{equation}
\end{theorem}
The following corollary clarifies in which sense the Bernstein duality is a generalisation of the moment duality.

\begin{corollary}\label{coro:bernstein-duality}
	Consider the Bernstein coefficient process $V$ started at $V_0=e_{n+1}$. 
	\begin{enumerate}
		\item For all $x\in[0,1]$ and $t\geq 0$\begin{equation}\label{eq:duality-selection}    \ \  \E_x\left[X_t^n \right] \ = \ \E_{e_{n+1}}\left[ \langle B_{L_t}(x), V_t\rangle \right]. \end{equation}
		\item Let $m\geq2$ and assume that $d$ is of the form\begin{equation}\label{eq:str-d}
		d(x) \ = \ -x(1-x) \sum_{i=0}^{m-2} s_i x^i,\quad x\in[0,1],\end{equation}
	 where $(s_i)_{i=0}^{m-2}$ is a decreasing sequence in $\Rb_+$. Define $(\beta,p)\in \R^{m-1}_+\times \Ps_m$ via \begin{equation}\label{eq:TranslationToSD}
		\beta_m\coloneqq s_{m-2},\ \ \forall\ell\in \, ]m-1],\  \beta_\ell\coloneqq s_{\ell-2}-s_{\ell-1},\ \ \text{and} \ \ \forall \ell\in \, ]m], i\in[\ell]_0,\  p_{i,\ell}=\1_{\{i=\ell\}}.
		\end{equation} Then $(\beta,p)\in \ccS_d$. Moreover, if $V_0=e_{n+1}$, then $\langle B_{L_t}(x), V_t \rangle = x^{L_t}$, for all $t\geq 0$. In particular, the Bernstein duality coincides with the moment duality, i.e. $ \E_x[X_t^n ] \ = \ \E_{n}[ x^{L_t} ].$
	\end{enumerate}
\end{corollary}

\begin{remark}
\citet{gonzalezcasanova2018} considered drift terms of the form ${-x(1-x)\sum_{i\geq0}s_ix^i}$ for a non-negative decreasing sequence $(s_i)_{i\in\Nb_0}$. They establish a moment duality to the line-counting process of an ASG. If $(s_i)_{i\in\Nb_0}$ has finite support, then their result agrees with Corollary~\ref{coro:bernstein-duality}--(2).
\end{remark}

\begin{remark}[Mutation-selection models]
The particular form of our drift term excludes models with mutations. For some forms of selection, \citep{EGT10,barbour2000,etheridge2009} obtain a weighted moment duality. \citep{GJL2016} further extend this to models with recombination. A moment duality for $\Lambda$-Wright--Fisher processes with mutation and genic selection is established in \cite{CM19}. For the diploid mutation-selection equation (i.e. if $\Lambda=0$ in~\eqref{eq:SDEoriginal02} and $d$ is a specific cubic polynomial), \citep{BCH2018interaction} formulate an ASG-based dual process on weighted trees. We believe that by introducing suitable operators that reflect mutations in the ancestral structures, the Bernstein duality translates to this framework.
\end{remark}

Proposition~\ref{ASPvsVCP} is proved in Section~\ref{s4.1}. The proof of Theorem~\ref{thm:bernstein-duality} and Corollary~\ref{coro:bernstein-duality} is provided in Section~\ref{s4.2}.

%%%%%%%%%%%%%%%%%%%%%%%%%%%%%%%%%%%%%%%%%%%%%%%%%%%%%%%%%%%%%%%%Subsection2.4%%%%%%%%%%%%%%%%%%%%%%%%%%%%%%%%%
\subsection{Properties of the Bernstein coefficient process and its leaf process}\label{s2.4}
In this section we expose some properties of the Bernstein coefficient process and its leaf process. These properties will be crucial for the statements about the ergodic behaviour of~$V$.

\smallskip

Recall that the leaf process depends only on the branching rates $\beta$ and on the measure $\Lambda$, and not on the colouring rule $p$. The following two quantities play a crucial role in the analysis of~$X$ and~$V$.
\begin{definition}[Effective branching rate and coalescence impact]\label{def:EffectiveBranchingRate}
For $\beta\in \R_+^{m-1}$, we define the \emph{effective branching rate} as
	\[ b(\beta) \ \coloneqq \ \sum_{\ell=2}^m \beta_\ell(\ell-1).\]
	For $\Lambda\in\Ms_f^*([0,1])$, we define the \emph{coalescence impact} as 
      \begin{equation*}
        c(\Lambda)\ \coloneqq\ \int_{[0,1]}\lvert \log(1-r)\rvert\frac{\Lambda(\dd r)}{r^2}. 
        \end{equation*}
\end{definition}
\begin{remark}\label{rem:interpretationCoalescenceImpact}
	The quantity $c(\Lambda)$ was introduced in~\cite{herriger2012conditions} as $\lim_{k\to\infty} \log(k)/\E[\tau_k]$, where $\tau_k$ is the absorption time of the $\Lambda$-coalescent started with $k$ blocks. \cite[Thm. 1]{GoWa18} shows that $c(\Lambda)=\lim_{k\to\infty} \log(k)/\tau_k$ in probability. This result says that the number of blocks in a $\Lambda$-coalescent decays at least exponentially. If $c(\Lambda)=\infty$ the decay is super-exponential. Note that $c(\Lambda)<\infty$ if and only if $\int_{[0,1]}r^{-1}\Lambda(\dd r)<\infty$ and $\int_{[0,1]}\lvert \log(1-r)\rvert\Lambda(\dd r)$. Hence, $c(\Lambda)<\infty$ if and only if the underlying $\Lambda$-coalescent has dust and the size of the last merger is tight \cite[Thms. 1 and 3]{GoWaSch18}.
\end{remark}
\subsubsection{Long time behaviour: invariant distributions}
This part contains the analysis of the long time behaviour of the processes $V$ and $L$. We start with the following simple criterion for positive recurrence or transience of the leaf process.
\begin{theorem}[Classification]\label{thm:degreeprocessrecurrenttransient(intro)}
Assume $\Lambda\neq \delta_1$. The leaf process~$L$ with parameters $(\beta,\Lambda)$ is
\begin{itemize}
	\item positive recurrent if $b(\beta)<c(\Lambda)$,
	\item transient if $b(\beta)>c(\Lambda)$.
\end{itemize}
\end{theorem}
\begin{remark}
If $\Lambda=\delta_1$, the communication class of~$1$ is always positive recurrent, see Corollary~\ref{coro:positiverecurrent}. If $\beta_2>0$, this communication class is~$\N$.
\end{remark}
\begin{remark}
$c(\Lambda)=\infty$ for the Kingman coalescent ($\Lambda=\delta_0$) and the Bolthausen-Sznitman model ($\Lambda=U[0,1]$, see~\citep{Bolthausen1998}). Therefore, for these models the leaf process is always positive recurrent. In contrast, for the Eldon-Wakeley coalescent ($\Lambda=\delta_c$ for some $c\in (0,1)$, see~\citep{Eldon2621}), we have $c(\Lambda)<\infty$.
\end{remark} 
\begin{remark}
 Under the assumption $c(\Lambda)< \infty$, the first part of Theorem~\ref{thm:degreeprocessrecurrenttransient(intro)} is already present in \citep[Thm. 4.6]{gonzalezcasanova2018}. They also show that if $b(\beta)>c(\Lambda)$, then the process~$L$ is not positive recurrent.
\end{remark}
A consequence of Theorem~\ref{thm:degreeprocessrecurrenttransient(intro)} is that if $b(\beta)<c(\Lambda)$, then the leaf process admits a unique stationary distribution. The latter can be characterised as the solution of a linear system of equations, see Eq.~\eqref{eq:fearnhead-recursion}. This system is a variant of the well-known Fearnhead recursions, which were introduced in \citep{fearnhead2002common} (see also \citep{Taylor2007}) for a Wright--Fisher diffusion model with mutation and genic selection, and it was later extended to a $\Lambda$-Wright--Fisher model in~\citep{BLW16}.
\smallskip

The next result tells us that the condition $b(\beta) < c(\Lambda)$ also assures the existence of invariant distributions for the Bernstein coefficient process. Uniqueness and convergence towards an invariant distribution are provided under some natural extra conditions.

\begin{proposition}[Invariant distributions] \label{prop:abs}
The Bernstein coefficient process $V$ keeps the entries $V_{t}(0)$ and $V_{t}(L_t)$ constant along time. Moreover, if $b(\beta) < c(\Lambda)$, then the following assertions hold.

	\begin{enumerate}  
		\item For every $a,b\in \R$, the Bernstein coefficient process $V$ has a unique invariant probability measure $\mu^{a,b}$ with 
		support included in $\{v\in \Rb^{\infty} : \ v_0=a,\, v_{\dim(v)-1}=b \}.$
		\item Let $V_{\infty}^{a,b}$ be a random variable with law~$\mu^{a,b}$. If $V_0=v$ with $v_0=a$ and $v_{\dim(v)-1}=b$, then \[V_t\ \xrightarrow[t\to\infty]{(d)}\ V_{\infty}^{a,b}.\]
	\end{enumerate}
\end{proposition}
The proofs of Theorem~\ref{thm:degreeprocessrecurrenttransient(intro)} and Proposition~\ref{prop:abs} are given in Sections~\ref{s5.1} and~\ref{subsec:bcpinvariant}, respectively. Other complementary results on the stationary distribution of the leaf process are given in Section~\ref{s5.2}. 
\subsubsection{Small time behaviour: coming down from infinity}
Now we consider the small time behaviour of the leaf process $L$ as the initial number of particles tends to infinity. It is not hard to see that $L$ is stochastically monotone in the initial value, i.e. for any $k\in\Nb$ and $t\geq 0$, $n\mapsto \Pb_n(L_t\geq k)$ is non-decreasing, see Remark \ref{SM}. The work of \citet[Thm. 1]{clifford1985} yields an order-preserving coupling of the sequence of leaf processes $\{(L^{n}_{t}:\,t>0)\}_{n\in\N}$, where $n$ indicates the initial value of $L^n$. This means that $L^n_t\leq L^{\hat{n}}_t$ for any $n\leq \hat{n}$ and any $t\geq 0$ almost surely. Hence, we can define the process $(L^{\infty}_t:\,t>0)$ valued in $\N\cup\{\infty\}$ as the monotone limit of $(L^n_t:\,t\geq 0)$ as $n\to\infty$. In particular, $\lim_{t\to0^+} L_t^\infty=\infty$. 

\begin{definition}[Coming down from infinity]
We say the leaf process \emph{comes down from infinity} (c.d.i.) if and only if for every $\varepsilon>0$,
$L_{\varepsilon}^{\infty}<\infty$ a.s.. We say that it stays infinite if for every $\varepsilon>0$,
$L_{\varepsilon}^{\infty}=\infty$ a.s..
\end{definition}

It follows from Remark~\ref{rem:interpretationCoalescenceImpact} that if the $\Lambda$-coalescent c.d.i., then $c(\Lambda)=\infty$. In particular, the corresponding leaf process is positive recurrent. In addition, the following generalisation of \citet[Prop. 23]{Pit99} holds.

\begin{theorem}[Criterion for c.d.i.]\label{teo:cdi-criterium(intro)}
Assume that $\Lambda$ has no mass at $1$. Then the leaf process $L$ either c.d.i.~or stays infinite. 
Furthermore, $L$ c.d.i.~if and only if the underlying $\Lambda$-coalescent c.d.i..
\end{theorem}

The proof of Theorem \ref{teo:cdi-criterium(intro)} is given in Section~\ref{s5.3}
%%%%%%%%%%%%%%%%%%%%%%%%%%%%%%%%%%%%%%%%%%%%%%%%%%%%%%%%%%%%%%%%Subsection2.5%%%%%%%%%%%%%%%%%%%%%%%%%%%%%%%%%
\subsection{Absorption probability and time to absorption}\label{s2.5} 
If the leaf process and the type-$a$ frequency process are in moment duality, one can typically translate the long time behaviour of~$L$ into (time) asymptotic properties of~$X$, see \cite{foucart2013impact, gonzalezcasanova2018}. This method extends to the Bernstein duality and leads to results on absorption probabilities and on the time to fixation.

Define $h(x)\coloneqq\P_x( \lim_{t\to\infty}X_t=1)$, i.e. $h(x)$ is the absorption probability of~$X$ in~$1$ starting from $X_0=x$. Using the notation of Proposition~\ref{prop:abs}, set $V_\infty\coloneqq V_\infty^{0,1}$. Recall that $L_{\infty}=\dim(V_\infty)-1$.
\begin{proposition}[Absorption probabilities] \label{prop:absX}
Assume $b(\beta) < c(\Lambda)$.
\begin{enumerate}
	\item For all $x\in[0,1]$, \begin{equation}
	 h(x) \ = \ \E\left[\langle  B_{L_\infty}(x), V_\infty \rangle\right]. \label{eq:represAbsoprtion}\end{equation}
    \item The boundary points $0$ and $1$ are accessible from any point $x\in (0,1)$, i.e. $h(x)\in (0,1)$.
\end{enumerate}
\end{proposition}
\begin{remark}
Note that \eqref{eq:represAbsoprtion} can be expressed as
 $$h(x)\ =\ \sum_{\ell= 0}^\infty \P(L_{\infty}=\ell) \sum_{i=0}^{\ell}d_{i,\ell}b_{i,\ell}(x),$$
where $d_{i,\ell}\coloneqq\E[\, V_{\infty}(i)\mid L_{\infty}=\ell]$. Moreover, under the extra assumption that $L_\infty$ admits exponential moments, we obtain a series expansion of~$h$ around $0$. See Proposition~\ref{lem:absorptionprobabilityanalytic} for more details. This is similar in spirit to \citet[Thm.~2.4]{BLW16}.
 \end{remark}

Given a colouring rule $p\in\Ps_m$, define $\bar{p}\in\Ps_m$ by $\bar{p}_{i,\ell}\coloneqq1-p_{\ell-i,\ell}$. Consider the Bernstein coefficient processes $V=(V_t^n:\,t \geq 0)$ and $W=(W_t^n:\,t\geq 0)$ starting at $e_{n+1}$, and with parameters $(\beta,p,\Lambda)$ and $(\beta,\bar{p},\Lambda)$, respectively. One can easily construct $V$ and~$W$ on the basis of the same~$L$. The next result relates the distribution of the time to absorption of $X$ in~$\{0,1\}$ to the mean asymptotic behaviour of the process $(V_t^n,W_t^n:\,t\geq 0)$ and to the hitting time of~$1$ for~$L^n$ as $n\to\infty$ .

\begin{proposition}[Absorption time]\label{prop:absorptiontime} Fix $(\beta,p,\Lambda)$ such that~$L$ c.d.i.. For $n\in \Nb$, let $V^n$ and $W^n$ as above. Define the random polynomial
\[ Q_t^n(x) \ \coloneqq \ \langle  B_{L_t^n}(x), V_t^n\rangle + \langle  B_{L_t^n}(1-x), W_t^n\rangle.\ \]
Let $T\coloneqq \inf\{t>0: X_t\in\{0,1\}\}$ and for $n\in \N\cup \{\infty\}$, let $\tau^{(n)}\coloneqq\inf\{ t>0 \ : \ L_t^{n} = 1 \}$. Then
\begin{enumerate}
	\item $\P_x(T\leq t) \ = \lim\limits_{n\to\infty}\ \E\left[ Q^{n}_t(x) \right]$
	\item $\E_{x}[T] \ = \ \lim\limits_{n\to\infty} \E\left[\int_0^{\tau^{(n)}} \big(1-Q_t^n(x)\big)\dd t \right]\ \leq \ \E\left[\tau^{(\infty)} \right].$
\end{enumerate}

In particular, $T<\infty$ almost surely.
\end{proposition}
An heuristic interpretation of $Q_t^{n}(x)$ in terms of the ASG is as follows. Consider the ASG $(\Gs_t^n,p)$ in $[0,t]$ and assign to each leaf the type~$a$ with probability~$x$ (resp.~$A$ w.p. $1-x$). Note that colouring $\Gs_t^n$ with colouring rule~$\bar{p}$ is the same as colouring $\Gs_t^n$ with colouring rule~$p$, but the role of~$a$ and~$A$ are interchanged. So $Q_t^{n}(x)$ corresponds to the conditional probability given $(L_u^n:\,u\in[0,t])$, that the roots have either all type~$a$ or all type~$A$. At the time the leaf process hits~$1$, the ASG has a bottleneck so that whatever is the type at the bottleneck, this is the type of the entire sample. Hence, the probability of fixation of one type before time $t$ is recovered by averaging $Q_t^n(x)$ over all possible realisations of $(L_u^n:\,u\in[0,t])$ and by taking $n\to\infty$, i.e. starting the ASG from the entire population.
\smallskip

Complementary properties and the proofs of the results of this section can be found in Section~\ref{s6}.

%%%%%%%%%%%%%%%%%%%%%%%%%%%%%%%%%%%%%%%%%%%%%%%%%%%%%%%%%%%%%%%%Subsection2.5%%%%%%%%%%%%%%%%%%%%%%%%%%%%%%%%%

\subsection{Minimality} \label{s2.6}
According to the results of Section~\ref{s2.1} and~\ref{s2.2}, infinitely many potential genealogies are associated to the SDE~\eqref{eq:SDEoriginal02} if the drift~$d$ is a polynomial vanishing at the boundary. In view of Theorem~\ref{thm:degreeprocessrecurrenttransient(intro)} and Proposition~\ref{prop:abs}, it seems natural to distinguish potential genealogies with respect  to their effective branching rate.

\begin{definition}[$b$-minimal selection decomposition]\label{def:minimalSD}
	Let $d$ be a polynomial with $d(0)=d(1)=0$. We refer to a selection decomposition~$(\beta,p)\in\ccS_d$ as $b$-minimal if $$b(\beta)=\inf_{(\beta',p')\in\ccS_d} b(\beta')\eqqcolon b_\star(d).$$
	We call $b_\star(d)$ the minimal effective branching rate of $d$.
\end{definition}
\smallskip
Note that if $f$ is a polynomial with $\deg(f)\leq m$ that vanishes at the boundary, then the first and last entries of its $m$-Bernstein coefficient vector are $0$. We denote by $\rho(f)\coloneqq (\rho_i(f))_{i\in[m-1]}\in\Rb^{m-1}$ the vector containing the entries~$i\in[m-1]$ of the $m$-Bernstein coefficient vector of $f$, i.e. $\rho(f)$ is the unique vector such that $f(x)=\sum_{i=1}^{m-1}\rho_i(f) b_{i,m}(x)$ for all $x\in\Rb$.
The following notion will play an important role in order to understand the structure of $b$-minimal selection decompositions.

\begin{definition}[$\lambda$-decomposability]\label{def:decomp}
Let $\lambda>0$. We say that a polynomial $d$ with $\deg(d)\leq m$ and $d(0)=d(1)=0$ is \emph{$\lambda$-decomposable} if it admits a selection decomposition with effective branching rate~$\lambda$. Similarly, we say that $v\coloneqq(v_i)_{i\in[m-1]}\in\Rb^{m-1}$ is \emph{$\lambda$-decomposable} if the polynomial $x\mapsto\sum_{i=1}^{m-1} v_i b_{i,m}(x)$ is $\lambda$-decomposable. Denote by $\Ss_\lambda\subset\Rb^{m-1}$ the set of $\lambda$-decomposable vectors, i.e.  
$$\cS_\lambda \ \coloneqq \ \{\rho(d_{\beta,p}):(\beta,p)\in \R_+^{m-1}\times \Ps_m,\,b(\beta) = \lambda \}.$$ 
\end{definition}
\begin{figure}[t!]
	\begin{center}
		\scalebox{0.6}{
			\begin{tikzpicture}
			%Box
			\draw[line width=.2mm] (-4,-4) -- (-4,4) -- (4,4) -- (4,-4) -- (-4,-4);
			\draw[line width=.2mm] (-3.9,0) -- (-4,0) node [left] {\scalebox{1.5}{$0$}};
			\draw[line width=.2mm] (-3.9,2) -- (-4,2) node [left] {\scalebox{1.5}{$0.2$}};
			\draw[line width=.2mm] (-3.9,-2) -- (-4,-2) node [left] {\scalebox{1.5}{$-0.2$}};
			\draw[line width=.2mm] (3.9,0) -- (4,0); %node [right] {$0$};
			\draw[line width=.2mm] (3.9,2) -- (4,2); %node [right] {$0.2$};
			\draw[line width=.2mm] (3.9,-2) -- (4,-2); %node [right] {$-0.2$};
			\draw[line width=.2mm] (0,4) -- (0,3.9);
			\draw[line width=.2mm] (-2,3.9) -- (-2,4);
			\draw[line width=.2mm] (2,4) -- (2,3.9);
			\draw[line width=.2mm] (0,-3.9) -- (0,-4) node [below] {\scalebox{1.5}{$0$}};
			\draw[line width=.2mm] (-2,-3.9) -- (-2,-4) node [below] {\scalebox{1.5}{$-0.2$}};
			\draw[line width=.2mm] (2,-3.9) -- (2,-4) node [below] {\scalebox{1.5}{$0.2$}};
			\node at (-4,-4.23) [left] {\scalebox{1.5}{$-0.4$}};
			\node at (-4,4) [left] {\scalebox{1.5}{$0.4$}};
			\node at (4,-4) [below] {\scalebox{1.5}{$0.4$}};
			
			\draw[line width=.2mm, gray] (-4,0) -- (4,0);
			\draw[line width=.2mm, gray] (0,-4) -- (0,4);
			\draw[line width=.2mm, gray] (4,-4) -- (-4,4);
			\draw[line width=.2mm, gray] (-4,-4) -- (4,4);
			% fsm with rate 1
			\draw[line width=.5mm] (-3.33,-3.33) -- (-1.80,1.80) -- (3.33,3.33) -- (3.33,-3.33) -- (-3.33,-3.33);
			
			\node at (0,-4.75) [below] {\scalebox{2}{$v_1$}};;
			\node at (-5.5,0) {\scalebox{2}{$v_2$}};;
			\end{tikzpicture}}
	\end{center}
	\caption{The region delimited by the bold lines corresponds to the set $\Ss_1$ of $1$-decomposable vectors.}
	\label{fig:regions}
\end{figure}
It turns out that $\Ss_\lambda$ is a polytope with the property that~$\Ss_\lambda=\lambda\Ss_1$, see Proposition~\ref{prop:s-polytope}. This property is crucial for proving the next result. 

\begin{proposition}\label{exminimal}
For every polynomial $d$ with $d(0)=d(1)=0$, there exists a $b$-minimal selection decomposition and the minimal effective branching rate satisfies the following relation
 \begin{equation}\label{eq:minimal-branching}
 b_\star(d) \ = \ \inf\{\lambda>0: \rho(d) \in \lambda \cS_1 \}.
 \end{equation}
\end{proposition}

The properties of $S_\lambda$ allow us to derive a recipe for finding a $b$-minimal selection decomposition, see Algorithm~\ref{teo:algorithm}. Moreover, they lead to the explicit construction of $\ccS_d$ in Corollary~\ref{coro:geometriccharacterisationced}.

\smallskip

The case $\deg(d)=3$ is relevant in applications, because it contains the selection term of a diploid Wright--Fisher model with dominance (see \ref{eq:wf-dominance}). In this case the faces of the polygon $\cS_\lambda$ have a natural biological interpretation. More specifically, we show that the face of $\cS_{b_{\star}(d)}$ that contains~$\rho(d)$ depends on the direction of selection and the nature of the dominance (recessive/dominant positive/negative selection). Fig.~\ref{fig:regions} illustrates~$\Ss_1$ for~$\deg(d)=3$. More details are exposed in Section~\ref{sec:minimalSDm3}.

\begin{remark}
	If we consider only potential genealogies that are $b$-minimal, we recover classical cases of the literature. For instance, the ASG of \citeauthor{KroNe97} is recovered as the only $b$-minimal dual to the Wright--Fisher diffusion with genic selection \citep{KroNe97, NeKro97}, and the ASG with a minority rule introduced by~\citeauthor{Ne99} as the only $b$-minimal dual of the Wright--Fisher diffusion with balancing selection~\citep{Ne99}.
\end{remark}

A classification of selection decompositions according to their effective branching rate provides only a partial picture of the connection between different ASGs. From a biological perspective, one is inclined to say that an ASG~$\Gs$ is better than $\tilde{\Gs}$ if one recovers~$\Gs$ by erasing superfluous branches from $\tilde{\Gs}$. This motivates the following definitions.	

\begin{definition}[Thinning]\label{def:thinning}
A lower-triangular stochastic matrix~$\Tb\coloneqq(\Tb_{{k,i}})_{k,i=1}^m$ is called a thinning mechanism. For $\beta\in \Rb_+^{m-1}$, define $\Tb\beta\in \Rb_+^{m-1}$ via\begin{equation*}
 (\Tb\beta)_\ell \coloneqq \sum_{k=\ell}^m \beta_k \Tb_{k,\ell},\qquad \ell\in\,]m].\end{equation*}
\end{definition}

Thinning mechanisms act on ASGs as follows. A $(\Tb \beta,p,\Lambda)$-ASG can be constructed from a $(\beta,p,\Lambda)$-ASG by removing at each $k$-branching point (exactly) $k-i$ lines (chosen uniformly at random) with probability $\Tb_{k,i}$, $k\in]m]$, $i\in[k]$.

\begin{definition}[Graph-minimal]\label{def:graphminimal}
Let~$d$ be a polynomial and $(\beta,p)\in \ccS_d$. We say that $(\beta,p)$ is \emph{graph-minimal} for~$d$ if and only if the only thinning mechanism~$\Tb$ such that $(\Tb\beta ,p')\in \ccS_d$ for a colouring rule~$p'$ is the identity.
\end{definition}
In other words, if $(\beta,p)$ is not \emph{graph-minimal} for~$d$, then there is a thinning mechanism $\Tb$ different from the identity and a colouring rule $p'$ such that the $(\Tb\beta,p',\Lambda)$-ASG still is a potential genealogy for \eqref{eq:SDEoriginal02} with drift $d$. 

\smallskip

Both notions of minimality aim to minimise in some sense the vector $\beta$ of branching rates, under the constraint that $\beta$ can be augmented by a colouring rule to an element of $\ccS_d$; this makes it a notion of minimality for selection decompositions. It is unclear if the two notions of minimality are equivalent. The next result provides a partial answer.
\begin{theorem}\label{prop:graph:minimal} 
Any $b$-minimal selection decomposition $(\beta,p)\in\ccS_d$ is also graph-minimal. 
\end{theorem}
For $\deg(d)=3$ we prove that the converse is also true. In particular, this shows that in this case $b$-minimal selection decompositions are the only ones not containing superfluous branches.

\begin{proposition}\label{prop:equivalenceminimality(intro)}
Assume~$\deg(d)=3$. For every $(\beta,p)\in\ccS_d$ with $b(\beta)>b_\star(d)$, there is a thinning mechanism $\Tb$ different from the identity and a colouring rule $p'$ such that $(\Tb\beta,p')\in\ccS_d$ is $b$-minimal.
\end{proposition}

\subsection{Open questions}
We list several open questions that stem from our work.
\begin{enumerate}
\item The present article only deals with a polynomial drift $d$ vanishing at the boundary. Our approach should easily extend to any continuous function $d$ of the form 
\[d(x) = \sum_{\ell=2}^\infty \beta_{\ell} \sum_{i=0}^\ell \left(p_{i,\ell} - \frac{i}{\ell}\right) b_{i,\ell}(x).\]
Can one give a simple characterisation of the set of functions that admit such a decomposition?
\item Is it true that, if the leaf process associated to a $b$-minimal selection decomposition is transient, then $0$ or $1$ are not accessible? This was answered positively by~\cite{gonzalezcasanova2018} for drifts of the form $-x(1-x)\sum_{i=0}^\infty s_i x^i$ for a non-negative decreasing sequence $(s_i)_{i\in\Nb_0}$. The general case seems more involved.
\item We showed that if the leaf process c.d.i.~then 
the fixation time has finite expected value. Is the converse true?
\item Does the equivalence between $b$-minimality and graph-minimality hold for polynomials $d$ with $\deg(d)>3$?

\item The true genealogy (which is random) is a metric tree embedded in the ASG. A consequence of Theorem~\ref{thm:infiniteed} is that infinitely many ASGs fit to a given drift. However, is the distribution of the embedded trees the same?

\item Mutation-selection models usually have a drift term that does not vanish at the boundary. We believe that the Bernstein duality extends to this setting by appropriately adapting the transitions of the Bernstein coefficient processes to capture the effect of mutations. 
\item \citet{barbour2000} use a moment duality to derive a transition function expansion for a Wright--Fisher model with mutation and selection (see also~\citep{EGT10}). Can one use the Bernstein duality approach to generalise such a transition function expansion to general selection models?
\item Duality methods have been successfully used in the study of spatial models with selection (see, e.g., \citep{etheridge2017, EFPS17}). It would be interesting to see if the concepts of selection decomposition and Bernstein duality can be extended to treat more general selection mechanisms in this setting.

\item Recently, \citet{CasanovaSmadi} have answered question (Q1') in a multidimensional setting (i.e with more than two types) and with mutations. They design a fixed-size Wright--Fisher population model whose asymptotic type frequencies converge to a multi-dimensional version of~\eqref{eq:SDEoriginal02}. In this framework they study fixation and extinction properties in some classical ecological models such as the rock-paper-scissor and food-web models. Given the intriguing biological applications presented in~\citep{CasanovaSmadi}, it would be interesting to investigate the extension of our duality result with regard to (Q2--Q4) in higher dimensions.
\end{enumerate}

%%%%%%%%%%%%%%%%%%%%%%%%%%%%%%%%%%%%%%%%%%%%%%%%%%%%%%%%%%%%%%%%%%%%%%%%%%%%%%%%%%%%%%%%%%%%%%%%%%%%%%%%%%%%%%%%%%%%%%%%%%%%%%%%%%%%%%%%%%%%%%%%%%%%%%%%%%%%%
%%%%%%%%%%%%%%%%%%%%%%%%%%%%%%%%%%%%%%%%%%%%%%%%%%%%%%%%%%%%%%%%Section3%%%%%%%%%%%%%%%%%%%%%%%%%%%%%%%%%%%%%%%%%%%%%%%%%%%%%%%%%%%%%%%%%%%%%%%%%%%%%%%%%%%%%
%%%%%%%%%%%%%%%%%%%%%%%%%%%%%%%%%%%%%%%%%%%%%%%%%%%%%%%%%%%%%%%%%%%%%%%%%%%%%%%%%%%%%%%%%%%%%%%%%%%%%%%%%%%%%%%%%%%%%%%%%%%%%%%%%%%%%%%%%%%%%%%%%%%%%%%%%%%%%

\section{From selection decompositions to selection mechanisms}\label{s3}
%%%%%%%%%%%%%%%%%%%%%%%%%%%%%%%%%%%%%%%%%%%%%%%%%%%%%%%%%%%%%%%%Subsection3.1%%%%%%%%%%%%%%%%%%%%%%%%%%%%%%%%%
\subsection{Existence and non-uniqueness of selection decompositions}\label{sect:mBernDeco}
In this section we prove Theorem~\ref{thm:infiniteed}, which states that every polynomial $d$ such that $d(0)=d(1)=0$ admits infinitely many selection decompositions (see Def.~\ref{def:SelectionDecomposition}). We also prove Proposition~\ref{exminimal} establishing the existence of selection decompositions that minimise the effective branching rate. The following proposition will be useful for the proofs of these results. Recall from Definition~\ref{def:decomp} that $\cS_\lambda$ is the set of $\lambda$-decomposable vectors, i.e.  
$\cS_\lambda =\{\rho(d_{\beta,p}):(\beta,p)\in \R_+^{m-1}\times \Ps_m,\,b(\beta) = \lambda \}$, where $d_{\beta,p}$ is given in~\eqref{eq:driftbetap}.
\begin{proposition}[Scaling property]\label{prop:s-polytope}
	For any $\lambda>0$ the set $\cS_\lambda$ is a polytope with the property that $\cS_\lambda = \lambda \cS_1$. In particular, $\Ss_\lambda\subseteq \Ss_{\lambda'}$ for $\lambda'\geq \lambda$. 
\end{proposition}

\begin{proof}
	The sets $G_\lambda\coloneqq \{\beta \in \R_+^{m-1}: b(\beta)=\lambda\}$ and $\Ps_m$ are polytopes. Since the Cartesian product of polytopes is a polytope, $G_\lambda \times \Ps_m$ is a polytope.
	For every $p\in\Ps_m$, the map $\beta\mapsto \rho(d_{\beta,p})$ is linear, and for every $\beta\in\Rb_+^{m-1}$, the map $p\mapsto\rho(d_{\beta,p})$ is affine. It follows that $\cS_\lambda$ is a polytope. The property $\cS_\lambda = \lambda \cS_1$ is a
	consequence of the fact that $G_\lambda = \lambda G_1$. It follows that $\Ss_{\lambda}=\lambda\Ss_1 \subseteq \lambda'\Ss_1=\Ss_{\lambda'}$ for $\lambda'\geq\lambda$. 
\end{proof}

Proposition~\ref{prop:s-polytope} provides a geometric framework to prove Theorem \ref{thm:infiniteed} and Proposition \ref{exminimal}.

\begin{proof}[Proof of Theorem \ref{thm:infiniteed}]
	Fix an arbitrary colouring rule $p\in\Ps_m$. For every $w\in \{0,1\}^{m-1}$, define the colouring rule $p^w$ by replacing $(p_{i,m})_{i=1}^{m-1}$ with $w$; keep the other entries of $p$ unchanged. Furthermore, set $a\coloneqq(a_\ell)_{\ell=2}^m \in\Rb^{m-1}$ with $a_m\coloneqq1/(m-1)$ and $a_\ell\coloneqq0$ for $\ell\neq m$. A straightforward calculation yields 
	$$\rho(d_{a,p^w})=\frac{w}{m-1}-\frac{c}{m(m-1)},$$
	where $c\coloneqq(i)_{i=1}^{m-1}$. Since $b(a)=1$, $\rho(d_{a,p^w})\in \Ss_1$ for all $w\in \{0,1\}^{m-1}$. $\Ss_1$ is a polytope and therefore 
	\[ \ \conv{\left\{0,\frac{1}{m-1}\right\}^{m-1}}-\frac{c}{m(m-1)}\subset\Ss_1,\]
	where $\conv{K}$ denotes the convex hull of~$K$. 
	Hence, $\cS_1$ contains an open neighbourhood of the origin of $\R^{m-1}$. Together with Proposition~\ref{prop:s-polytope} it follows that
	\begin{equation}\label{exSD}
\R^{m-1}\ = \ \cup_{\lambda>0} \cS_\lambda.
	\end{equation}
	A polynomial $d$ with~$\deg(d)=m$ admits a selection decomposition with effective branching rate $\lambda$ if and only if $\rho(d)\in \Ss_\lambda$. Identity~\eqref{exSD} implies that $\rho(d)\in \Ss_{\lambda_0}$ for some~$\lambda_0$. The scaling relation implies that $\rho(d)\in \Ss_\lambda$ for all $\lambda\geq \lambda_0$. In particular, $\ccS_d$ is infinite.
\end{proof}

\begin{proof}[Proof of Proposition~\ref{exminimal}]
Let $d$ be a polynomial with $d(0)=d(1)=0$ with $\deg(d)\leq m$. If $d=0$, $(0,p)\in\ccS_d$ for any $p\in\Ps_m$. Since $b(0)=0$, $(0,p)$ is $b$-minimal. Moreover, $\rho(0)=0\in\lambda\Ss_1$ for all $\lambda>0$,  and the result follows in this case. Assume now that $d\neq 0$. A consequence of~\eqref{exSD} and Proposition~\ref{prop:s-polytope} is that
	\begin{equation}\label{exmSD}
 \R^{m-1}\setminus\{0\}\ = \ \cup_{\lambda>0} \partial\cS_\lambda.
	\end{equation}
	Thus, there is $\lambda_\star(d)$ such that $\rho(d)\in\partial\Ss_{\lambda_\star(d)}$. It follows that there is $(\beta,p)\in \ccS_d$ with $b(\beta)=\lambda_\star(d)$. Moreover, by construction $\lambda_\star(d)=b_\star(d)$. The first part of the statement follows. By the definition of $\Ss_\lambda$ and $b_\star(d)$, we have $b_\star(d)=\inf\{\lambda>0: \rho(d)\in \Ss_\lambda\}.$ The scaling property $\Ss_\lambda=\lambda \Ss_1$ (see Prop.~\ref{prop:s-polytope}) yields \eqref{eq:minimal-branching}. 
\end{proof}

%%%%%%%%%%%%%%%%%%%%%%%%%%%%%%%%%%%%%%%%%%%%%%%%%%%%%%%%%%%%%%%%Subsection3.2%%%%%%%%%%%%%%%%%%%%%%%%%%%%%%%%%
\subsection{Existence, uniqueness, and convergence}\label{sect:EUS}
In this section we first show that the SDE~\eqref{eq:SDEoriginal02} is well-posed for every polynomial $d$ vanishing at the boundary. We then prove Theorem~\ref{teo:convergence}, which provides conditions for a sequence of Moran models to converge to the solution of the SDE~\eqref{eq:SDEoriginalBer}. Finally, we prove Corollary~\ref{Natsequence}, which provides an explicit sequence of Moran models converging to the solution of the SDE~\eqref{eq:SDEoriginalBer}.

\begin{lemma}[Existence and uniqueness]\label{wpSDE} Let $d:\Rb\to\Rb$ be a polynomial with $d(0)=d(1)=0$ and let $\Lambda$ be a finite measure on~$[0,1]$. Let $W$ be a standard Brownian motion and let $\tilde{N}$ be an independent compensated Poisson measure on $[0,\infty)\times(0,1]\times[0,1]$ with intensity $\dd t\times r^{-2}\Lambda(\dd r)\times \dd u$. Then for any $x_0\in[0,1]$, there is a pathwise unique strong solution $X\coloneqq(X_t:\,t\geq 0)$ to the SDE \eqref{eq:SDEoriginal02} such that $X_0=x_0$ and $X_t\in[0,1]$ for all $t\geq 0$.
\end{lemma}
\begin{proof}
	First set $a(x)\coloneqq d(x)$, $\sigma(x)\coloneqq\sqrt{\Lambda(\{0\})x(1-x)}$ and $g_0(x,r,u)\coloneqq r((1-x)\1_{\{u\leq x\}}-x\1_{\{u>x\}})$ for $x\in[0,1]$, $(r,u)\in (0,1]\times[0,1]$, complemented by $a(x)\coloneqq\sigma(x)\coloneqq g_0(x,r,u)=0$ whenever $x\notin[0,1]$. With these definitions the SDE 
\begin{equation}
\dd X_t\ = \ a(X_t)\dd t\ + \ \sigma(X_t)\, \dd W_t+\int\limits_{(0,1]\times [0,1]}g_0(X_t,r,u)\,\tilde{N}(\dd t, \dd r,\dd u), \quad X_0=x_0\in \R,\label{eq:SDEext}
\end{equation}
extends the SDE~\eqref{eq:SDEoriginal02} to $\R$. Clearly $a$ and $\sigma$ are continuous and $g_0$ is measurable. Moreover, for $x\in[0,1]$ and $(r,u)\in (0,1]\times[0,1]$, we have $-x\leq g_0(x,r,u)\leq 1-x$. It follows from \citep[Prop. 2.1]{FuLi10} that any solution of \eqref{eq:SDEext} starting at a point $x_0\in[0,1]$ remains in $[0,1]$ and therefore is also a solution of~\eqref{eq:SDEoriginal02}. The converse is by construction true. It remains to prove existence and pathwise uniqueness of strong solutions of the SDE \eqref{eq:SDEext}. We do this via \citep[Thm.~5.1]{li2012strong}. We need to verify conditions (3a), (3b), and (5a) of that paper. First note that $a$ is Lipschitz continuous (it is a polynomial in $[0,1]$), and hence condition (3a) is satisfied. Condition~(3b) concerns only $\sigma$ and $g_0$. In 
	\citep[Lem.~3.6]{gonzalezcasanova2018} it was proved that there is a constant $c>0$ such that $\lvert\sigma(x)-\sigma(y)\rvert^2\leq c\lvert x-y\rvert$. In addition, a straightforward calculation shows that
	$$\int_{(0,1]\times[0,1]}\lvert g_0(x,r,u)-g_0(y,r,u)\rvert^2\,\frac{\Lambda(\dd r)}{r^2}\,\dd u\leq \Lambda((0,1])\lvert x-y\rvert.$$
	Hence, condition (3b) is satisfied. It remains to verify condition (5a). Since $a$ and $\sigma$ are bounded, this amounts to prove that $x\mapsto\int_{(0,1]\times[0,1]} g_0(x,r,x)^2r^{-2}\Lambda(\dd r)\dd u$ is bounded. This directly follows from the fact that $g_0(x,r,u)^2\leq 2r^2$. 
\end{proof}
\begin{lemma}[Operator core]\label{core}
	The solution of the SDE~\eqref{eq:SDEoriginal02} is a Feller process with generator $A$ satisfying
	\begin{equation}\label{eq:generatorX}
	Af(x)=A_{\mathrm{s}} f(x)+A_{\mathrm{wf}}f(x)+ A_{\Lambda}f(x),\quad\textrm{for any}\quad f\in\Cs^2([0,1]),\  x\in[0,1],
	\end{equation}
	where
	\begin{align*}
	A_{\mathrm{s}} f(x)&=d(x)f'(x),\qquad A_{\mathrm{wf}}f(x)=\frac{\Lambda(\{0\})}{2}\,x(1-x)\,f''(x),\\
	A_{\Lambda}f(x)&=\int_{(0,1]}x\big[f(x+r(1-x))-f(x)\big]+(1-x)\big[f(x-rx)-f(x)\big]\frac{\Lambda(\dd r)}{r^2}.
	\end{align*}
	Moreover, $\Cs^\infty([0,1])$ is a core for $A$.
\end{lemma}
\begin{proof}
	Let $X$ be the unique strong solution of the SDE~\eqref{eq:SDEoriginal02}. It follows from standard theory of SDEs that~$X$ is a strong Markov process with generator $A$ satisfying \eqref{eq:generatorX}. Since pathwise uniqueness implies weak uniqueness (see \citep[Thm.~1]{BLP15}), we infer from \citep[Cor.~2.16]{Ku11} that the martingale problem associated to $A$ in $\Cs^{2}([0,1])$ is well-posed. Moreover, an inspection of the proof shows that this is also true in $\Cs^{\infty}([0,1])$. Using \cite[Prop. 2.2]{vanC92}, we infer that $X$ is Feller. The fact that $\Cs^\infty([0,1])$ is a core follows then from \cite[Thm.~2.5]{vanC92}.
\end{proof}

We now prove the convergence of the Moran-type models to the solution of the SDE~\eqref{eq:SDEoriginalBer}.
\begin{proof}[Proof of Theorem \ref{teo:convergence}]
	Let $X$ be the solution of \eqref{eq:SDEoriginalBer} with $X_0=x_0\in[0,1]$. Denote by $A$ and $A^{N}$ the generator of $X$ and $\tilde{X}^{(N)}\coloneqq(X^{(N)}_{Nt}:\,t\geq 0)$, respectively. 
	The main ingredient of the proof is to show that for every $f\in\Cs^\infty([0,1])$
	\begin{equation}
	\sup_{x\in\square_N}\lvert A^{N} f\lvert_{\square_N}(x) - Af(x)\rvert \xrightarrow[N\to\infty]{} 0, \label{eq:convgenerators}
	\end{equation}
	where $\square_N\coloneqq\{k/N:k\in[N]_0\}$ and $f\lvert_{\square_N}$ denotes the restriction of $f$ to $\square_N$. Let us first assume that this is true. Since $\Cs^{\infty}([0,1])$ is a core for $A$ (see Lemma \ref{core}), we can apply \citep[Thm.~1.6.1 and Thm.~4.2.11]{Ku86} to deduce the convergence in distribution of $\tilde{X}^{(N)}$ to $X$. It remains to prove \eqref{eq:convgenerators}. 
	
	Let $K^{n}_{m,k}\sim\hypdist{n}{m}{k}$ for $n\geq m\vee k$, $i\leq k\wedge m$. Define the discrete differential operators 
	$$D_h f(x)\coloneqq\frac{f(x+h)-f(x)}{h}\quad\textrm{and}\quad D_h^{(2)}f(x)\coloneqq\frac{f(x+h)+f(x-h)-2f(x)}{h^2}.$$ Then the generator~$A^N$ takes the form 
	\begin{equation}
	A^Nf(x)=\sum_{\ell=2}^m A^{N}_{\text{s},\ell}f(x)+A^N_{\mu_N}f(x)+A^N_{0}(x),\qquad f:\square_N\to\R,\, x\in\square_N,\label{eq:generatorMM}
	\end{equation} 
	where 
	\begin{align*}
	A^N_{0}f(x)&= \frac{\mu_N(\{0\})}{2}\,\frac{N}{N-1} x(1-x)\, D_{\frac{1}{N}}^{(2)}f(x),\\
	A^{N}_{\text{s},\ell}f(x) &=N \beta_\ell^{(N)} \left\{(1-x)\,\E\left[p_{K^{N-1}_{Nx,\ell-1},\ell}\right]D_{\frac{1}{N}}f(x)-x\,\E\left[1-p_{K^{N-1}_{Nx-1,\ell-1}+1,\ell}\right]D_{-\frac{1}{N}}f(x)\right\},\\
	A^N_{\mu_N}f(x)&=N^2\sum_{r=1}^{N-1} \! \left\{x\E\left[f\!\left(x+\frac{r-K^{N-1}_{Nx-1,r}}{N}\right)\!-f(x)\right]\!+(1-x)\E\left[f\!\left(x-\frac{K^{N-1}_{Nx-1,r}}{N}\right)\!-f(x)\right]\right\}\mu_N(r),
	\end{align*}
	for $x\in \square_N\setminus\{0,1\}$, and the three operators vanish at $0$ and $1$.
	Similarly, consider the generator $A$ from~\eqref{eq:generatorX} with 
	$d(x)=d_{\beta,p}(x)$ given by \eqref{eq:driftbetap}. In particular, the selective term admits a decomposition as $A_{\mathrm{s}}=\sum_{\ell=2}^m A_{\mathrm{s},\ell}$, where
	$A_{\mathrm{s},\ell} f(x)\coloneqq\beta_\ell(\E[p_{B_{\ell,x},\ell}]-x)f'(x)$ and $B_{\ell,x}\sim\bindist{\ell}{x}.$
	We first deal with the term associated to small neutral reproduction. Appropriate Taylor expansions and the triangular inequality yield for $f\in\Cs^{\infty}([0,1])$
	\begin{equation}
	\sup_{x\in\square_N}\lvert A^{N}_{0} f\lvert_{\square_N}(x) - A_{\mathrm{wf}}f(x)\rvert \leq \frac{N}{N-1}\mu_N(\{0\})\frac{\lVert f'''\rVert_{\infty}}{6N}+\left\lvert \frac{N}{N-1}\mu_N(\{0\})-\Lambda(\{0\})\right\rvert\frac{\lVert f''\rVert_{\infty}}{2}\xrightarrow[N\to\infty]{}0. \label{wfterm}
	\end{equation}
	For the term associated to $\ell$-replacements, first note that 
	$$(1-x)\E\left[p_{K^{N-1}_{Nx,\ell-1},\ell}\right]+x\E\left[p_{K^{N-1}_{Nx-1,\ell-1}+1,\ell}\right]=\E\left[p_{K^{N}_{Nx,\ell},\ell}\right].$$
	Consequently, appropriate Taylor expansions and the triangular inequality lead to
	\begin{equation*}
	\sup_{x\in\square_N}\lvert A^{N}_{\text{s},\ell} f\lvert_{\square_N}(x) - A_{\mathrm{s},\ell}f(x)\rvert \leq \lvert N\beta_\ell^{(N)}-\beta_\ell\rvert \lVert f'\rVert_{\infty}+
	\beta_\ell\left(\frac{\lVert f''\rVert_{\infty}}{2N}+\lVert f'\rVert_{\infty}R_N^\ell\right),
	\end{equation*}
	where 
	$$R_N^\ell\coloneqq\sup_{x\in\square_N}\left\lvert \E\left[p_{K^{N}_{Nx,\ell},\ell}\right]-\E\left[p_{B_{\ell,x},\ell}\right]\right\rvert\leq \binom{\ell}{\lfloor\frac{\ell}{2}\rfloor}\frac{2\ell^3}{N-\ell}.$$
	The previous inequality  follows from Lemma \ref{hyptobin}. We conclude that
	\begin{equation}\label{ellterm}
	\lim_{N\to\infty }\sup_{x\in\square_N}\lvert A^{N}_{\text{s},\ell} f\lvert_{\square_N}(x) - A_{\mathrm{s},\ell}f(x)\rvert =0.
	\end{equation}
	For the last term, note that
	$$\lvert A^{N}_{\mu_N} f\lvert_{\square_N}(x) - A_{\Lambda}f(x)\rvert \leq \varepsilon_{N,1}(x)+\varepsilon_{N,2}(x)+\varepsilon_{N,3}(x),$$
	where 
	\begin{align*}
	\varepsilon_{N,1}(x)&\coloneqq\left\lvert N^2\sum_{r=1}^{N-1} \mu_N(r) x\E\left[f\left(x+\frac{r-K^{N-1}_{Nx-1,r}}{N}\right)-f\left(x+\frac{r(1-x)}{N}\right)\right]\right\rvert,\\
	\varepsilon_{N,2}(x)&\coloneqq\left\lvert N^2\sum_{r=1}^{N-1} \mu_N(r) (1-x)\E\left[f\left(x-\frac{K^{N-1}_{Nx-1,r}}{N}\right)-f\left(x-\frac{rx}{N}\right)\right]\right\rvert,\\
	\varepsilon_{N,3}(x)&\coloneqq\left\lvert N^2\sum_{r=1}^{N-1} \mu_N(r)\left\{ x f\left(x+\frac{r(1-x)}{N}\right)+(1-x)f\left(x-\frac{rx}{N}\right)-f(x)\right\}-A_\Lambda f(x)\right\rvert.
	\end{align*}
	Second order Taylor expansions with integral remainder around $x+r(1-x)/N$ and $x-rx/N$, and standard properties of the hypergeometric distribution lead to
	$$\sup_{x\in\square_N}\big(\varepsilon_{N,1}(x)+\varepsilon_{N,2}(x)\big)\leq\left(\frac{N}{N-1}\lVert f'\rVert_{\infty}+ \lVert f''\rVert_{\infty}\right)\sum_{r=1}^{N-1} \mu_N(r)r.$$
	Moreover, for any $\gamma\in(0,1)$, we have
	$$\sum_{r=1}^{N-1} \mu_N(r)r=\sum_{1\leq r\leq \gamma N} \mu_N(r)r+\sum_{N\gamma<r\leq N } \mu_N(r)r\leq M_{\mu_N}\left(T^N{\mu_N}([0,\gamma])-\frac{\mu_N(\{0\})}{M_{\mu_N}}+\frac{1}{\gamma N}T^N\mu_N([\gamma,1])\right).$$
	Hence, the Portmanteau theorem yields 
	$\limsup_{N\to\infty}\sum_{r=1}^{N-1} \mu_N(r)r\leq \Lambda((0,\gamma]).$
	Since this holds for any $\gamma\in(0,1)$, we conclude that the previous $\limsup$ is $0$. Hence,
	\begin{equation}\label{eq12}
	\lim_{N\to\infty }\sup_{x\in\square_N}\big(\varepsilon_{N,1}(x)+\varepsilon_{N,2}(x)\big)=0.
	\end{equation}
	Now we set $f_x(r)\coloneqq r^{-2}(x[f(x+r(1-x))-f(x)]+(1-x)[f(x-rx)-f(x)])$. Note that $\lvert f_x(r)\rvert \leq \lVert f''\rVert_{\infty}/2$. Using this and the definition of $T^N\mu_N$, we get
	\begin{equation}\label{en31}
	\varepsilon_{N,3}(x)\leq \lvert M_{\mu_N}-\Lambda([0,1])\rvert \frac{\lVert f''\rVert_{\infty}}{2}+\Lambda([0,1])\left\lvert \int_{(0,1]} f_x(r)T^N\mu_N(\dd r)-\int_{(0,1]}f_x(r)\frac{\Lambda(\dd r)}{\Lambda([0,1])}\right\rvert.
	\end{equation}
	For any $\gamma\in(0,1)$, we have
	\begin{equation}\label{en32}
	\left\lvert \int_{(0,\gamma]} f_x(r)T^N\mu_N(\dd r)-\int_{(0,\gamma]}f_x(r)\frac{\Lambda(\dd r)}{\Lambda([0,1])}\right\rvert\leq \frac{\lVert f''\rVert_{\infty}}{2}\left(T^N\mu_N((0,\gamma])+\frac{\Lambda((0,\gamma])}{\Lambda([0,1])}\right),
	\end{equation}
	and
	\begin{align}\label{en33}
	\left\lvert \int_{ [\gamma,1]} f_x(r)T^N\mu_N(\dd r)-\int_{[\gamma,1]}f_x(r)\frac{\Lambda(\dd r)}{\Lambda([0,1])}\right\rvert&\leq \left\lvert \int_{ [0,1]} \tilde{f}_x(r)T^N\mu_N(\dd r)-\int_{[0,1]}\tilde{f}_x(r)\frac{\Lambda(\dd r)}{\Lambda([0,1])}\right\rvert\nonumber\\
	&\qquad\qquad+ \frac{\lVert f''\rVert_{\infty}}{2}\left\lvert T^N\mu_N([0,\gamma])-\frac{\Lambda([0,\gamma])}{\Lambda([0,1])}\right\rvert,
	\end{align}
	where $\tilde{f}_x$ is the continuous function that coincides with $f_x$ on $[\gamma,1]$, and is constant in $[0,\gamma]$. Note that $\tilde{f}_x$ is Lipschitz continuous and bounded. Moreover, 
	$\lVert \tilde{f}_x \rVert_{\text{Lip}}\leq 4\lVert f \rVert_{\infty}/\gamma^3+2\lVert f'\rVert_{\infty}/\gamma^2\coloneqq C_\gamma(f),$
	where 
	$\lVert f\rVert_{\text{Lip}}\coloneqq\lVert f\rVert_\infty \vee \sup_{x\neq y} \lvert f(x)-f(y)\rvert /\lvert x-y\rvert $ denotes the bounded Lipschitz norm of $f$.
	Thus,
	\begin{equation}\label{en34}
	\left\lvert \int_{ [0,1]} \tilde{f}_x(r)T^N\mu_N(\dd r)-\int_{[0,1]}\tilde{f}_x(r)\frac{\Lambda(\dd r)}{\Lambda([0,1])}\right\rvert\leq C_\gamma(f) d_{\text{Lip}}\left(T^N\mu_N,\frac{\Lambda}{\Lambda([0,1])}\right),
	\end{equation}
	where $d_{\text{Lip}}(\nu_1,\nu_2)\coloneqq\sup\{\lvert\int f d\nu_1-\int f d\nu_2\rvert: \lVert f\rVert_{\text{Lip}}\leq 1\}$ denotes the bounded Lipschitz metric in the space of probability measures on $[0,1]$. Assume that $\gamma$ is a continuity point of $\Lambda$. Combining \eqref{en31}, \eqref{en32}, \eqref{en33}, \eqref{en34}, and letting $N\to\infty$, we deduce
	\begin{equation}
	\limsup_{N\to\infty }\sup_{x\in\square_N}\varepsilon_{N,3}(x)\leq \lVert f''\rVert_{\infty}\frac{\Lambda((0,\gamma])}{\Lambda([0,1])}.
	\end{equation}
	This holds for any continuity point $\gamma\in(0,1)$. Hence, the previous limit exists and equals $0$. Together with \eqref{eq12}, this implies that
	\begin{equation}\label{Lterm}
	\lim_{N\to\infty }\sup_{x\in\square_N}\lvert A^{N}_{\mu_N} f\lvert_{\square_N}(x) - A_{\Lambda}f(x)\rvert=0.
	\end{equation}
	This ends the proof.
\end{proof}
Finally, we prove Corollary~\ref{Natsequence}, which provides for given~$(\beta,p,\Lambda)$ an explicit sequence Moran models converging to the solution of~\eqref{eq:SDEoriginalBer}.

\begin{proof}[Proof of Corollary~\ref{Natsequence}]
	It is enough to show that conditions (1) and (2) of Theorem \ref{teo:convergence} are satisfied. Since $N\beta^{(N)}=\beta$ and $\mu_N(\{0\})=\Lambda(\{0\})$, condition~(1) and the first part of condition~(2) are satisfied. To prove the other two parts of condition~(2), it suffices to show that for every $f\in\Cs([0,1])$
	$$\sum_{k=1}^{N-1}f\left(\frac{k}{N}\right)\frac{k^2}{N^2}\lambda_{N,k+1}^0\binom{N}{k+1}\xrightarrow[N\to\infty]{}\int_{(0,1]}f(r)\Lambda(dr).$$
	By the definition of $\lambda_{N,k+1}^0$ and a straightforward calculation, 
	$$\sum_{k=1}^{N-1}f\left(\frac{k}{N}\right)\frac{k^2}{N^2}\lambda_{N,k+1}^0\binom{N}{k+1}=\frac{N-1}{N}\int_{(0,1]}\E\left[f\left(\frac{B_{N-2,r}+1}{N}\right)\frac{B_{N-2,r}+1}{B_{N-2,r}+2}\right]\,\Lambda(dr),$$
	where $B_{N-2,r}\sim\bindist{N-2}{r}.$ The result follows then as an application of the law of large numbers and the dominated convergence theorem.
\end{proof}
We end this section with an auxiliary result used in the proof of Theorem~\ref{teo:convergence}. For $N\in\N$, $m_N\in[N]_0$, and $\ell\in[N]$, consider $K^N_{m_N^{},\ell}\sim\hypdist{N}{m_N^{}}{\ell}$. It is well-known that
$$\lim_{N\to\infty}\frac{m_N^{}}{N}=p\in[0,1]\quad\Rightarrow\quad K^N_{m_N^{},\ell}\xrightarrow[N\to\infty]{(d)} B_{\ell,p}.$$ 
The next lemma provides a uniform version of this result.
\begin{lemma}\label{hyptobin}
	For $N\in\N$, $\ell\in[N]$, and $x\in[0,1]$, we have
	$$\sup_{x\in\square_N,i\in[\ell]_0}\left\lvert\P(K^{N}_{Nx,\ell}=i)-\P(B_{\ell,x}=i)\right\rvert\leq \binom{\ell}{\lfloor\frac{\ell}{2}\rfloor}\frac{2\ell^2}{N-\ell+1}.$$ 
\end{lemma}
\begin{proof}
	First note that
	$$\left\lvert\P(K^{N}_{Nx,\ell}=i)-\P(B_{\ell,x}=i)\right\rvert=\binom{\ell}{i}\left\lvert \prod_{k=0}^{i-1}\frac{Nx-k}{N-k}\prod_{k=0}^{\ell-i-1}\frac{N(1-x)-k}{N-i-k}- x^i(1-x)^{\ell-i}\right\rvert.$$
	The result then follows from the fact that if $(a_n)_{n=0}^m$ and $(b_n)_{n=0}^m$ are two sequences of numbers in $[0,1]$, then 
	$\lvert\prod_{k=0}^m a_k-\prod_{k=0}^m b_k\rvert\leq \sum_{k=0}^m\lvert a_k-b_k\rvert .$
	
\end{proof}

\section{Ancestral structures and Bernstein duality}\label{s4}
\subsection{Ancestral selection graph and ancestral selection polynomial} \label{s4.0}
In this section we formalise the definition of the ASG and the ASP. We start out by describing a class of \emph{directed acyclic graphs} (DAGs) that encode the branching-coalescing system underlying the ASG. The type of graph illustrated in Fig.~\ref{fig:ASG} is what we have in mind. The vertices of this graph have a label and a time coordinate. Each label has a birth and a death time. There are \emph{horizontal} and \emph{vertical} edges. Vertices having the same label are connected by a directed horizontal edge following the increasing time direction (i.e. from right to left). Vertical edges connect vertices with different labels and encode branching and coalescence events. 

\begin{figure}[t!]
	
	\centering
	\scalebox{0.8}{\begin{tikzpicture}
		
		\draw[dotted] (9.5,-0.5) -- (9.5,4.5);
		\draw[dotted] (8,-0.5) -- (8,4.5);
		\draw[dotted] (6,-0.5) -- (6,4.5);
		\draw[dotted] (4,-0.5) -- (4,4.5);
		\draw[dotted] (2,-0.5) -- (2,4.5);
		\node[right] at (9.5,-0.5) {$0$};
		\node[right] at (8,-0.5) {$t_1$};
		\node[right] at (6,-0.5) {$t_2$};
		\node[right] at (4,-0.5) {$t_3$};
		\node[right] at (2,-0.5) {$t_4$};
		
		% asg Horizontal
		\draw[opacity=1, line width=.5mm, marking1]  (9.5,3.5) -- (2,3.5);
		
		\draw[opacity=1, line width=.5mm, marking1] (8,3) -- (2,3);
		
		\draw[opacity=1, line width=.5mm, marking1] (6,4) -- (0,4);
		\draw[opacity=1, line width=.5mm, marking1] (6,2.5) -- (2,2.5);
		\draw[opacity=1, line width=.5mm,marking1]  (6,2) -- (0,2);
		\draw[opacity=1, line width=.5mm,marking1] (9.5,1) -- (4,1) ;
		\draw[opacity=1, line width=.5mm] (4,0.5) -- (3.7,.5) ;
		\draw[opacity=1, line width=.5mm,marking1] (3.7,0.5) -- (0,.5) ;
		\draw[opacity=1, line width=.5mm,marking1]  (9.5,0) -- (4,0) ;
		\draw[opacity=1, line width=.5mm,marking1] (2,1.5) -- (0,1.5);

		% asg vertical
		\draw[opacity=1, line width=.5mm, {Stealth[length=2mm,width=2mm,open]}-] (8,3) -- (8,3.5);
		\draw[opacity=1, line width=.5mm, -{Stealth[length=2mm,width=2mm,open]}] (6,3) -- (6,2.5);
		
		\draw[line width=.5mm, -{Stealth[length=2mm,width=2mm,open]}] (6,3) .. controls (6.3,2.5) and (6.2,2.5) .. (6,2);
		\draw[line width=.5mm, -{Stealth[length=2mm,width=2mm,open]}] (6,3) .. controls (6.1,3.5) and (6.1,3.5) .. (6,4);

		\draw[line width=.5mm, -{Stealth[length=2mm,width=2mm,open]}] (6,3) .. controls (6.1,3.5) and (6.1,3.5) .. (6,4);
		
		\draw[line width=.5mm, -{Stealth[length=2mm,width=2mm,open]}] (2,3.5) .. controls (1.7,3) and (1.7,2) .. (2,1.5);
		\draw[line width=.5mm, -{Stealth[length=2mm,width=2mm,open]}] (2,3) .. controls (1.8,2.5) and (1.8,2) .. (2,1.5);
		\draw[line width=.5mm, -{Stealth[length=2mm,width=2mm,open]}] (2,2.5) .. controls (1.8,2.2) and (1.8,1.9) .. (2,1.5);

		\draw[opacity=1, line width=.5mm, {Stealth[length=2mm,width=2mm,open]}-] (4,.5) -- (4,0);
		\draw[opacity=1, line width=.5mm, -{Stealth[length=2mm,width=2mm,open]}] (4,1) -- (4,0.5);

		\end{tikzpicture}  }
	\caption{The bc-DAG representation of the branching-coalescing particle system associated to an ASG. Every vertex of the graph corresponds to an instance of a particle. The vertical coordinate corresponds to its label (in $[0,1]$), the horizontal coordinate to time. Vertical edges arise at times $t_1,\ldots,t_4$. Horizontal edges connect in increasing time direction vertices with the same label.}
	\label{fig:ASG}
\end{figure}

\begin{definition}[bc-DAG] Let $\Us$ be a countable set. A DAG $G\coloneqq(\Vs,\Es)$ is a \emph{branching-coalescing DAG} (bc-DAG) with label set $\Us$ if it has the following properties.
	\begin{itemize}
		\item The vertex set is of the form $\Vs=\cup_{u\in\Us}\{u\}\times[b_u,d_u]$ for some $\{b_u,d_u\}_{u\in\Us}\subset\Rb_+$ such that: for all $u\in\Us$, $b_u<d_u$; for any $t\in D\coloneqq\{d_u:\,u\in\Us\}$, there is a unique $u(t)\in\Us$ such that $b_{u(t)}=t$; and for any $t\in I\coloneqq\{b_u:\,u\in\Us\}\setminus D$, there is at least one $v\in\Us$ with $t\in(b_{v},d_{v})$.
		\item The edge set is of the form $\Es=\Es^h\cup\Es^v$, with $\Es^h\coloneqq\{((u,s),(u,t)):\, u\in\Us,\,b_u\leq s< t\leq d_u\}$ and $\Es^v\subset\{((u,t),(v,t)):\, u,v\in\Us,\, u\neq v,\, t\geq 0\}$.
		\item The sets $\Es^v(t)\coloneqq\Es^v\cap (\Us\times\{t\})^2$ are characterised as follows. For any $t\notin D\cup I$, $\Es^v(t)=\emptyset$. For all $t\in D$, $\Es^v(t)=\{((w,t),(u(t),t)):\,w\in\Us\textrm{ with }d_w=t\}$. For any $t\in I$, there is a unique $w(t)\in\Us$ with $t\in(b_{w(t)},d_{w(t)})$  such that $\Es^v(t)=\{((w(t),t),(u,t)):\,u\in\Us\textrm{ with }b_u=t\}$. 
		\item For all $t\geq 0$, the sets $\{u\in\Us:\, b_u\leq t\}$ and $\Es^v\cap(\Us\times[0,t])^2$ are finite. 
	\end{itemize}
	The \emph{vertical outdegree} of $(u,t)\in \Vs$ is the cardinality of $\{(w,t):\, ((u,t),(w,t))\in\Es^v(t)\}$; the \emph{vertical indegree} of $(u,t)\in \Vs$ is the cardinality of $\{(w,t):\, ((w,t),(u,t))\in\Es^v(t)\}$. For simplicity we assume that there is $v\in \Us$ with $b_v=0$, and we call the elements of $R(G)\coloneqq\{(u,0)\in \Vs: u\in \Us\}$ \emph{roots}. We denote by $G_t$ the subgraph of $G$ induced by all the vertices in $(u,s)\in\Vs$ with $s\leq t$; the elements of $L(G_t)\coloneqq\{(u,t)\in\Vs:\, u\in\Us,\, t\neq d_u\}$ are called \emph{leaves} of $G_t$.
\end{definition}
Note that a bc-DAG is completely determined by the label set $\Us$, the collection of points $\{b_u,d_u\}_{u\in\Us}$, and the function $w:I\to \Us$.
Next, we formalise the propagation of types in an ASG. It is instructive to have Figs.~\ref{fig:representASG} and~\ref{fig:ASG} in mind. Intuitively, colours will propagate following reverse directed paths (if we invert the direction of all the edges), with vertices in $\{(w(t),t): t\geq 0\}$ acting as barriers; the type after such a barrier is determined randomly using the colouring rule $p$. This motivates the following notion. For $x,y\in\Vs$ we say that~\emph{$x$ is colour-connected to the right of $y$} if there is a directed path from $x$ to $y$, and there is no path from $x$ to $y$ using vertices in $\{(w(t),t): t\geq 0\}$; we then write $y\vdash x$. We are now ready to state the definition of a random colouring of a bc-DAG.

\begin{definition}[Colouring of a bc-DAG]\label{colouring}
	Fix $m\in\Nb\setminus\{1\}$, $t\geq 0$, and a bc-DAG~$G\coloneqq(\Vs,\Es)$ such that its vertices have vertical outdegree at most~$m-1$. A \emph{leaves-colouring} of~$G_t$ is a vector~$\vec{c}:=(c_\ell)_{\ell\in L (G_t)}\in\{a,A\}^{L(G_t)}$. Given a leaves-colouring~$\vec{c}$ and a colouring rule~$p\in\Ps_m$, we randomly colour the vertices of~$\Vs_t$ according to the following rules.
	\begin{itemize}
		\item Each leaf $\ell$ gets colour $c_\ell$.
		\item If $x,y\in\Vs_t$ with $y\vdash x$ and $y$ has vertical in- and outdegree~$0$, or if $y$ is of the form $(u,s)$ with $s\in D$, then $x$ gets the colour of $y$.
		\item If $y\in\Vs_t$ is of the form $(u,s)$ with $s\in I$ and outdegree $k-1\geq 1$, let $i$ be the number of vertices with colour $a$ among $y$ and its vertical neighbours. Then all the vertices in $\{x\in\Vs_t:\, y\vdash x\}$ get colour $a$ (resp. $A$) with probability $p_{i,k}$ (resp. $1-p_{i,k}$).  
	\end{itemize}
	Let $N_t\coloneqq \lvert L(G_t)\rvert$. For $k\in[N_t]_0$, let~$U(k)$ be a uniform random variable on the set of leaves-colourings of~$G_t$ that assign~$k$ leaves colour~$a$ and~$N_t-k$ leaves colour~$A$. Define~$\cR(G_t,p,k)$ to be the probability that all roots of~$G_t$ have colour~$a$ if~$G_t$ is coloured according to~$U(k)$ and~$p$.
\end{definition}
It is straightforward to see that the previous definition associate to each vertex in $\Vs_t$ a (random) colour.

\smallskip

In a next step we provide the precise definition of the ASG and the ASP on the basis of bc-DAGs. The ASG is constructed on the basis of a branching-coalescing system of (marked) particles with the following dynamic. Let $\beta\in \R_+^{m-1}$ and $\Lambda\in \cM_f^*([0,1])$. Each particle carries a label in $[0,1]$. Start at time $t=0$ with $n$ particles. Let $U_i$ be the label of the $i$-th particle, where $(U_i)_{i=1}^n$ are independent random variables that are uniformly distributed in $[0,1]$. If there are currently $n$ particles, then
\begin{enumerate}
	\item for every $\ell\in\,]m]$ at rate $n \beta_\ell$, mark one of the existing particles uniformly at random and generate $\ell-1$ new particles. Each new particle carries a new label that is independent from the other labels and uniformly distributed in $[0,1]$, 
	\item for every $k\in\, ]n]$ at rate $\binom{n}{k}\lambda_{n,k}$, eliminate $k$ particles uniformly at random and generate a new particle with a new label, which is independent of the previous labels and uniformly distributed in $[0,1]$. 
\end{enumerate}
By construction, a particle with label $u$ has an associated birth time $b_u$ and a death time $d_u$, see also Fig.~\ref{fig:ASG}. Let ${\mathcal U}$ be the (random) set of labels assigned from $t=0$ to $t=\infty$. Let $D\coloneqq\{D_i\}_{i\in \N}$ (resp. $I\coloneqq \{I_i\}_{i\in \N}$)  be the set of times at which the number of particle decreases (resp. increases). In addition, we consider the function $w:I\to\Us$ that assigns to any time $I_i$ the label of the particle that has a mark at that time. Clearly, the sets $\Us$ and $\{b_u,d_u\}_{u\in\Us}$ together with the function $w$ induce a (random) bc-DAG~$\Gs$, which we  refer to as the $(\beta,\Lambda)$-bc-DAG with $n$ roots. The following is merely a more precise version of Definition~\ref{def:asg}.
\begin{definition}[ASG]\label{ASGprecise}
	Let $(\beta,p)\in \R_+^{m-1}\times\Ps_m$, $\Lambda\in\Ms_f^*([0,1])$, and $n\in\Nb$. The $(\beta,p,\Lambda)$-ASG starting with $n$ roots is the pair $(\Gs,p)$, where $\Gs$ is the $(\beta,\Lambda)$-bc-DAG with $n$ roots. Similarly, the $(\beta,p,\Lambda)$-ASG in $[0,t]$ starting with $n$ roots is the pair $(\Gs_t,p)$, where $\cG_t$ is the graph induced on the set of vertices with time coordinate less than $t$. In line with Definition~\ref{def:asg}, we write $L_t$ for the number of leaves in $\cG_t$.
\end{definition}
It follows from the definition of $\Gs$ that $(L_t:\, t\geq0)$ evolves according to the rates stated after Definition~\ref{def:asg}. Note also that by construction a coalescence event always leads to a new particle. This is consistent with the forward-in-time perspective, where the parent is never replaced by its offspring. Using the framework of bc-DAGs, the following formalises the ancestral selection polynomial of Definition \ref{def:asp}.
\begin{definition}[Ancestral selection polynomial] \label{def-asp}
	Let $(\Gs,p)$ be the $(\beta,p,\Lambda)$-ASG starting with $n$ roots. Denote by $\Gs_t$ the graph induced by the vertices of $\Gs$ with time coordinate less that $t$. Set
	\begin{equation}\label{eq:v_t-as-ev}
	V_t(i)\coloneqq\E[\Rs(\Gs_t,p,i)\mid (L_s:\, s\in [0,t])],\qquad i\in  [L_t]_0.
	\end{equation}
	The $(\beta,p,\Lambda)$-\emph{ancestral selection polynomial} (ASP) at time~$t$ is the random polynomial $P_t$ defined via
	\[P_t(x)=\sum_{i=0}^{L_t} V_t(i) b_{i,L_t}(x),\quad x\in[0,1].\]
\end{definition}
Note that $\Rs(\Gs_t,p,i)$ is $\sigma(\Gs_t)$-measurable and $\sigma(L_r:\,r\in[0,t]) \subset\sigma(\Gs_t)$ so that we recover Definition~\ref{def:asp} from \eqref{eq:v_t-as-ev} by using the tower property. 
\subsection{Connection between the ASP and the Bernstein coefficient process}\label{s4.1}
In this section we prove Proposition~\ref{ASPvsVCP}, which states that the Bernstein coefficient process and the process that keeps track of the Bernstein coefficient vector of the ancestral selection polynomial are equal in distribution.

\begin{proof}[Proof of Proposition~\ref{ASPvsVCP}] 
	Fix~$n\in \N$. Let~$V^\star\coloneqq(V_t^\star:\, t\geq0)$ be the Bernstein coefficient process with initial condition $V_0^\star = e_{n+1}$ so that $\langle  B_{L_0^\star}(x), V_0^\star \rangle=x^n$, where $L_t^\star \coloneqq \dim(V_t^\star)-1$. Let $V\coloneqq(V_t:\,t\geq0)$ be the Bernstein coefficients of the ASP associated with an ASG started with $n$ particles, see Eq.~\eqref{eq:v_t-as-ev}. Our aim is to prove that the two objects are identical in law. It follows from the transition rates of the ASG and~$L^\star$ that~$L\overset{(d)}{=}L^\star$. Thus, in order to prove that if $V_0=V_0^\star=e_{n+1}$, then $V \overset{(d)}{=}V^\star$, it is enough to show that 
	\begin{enumerate}
		\item[(i)]
		conditional on a positive jump of $L$ by $\ell-1$ at time $t$, corresponding to an $\ell$-branching event, we have $V_{t} = S^{L_{t-},\ell} V_{t-}$, where $S^{L_{t-},\ell}$ is the selection operator of Definition~\ref{def:SelCoagMatrices}.
		\item[(ii)] conditional on a negative jump of $L$ by $k-1$ at time $t$, corresponding to a coalescence event of $k$ leaves, we have $V_{t} =C^{L_{t-},k}V_{t-}$, where $C^{L_{t-},k}$ is the coalescence matrix of Definition~\ref{def:SelCoagMatrices}. 
	\end{enumerate}
	
	\medskip
	
	(i) \textit{Selection event.} Let $A(t,\ell)$ be the event of an $\ell$-branching at time $t$. In this case $L_{t} = L_{t-}+\ell-1$. For each $i \in [L_{t}]_0$, we need to determine $\E[\cR(\cG_{t},p,i) \mid \cG_{t-},\,A(\ell,t)]$. First note that conditional on $\cG_{t-}$ and $ A(\ell,t)$, the $\ell$-branching point, call it $u$, is chosen uniformly at random among the~$L_{t-}$ leaves at time $t-$. In other words, the $\ell$-branching point is grafted uniformly at random on a leaf available at time $t-$. To determine $\cR(\cG_{t},p,i)$, choose uniformly at random without replacement~$i$ leaves from the set $L_{t}$, and colour them with type~$a$. Starting from time~$t$, there is a subset of leaves~$\Gamma$~consisting of exactly~$\ell$ leaves that all collapse into a single leaf~$u$ at time~$t-$, see also Fig.~\ref{fig:graphs} (left). Let~$K_i$ be the number of type~$a$ in the subset of leaves~$\Gamma$ so that $K_i \sim \hypdist{L_{t}}{i}{ \ell}.$ According to the colouring rule~$p$, conditional on~$K_i$, the leaf~$u$ is of type~$a$ (resp. type~$A$) with probability~$p_{K_i,\ell}^{}$ (resp. $1-p_{K_i,\ell}^{}$). Furthermore, at time~$t-$ there are~$i-K_i$ leaves different from~$u$ that carry type~$a$. They are distributed uniformly at random among the remaining leaves. As argued above, since $u$ is chosen uniformly at random at time~$t-$,
	\[\E\big[\cR(\cG_{t},p,i) \mid \cG_{t-}, A(\ell,t)\big] \ = \ \E\big[ \, p_{K_i,\ell}^{}\, \cR(\cG_{t-},p,i+1-K_i) + (1-p_{K_i,\ell}^{}) \, \cR(\cG_{t-},p,i-K_i)  \mid \cG_{t-}\big],  \]
	and since $\sigma(L_{r}:\, r\in[0,t])\subset \sigma( \cG_{t-}, A(\ell,t))$,
	\begin{align*} \label{eq:tr1}
	\E\big[ \cR(\cG_{t},p,i) \mid (L_r;\, r\in[0,t-]), A(\ell,t) \big]  & =  \E\big[\, p_{K_i,\ell}^{} \, V_{t-}\left(i+1-K_i\right) +  (1- p_{K_i,\ell}^{})\,\ V_{t-}\left(i-K_i\right) \big] \\
	& =  (S^{{L_{t-},\ell}} V_{t-} )(i) ,
	\end{align*}
	which is the desired result.
	\smallskip 
	
	(ii) \textit{Coalescence event.} For $k\geq2$, consider a $k$-coalescence event and fix $i\leq  L_{t} = L_{t-}-k+1$. Choose~$i$ leaves at time~$t$ uniformly at random without replacement and colour them~$a$. Starting from time $t$, one leaf splits into~$k$ leaves at time~$t-$. Seen from~$t-$, this corresponds to $k$ leaves merging into one leaf at time~$t$, see Fig.~\ref{fig:graphs} (right). If this leaf at time~$t$ is of type $a$, which occurs with probability $i/ L_{t} = i/(L_{t-}-k+1)$, then there are~$i+k-1$ leaves of type $a$ at time $t-$. Otherwise, if this leaf at time~$t$ is of type $A$, there are~$i$ leaves of type~$a$ at time~$t-$. This translates into 
	\begin{equation*} 
	V_{t}(i)  =      \frac{i}{L_{t-}-k+1} V_{t-}({i+k-1}) \, + \,  \left(1- \frac{i}{L_{t-}-k+1}\right) V_{t-}(i)   =  (C^{L_{t-},k} V_{t-})(i).
	\end{equation*}
	The combination of (i) and (ii) yields that the transition rates of $V$ and $V^\star$ agree. 
\end{proof}

\begin{remark}\label{rem:leafanddimdonotagree}
It is plain from the definition that $L_t\geq \deg(P_t)$. It is tempting to conjecture that $L_t = \deg(P_t)$. However, it is easily seen to be false. For instance, consider the colouring rule $p_{0,3}=p_{2,3}=0$ and $p_{1,3}=p_{3,3}=1$ (i.e. minority rule) and branching rates $\beta=(\beta_2,\beta_3)=(0,1)$. Consider the $(\beta,p,\Lambda)$-ASG starting with $1$ root. In particular, $P_0(x)=x$ and $V_0=e_2=(0,1)^T$. Assume that there is a branching into three lines at time~$s$, followed by a coalescence of two lines at time~$t$ and no other transition ($s\leq t$). The effect of the branching and coalescence events is captured by the matrices
	\[S^{1,3}=\left(\begin{matrix}
	1 & 0\\
	0 & 1\\
	1 & 0\\
	0 & 1
	\end{matrix}\right)\quad \text{and} \quad C^{3,2}=\left( \begin{matrix}
	1 & 0 & 0 & 0 \\
	0 & \frac{1}{2} & \frac{1}{2} & 0 \\
	0 & 0 & 0 & 1
	\end{matrix}\right).\]
	A straightforward calculation yields $V_t=C^{3,2}(S^{1,3}e_2)=(0,\frac{1}{2},1)$. Furthermore, $L_0=1$, $L_{s}=3$, $L_{t}=2$, and $P_t(x)=\frac{1}{2}b_{1,2}(x)+b_{2,2}(x)=x$ so that $\deg(P_t)=1<L_t.$
\end{remark}
\begin{remark}[Topology on $\Rb^\infty$]\label{rem:topology}
	One is inclined to embed the state space $\Rb^\infty$ in the set of infinite sequences (adding infinite zeros at the end of every vector) equipped with the supremum norm. We are not using this embedding, because the transitions of the Bernstein coefficient process depend on the current dimension of the process, which does not necessarily coincide with the last non-zero entry of the process. A more appropriate metric is given as follows. For $u\in\Rb^n$ and $v\in \Rb^m$, with $n\leq m$, define
	\[d(u,v)\coloneqq d(v,u)\coloneqq\max_{i\in[n]_0}\lvert u_i-v_i\rvert + \max_{i\in[m]_0\setminus[n]_0}\lvert v_i\rvert +m-n.\]
	Note that the restriction of $d$ to $\Rb^n$ of coincides with the metric induced by the supremum norm, and that the distance between two vectors with different dimensions is at least $1$. Hence, a function $f:\Rb^\infty\to\Rb$ is continuous if and only if for any $n\in\Nb$, its restriction to $\Rb^n$ is continuous. In particular, the duality function in Theorem~\ref{thm:bernstein-duality} is continuous under this topology. Moreover, one can prove that $(\Rb^\infty,d)$ is a Polish space, providing a suitable setting for stochastic processes.
\end{remark}

\subsection{Bernstein duality (proofs)}\label{s4.2}
This section contains the proof Theorem~\ref{thm:bernstein-duality} establishing the Bernstein duality between the type-$a$ frequency process and the Bernstein coefficient process. We also prove Corollary~\ref{coro:bernstein-duality} connecting the Bernstein duality and the moment duality.

\smallskip

We start with a result stating that selection and coalescence matrices leave the first and last entries of a vector invariant.
\begin{lem}\label{lem:action-coal-frag}
	Let $n\in \N_0$ and $v\in \R^{n+1}$. Then for any $\ell\in]m]$ and $k\in ]n]$, $(S^{n,\ell}v)_0=v_0= 		(C^{n,k}v)_0,$ and $(S^{n,\ell}v)_{n+\ell-1} = 	v_n= (C^{n,k}v)_{n-k+1}$.
	Furthermore, for any $j\in[n+\ell-1]_0$ and $j'\in[n-k+1]_0$, we have $$\min_{i\in[n]_0} v_i\leq (C^{n,k}v)_{j'},\,(S^{n,\ell}v)_j\leq \max_{i\in[n]_0} v_i. $$ In particular,
	$\lVert S^{n,\ell}v\rVert_{\infty}\leq \lVert v\rVert_{\infty}$ and~$\lVert C^{n,\ell}v\rVert_{\infty}\leq \lVert v\rVert_{\infty}.$
\end{lem}
\begin{proof}
	This is plain from the definition of the selection and coalescence matrices.
\end{proof}
Next, we prove the Bernstein duality.

\begin{proof}[Proof of Theorem \ref{thm:bernstein-duality}]
	We want to prove that $X$ and $V$ are dual with respect to the duality function
	\[H:[0,1]\times \Rb^\infty\to\R_+,\qquad (x,v)\mapsto \langle B_{dim(v)-1}(x),v\rangle.\]
	First note that the generator $\mathscr{B}$ of $V$ can be expressed as 
	$$\mathscr{B}f(v)=\mathscr{B}_{\mathrm{s}}f(v)+\mathscr{B}_{\mathrm{wf}}f(v)+\mathscr{B}_\Lambda f(v), \quad v\in\Rb^\infty,\,  f\in\Cs(\Rb^\infty),$$
	where for $v$ with $\dim(v)=n+1$,
	\begin{align*}
	\mathscr{B}_{\mathrm{s}}f(v)&\coloneqq\sum_{\ell=2}^m n\beta_\ell \left(f(S^{n,\ell}v)-f(v)\right),\quad \mathscr{B}_{\mathrm{wf}}f(v)\coloneqq\frac{\Lambda(\{0\})}{2}n(n-1)\left(f(C^{n,2}v)-f(v)\right), \\
	\mathscr{B}_\Lambda f(v)&\coloneqq\sum_{k=2}^n \binom{n}{k}\lambda_{n,k}^0\left(f(C^{n,k}v)-f(v)\right).
	\end{align*}
	In addition, for any $v\in\Rb^\infty$, the function $x\mapsto H(x,v)$ is $\Cs^\infty([0,1])$, and therefore belongs to the domain of the generator $A$ of $X$ given in~\eqref{eq:generatorX}. Similarly, for any $x\in[0,1]$, $v\mapsto H(x,v)$ is continuous (with respect to the metric defined in Remark~\ref{rem:topology}), and hence belongs to the domain of $B$.
	Now we proceed to show that for $v\in\R^\infty$ and $x\in [0,1]$,
	\begin{equation} 
	AH(\cdot,v)(x)=\mathscr{B}H(x,\cdot)(v)\label{eq:dualitygeneratoridentity}.\end{equation}
	For this we prove the intermediate identities $A_\kappa H(\cdot,v)(x)=\mathscr{B}_\kappa H(x,\cdot)(v)$, for $\kappa\in\{\mathrm{s},\mathrm{wf},\Lambda\}$. Let $(Y_\ell^x)_{\ell\geq 0},\,(W_{\ell}^x)_{\ell\geq 0}$, $(K_{\ell,i}^n)_{0\leq i\leq\ell<n}$ be sequences of independent random variables with $Y_\ell^x,W_\ell^x\sim\bindist{\ell}{x}$ and $K_{\ell,i}^n\sim\hypdist{n+\ell-1}{\ell}{i}$.
	For any $v=(v_i)_{i=0}^n\in \R^{n+1}$,
	\[\frac{\partial H}{\partial x} (x,v)=n \E[v_{{Y_{n-1}^x}+1}^{}-v_{Y_{n-1}^x}^{}] \quad \text{and}\quad \frac{\partial^2 H}{\partial x^2}(x,v)=n(n-1) \E[v_{Y_{n-2}^x+2}^{}-2v_{Y_{n-2}^x+1}^{}+ v_{Y_{n-2}^x}^{}].\]
	Note also that 
	$$d(x)=d_{\beta,p}(x)=\sum_{\ell=2}^m\beta_\ell(\E\big[\,p_{W_\ell^x,\ell}^{}\big]-x)\quad\textrm{and}\quad H(x,v)=\E[v_{Y_n^x}^{}]=\E[(1-x)v_{Y_{n-1}^x}^{}+x v_{Y_{n-1}^x +1}^{}].$$ This together with the identity in law $(W_\ell^x,Y_{n-1}^x)\overset{(d)}{=}(K_{\ell,{Y_{n+\ell-1}^x}}^n,Y_{n+\ell-1}^x-K_{\ell,{Y_{n+\ell-1}^x}}^n)$
 yields
	\begin{align*}
	A_{\mathrm{s}}H(\cdot,v)(x)&=\left[\sum_{\ell=2}^m  \beta_\ell (\E\big[\,p_{W_\ell^x,\ell}^{}\big]-x )\right] \,n\,\E\big[v_{Y_{n-1}^x+1}^{}-v_{Y_{n-1}^x}^{}\big]\\
	&=\sum_{\ell=2}^m n\beta_\ell \Big( \E\big[\, p_{W_\ell^x,\ell}^{}\, v_{Y_{n-1}^x+1}+ (1-p_{W_\ell^x,\ell})\,v_{Y_{n-1}^x}^{}\big]- H(x,v)\Big) \\
	&=\sum_{\ell=2}^m n\beta_\ell \Big( \sum_{j=0}^{n+\ell-1} \E\big[\,p_{K_{\ell,j}^n,\ell}^{}\, v_{j-K_{\ell,j}^n+1}^{}+(1-p_{K_{\ell,j}^n,\ell}^{})\, v_{j-K_{\ell,j}^n}^{}\big] b_{j,{n+\ell-1}}(x)  -H(x,v)\Big)\\
	&=\mathscr{B}_{\mathrm{s}}H(x,\cdot)(v).
	\end{align*}
	This proves the identity for the part that corresponds to selection. For the Wright--Fisher part we have
	\begin{align*}
	&A_{\textrm{wf}}H(\cdot,v)(x)\\
	&=\frac{\Lambda(\{0\})}{2}n\sum_{i=1}^{n}{v}_i b_{i-1,n-1}(x) \big(i-1-x(n-1)\big)+n\sum_{i=0}^{n-1} v_i b_{i,n-1}(x) \big(n-i-1-(1-x)(n-1)\big)\\
	&=\frac{\Lambda(\{0\})}{2}n(n-1)\Bigg(\sum_{i=0}^{n-1} \Big( \frac{i}{n-1}v_{i+1}+\frac{n-1-i}{n-1} v_i\Big) b_{i,n-1}(x)-H(x,v)\Bigg)\\
	&= \mathscr{B}_{\mathrm{wf}}H(x,\cdot)(v).
	\end{align*}
	For the $\Lambda$-part, it is convenient to rewrite \[x+r(1-x)=r+x(1-r),\, 1-(x+r(1-x))=(1-x)(1-r),\, 1-x+rx=r+(1-x)(1-r),\]
	and use this in a straightforward calculation to obtain
	\begin{align*} 
	&A_{\Lambda}H(\cdot,v)(x)\\
	&=\int_{(0,1]}x\langle  B_n\big(r+(1-r)x\big),v\rangle + (1-x)\langle  B_n\big(x(1-r)\big),v\rangle - \langle  B_n(x),v\rangle\frac{\Lambda(\dd r)}{r^2}\\
	&=\sum_{k=2}^n \binom{n}{k}\lambda_{n,k}^0 \Bigg(\sum_{i=1}^{n-k+1} v_{i+k-1}\frac{i}{n-k+1} b_{i,n-k+1}(x) +\sum_{i=0}^{n-k}v_i  \frac{n-k+1-i}{n-k+1} b_{i,n-k+1}(x) -H(x,v) \Bigg)\\
	&= \mathscr{B}_{\Lambda}H(x,\cdot)(v),
	\end{align*}%
	which ends the proof of \eqref{eq:dualitygeneratoridentity}. Now, assume that the process $V$ starts at $V_0=v$. Using the definition of~$H$ and Lemma \ref{lem:action-coal-frag}, we obtain
	\begin{equation*}
	\sup_{x\in[0,1],s\in[0,T]} \lvert H(x,V_s)\rvert\leq \lVert v\rVert_\infty.
	\end{equation*}
	Similarly, 
	\begin{equation*}
	\sup_{x\in[0,1],s\in[0,T]}\lvert B_s H(x,\cdot)(V_s)\rvert\leq 2\lVert \beta\rVert_1 \lVert v\rVert_\infty \sup_{s\in[0,T]}L_s,
	\end{equation*}
	where $\lVert \beta\rVert_1 \coloneqq\sum_{\ell=2}^m\lvert \beta_\ell\rvert $. Moreover, using \citep[Lem. 3.3]{herriger2012conditions}, we deduce that
	\begin{equation*}
	\sup_{x\in[0,1],s\in[0,T]}\lvert (B_\Lambda +B_{\mathrm{wf}})H(x,\cdot)(V_s)\rvert\leq 2\lVert v\rVert_\infty \Lambda([0,1]) \left(\sup_{s\in[0,T]}L_s\right)^2,
	\end{equation*}
	Finally, note that one can couple $V$ to a pure birth process $\Gamma=(\Gamma_t:\,t\geq 0)$ such that: (a) $\Gamma_0=\dim(V_0)-1$, (b) each particle in~$\Gamma$ splits into $m$ particles at rate $\lVert\beta\rVert_1$, and (c) $\sup_{s\in[0,T]}L_s\leq \Gamma_T$. Since $\Gamma_T$ and $\Gamma_T^2$ are integrable, the result follows by~\citet[Cor.~4.4.13]{Ku86}.	
\end{proof}

Finally, we prove the result that connects the Bernstein duality to the classic moment duality.
\begin{proof}[Proof of Corollary~\ref{coro:bernstein-duality}]
	The proof of (1) follows directly from the Bernstein duality and the fact that $\langle B_n(x),e_{n+1}\rangle=x^n$. For (2), assume $(s_\ell)_{\ell=0}^{m-2}$ is a decreasing sequence in~$\R_+$. Recall that 
	\begin{equation*}
		\beta_m\coloneqq s_{m-2},\ \ \forall\ell\in \, ]m-1],\  \beta_\ell\coloneqq s_{\ell-2}-s_{\ell-1}\geq 0,\ \ \text{and} \ \ \forall \ell\in \, ]m], i\in[\ell]_0,\  p_{i,\ell}=\1_{\{i=\ell\}}.
	\end{equation*}
	A direct calculation yields $d_{(\beta,p)}(x)=-x(1-x)\sum_{i=0}^{m-2} s_ix^i\eqqcolon d(x),$
	which implies $(\beta,p)\in\ccS_d$. Another straightforward calculation yields $C^{n,k}e_{n+1}=e_{n-k+2}$ and for our choice of~$p$, $S^{n,\ell}e_{n+1}=e_{n+\ell}.$ In particular, if $V_0=e_{n+1}$, then for all $t\geq 0$, we have $V_t=e_{L_t+1}$ and $\langle B_{L_t}(x),V_t\rangle=x^{L_t}$. Hence, the Bernstein duality yields $$\E_x[X_t^n]=\E_{x}[\langle B_{n}(X_t),e_{n+1}\rangle]=\E_{e_{n+1}}[\langle B_{L_t}(x),V_t\rangle]=\E_{n}[x^{L_t}],$$ which proves the result.
\end{proof}

%%%%%%%%%%%%%%%%%%%%%%%%%%%%%%%%%%%%%%%%%%%%%%%%%%%%%%%%%%%%%%%%%%%%%%%%%%%%%%%%%%%%%%%%%%%%%%%%%%%%%%%%%%%%%%%%%%%%%%%%%%%%%%%%%%%%%%%%%%%%%%%%%%%%%%%%%%%%%
%%%%%%%%%%%%%%%%%%%%%%%%%%%%%%%%%%%%%%%%%%%%%%%%%%%%%%%%%%%%%%%%Section5%%%%%%%%%%%%%%%%%%%%%%%%%%%%%%%%%%%%%%%%%%%%%%%%%%%%%%%%%%%%%%%%%%%%%%%%%%%%%%%%%%%%%
%%%%%%%%%%%%%%%%%%%%%%%%%%%%%%%%%%%%%%%%%%%%%%%%%%%%%%%%%%%%%%%%%%%%%%%%%%%%%%%%%%%%%%%%%%%%%%%%%%%%%%%%%%%%%%%%%%%%%%%%%%%%%%%%%%%%%%%%%%%%%%%%%%%%%%%%%%%%%

\section{Properties of the Bernstein coefficient process and its leaf process}\label{s5}

In this section we study properties of the Bernstein coefficient process and its leaf process. The section begins with the proof of the condition for transience and recurrence of the leaf process. This allows us to derive and to characterise the invariant measures for~$L$ and $V$. In particular, we obtain a recursion for the tail probabilities of the stationary measure of the leaf process. Finally, we study the property of the leaf process to come down from infinity. 

\subsection{Recurrence and transience of leaf process}\label{s5.1}
The leaf process~$L$ of Definition~\ref{def:asg} with parameters~$(\beta,\Lambda)$ takes values in~$\N$ and has infinitesimal generator
\begin{equation}
\Ls f(n)\coloneqq \Ls_{\beta} f(n)+\Ls_{\Lambda} f(n),\quad f:\N\to\R, \label{eq:generatordegreeprocess}
\end{equation}
where \[\Ls_{\beta} f(n)\coloneqq\sum_{\ell=2}^m n\beta_\ell [f(n+\ell-1)-f(n)],\qquad \Ls_{\Lambda} f(n)\coloneqq\sum_{k=2}^n \binom{n}{k}\lambda_{n,k} [f(n-k+1)-f(n)].\] 

We want to prove Theorem~\ref{thm:degreeprocessrecurrenttransient(intro)}, which provides conditions for positive recurrence and transience of the leaf process. Let us stress again that the first part of Theorem~\ref{thm:degreeprocessrecurrenttransient(intro)} is already present in~\citet[Thm.~4.6]{gonzalezcasanova2018} in the case $c(\Lambda)<\infty$. The latter result also states that $L$ is not positive recurrent if $b(\beta)> c(\Lambda)$. We use here a different approach that follows the lines of \citet{foucart2013impact}. It allows us to show the transience of $L$ for $b(\beta)> c(\Lambda)$, and its positive recurrence for $b(\beta)<c(\Lambda)$. 
\smallskip

Define $\delta:\N\to\R$ and $f:\N \to \R$ via
\begin{align*}
 \delta(n)&\coloneqq \binom{n}{2}\Lambda(\{0\})-n\int_{(0,1]}\log\left(1-\frac{1}{n}(nr-1+(1-r)^n)\right)\frac{\Lambda(\dd r)}{r^2},\\
f(\ell)&\coloneqq\sum_{k=2}^\ell \frac{k}{\delta(k)}\log\Big( \frac{k}{k-1}\Big).
\end{align*}
The function $\delta$ was introduced in a more general setting in \cite{herriger2012conditions} to establish conditions for $\Xi$-coalescents to come down from infinity. The function $f$ is used in \cite{foucart2013impact} as an analytical tool to prove positive recurrence of the line-counting process of the ASG associated to the $\Lambda$-Wright--Fisher process with genic selection. It will be convenient to collect some known properties of these functions.
\begin{lem}[{\citet[Cor.~4.2, Eq.~(4.2)]{herriger2012conditions}}]\label{lem:propertiesspecialfct}
	\hspace{.1mm}
	\begin{itemize}
		\item The functions $n\mapsto\delta(n)$ and $n\mapsto\delta(n)/n$ are non-decreasing.
		\item $\delta(n)/n\to c(\Lambda)$ for $n\to\infty$.
	\end{itemize}
\end{lem}

\begin{lem}\label{lem:cond-cdi-coalescent}
	The function $f$ is non-negative and increasing. Further, the $\Lambda$-coalescent c.d.i.~if and only if $\lim_{n\to\infty} f(n)<\infty$.
\end{lem}
\begin{proof}
	The first part of the statement follows from the fact that $\delta(n)>0$ for $n>1$ (this is easy to see once one observes that this is true for $n=2$ and $\delta(n)$ is non-decreasing by Lemma~\ref{lem:propertiesspecialfct}). \cite{herriger2012conditions} shows that the $\Lambda$-coalescent c.d.i.~if and only if $\sum_{k\geq2} 1/\delta(k)<\infty$. Since $k\log(k/(k-1))/\delta(k) \sim 1/\delta(k)$ as $k\to \infty$, this completes the proof of the second part of the proposition.
\end{proof}

The next lemma is a generalisation of \citet[Lem.~2.3]{foucart2013impact}, which corresponds to the case $\beta_2>0$, $\beta_\ell=0$ for $\ell\neq 2$, and $p_{1,2}=0$.
\begin{lem}\label{lem:generatorinequality}
	We have 
	\begin{equation}
	\Ls f(n)\leq -1+\sum_{\ell=2}^m\beta_\ell \sum_{j=n+1}^{n+\ell-1}\frac{j}{\delta(j)}.\label{eq:lemgeneratorinequality}
	\end{equation}
	Furthermore, if $b(\beta)<c(\Lambda)$, then there exists $n_0\in \N$ and $\varepsilon>0$ such that $$\Ls f(n)\leq -1+ \frac{b(\beta)}{c(\Lambda)}+\varepsilon b(\beta)<0,\qquad \forall n\geq n_0,$$
	with the usual convention that $1/\infty=0$.
\end{lem}
\begin{proof}
	\citet[Proof Lem.~2.3]{foucart2013impact} proves that $\Ls_{\Lambda}f(n)\leq -1.$ Hence, for the first claim it suffices to prove that $\Ls_{\beta} f(n)\leq \sum_{\ell=2}^m\beta_\ell\sum_{j=n+1}^{n+\ell-1} j/\delta(j)$. Note that
	\begin{align*}
	n[f(n+\ell-1)-f(n)]=n\sum_{j=n+1}^{n+\ell-1} \frac{j}{\delta(j)}  \log\Big(1+\frac{1}{j-1}\Big)\leq\sum_{j=n+1}^{n+\ell-1} \frac{j}{\delta(j)}  \log\Big(\Big(1+\frac{1}{n}\Big)^n\Big)\leq \sum_{j=n+1}^{n+\ell-1} \frac{j}{\delta(j)},
	\end{align*}
	where we use that $(1+1/n)^n$ is monotonically increasing to Euler's number. The first claim follows in a straightforward way. For the second claim, note that Lemma~\ref{lem:propertiesspecialfct} implies that $n/\delta(n)$ is non-increasing and $n/\delta(n)\to 1/c(\Lambda)$ as $n\to\infty$. In particular, for all $\varepsilon>0$, there exists~$n_0\in \N$ such that for all $n\geq n_0$, $$\frac{n}{\delta(n)}\leq \frac{1}{c(\Lambda)}+\varepsilon.$$ For $n\geq n_0$, we can now estimate the right-hand side of~\eqref{eq:lemgeneratorinequality} by $$-1+\sum_{\ell=2}^m\beta_\ell \sum_{j=n+1}^{n+\ell-1}\frac{j}{\delta(j)}\leq -1+ \frac{b(\beta)}{c(\Lambda)}+b(\beta)\varepsilon,$$
	which is negative for $b(\beta)<c(\Lambda)$ and $\varepsilon>0$ small enough.
\end{proof}
The following lemma gives a condition for the positive recurrence of the leaf process. In the case of genic selection, this result again agrees with \citet[Lem.~2.4]{foucart2013impact}. Define for $n\in \N$,
\[T^{(n)}\coloneqq\inf \{s\geq 0: L_s<n\}\quad \text{and} \quad \mathcal{T}_n\coloneqq\inf \{s\geq J_1: L_s=n\},\]
where $J_1$ is the time of the first jump of~$L$.

\begin{lem}\label{lem:FiniteExpectedReturnTime}
	Assume $b(\beta)<c(\Lambda).$  Then there exists $n_0$ and a constant~$\tilde{c}$ such that for all $n\geq n_0$, \[\E_n\big[T^{(n_0)}\big]\ <\ \tilde{c} f(n).\] 
\end{lem}
\begin{remark}
	In the statement analogous to Lemma~\ref{lem:FiniteExpectedReturnTime}, \citet{foucart2013impact} has the additional condition $\sum_{k=2}^{\infty}1/\delta(k)=\infty$ to assure that the process is non-explosive. As already noted in~\citep[p.~4]{BLW16} one can easily get rid of this condition.
\end{remark}
\begin{proof}[Proof of Lemma~\ref{lem:FiniteExpectedReturnTime}]
	We mimic the proof of~\citet[Lem.~2.4]{foucart2013impact}. Note that $L$ is dominated by a process that has only~$m$-branchings at rate~$\sum_{\ell=2}^m\beta_\ell$ per existing line and no coalescences. Clearly, this process is non-explosive. Hence, also $L$ is non-explosive. Next, define for $N\in \N$, $f_N(n)\coloneqq f(n)\1_{\{n\leq N+m\}}.$ By Dynkin's formula, $$\Big(f_N(L_t)-\int_0^t \Ls f_N(L_s)\dd s:\,t\geq 0\Big)$$ is a martingale. By Lemma \ref{lem:generatorinequality}, there is~$n_0\in \N$ and $\varepsilon>0$ such that for all $n\geq n_0$, 
	\begin{equation}\label{gen-neg}
	\Ls f(n)\leq -1+b(\beta)/c(\Lambda)+\varepsilon b(\beta) <0.
	\end{equation}
	Set $C\coloneqq1-b(\beta)/c(\Lambda)-\varepsilon b(\beta)$ and let $n_0\leq n \leq N$. Define $S_N\coloneqq\inf\{s\geq 0: L_s\geq N+1\}$. Applying the optional stopping theorem yields \[\E_n\big[f_N(L_{T^{(n_0)}\wedge S_N\wedge k})\big]=f_N(n)+\E_n\left[\int_0^{T^{(n_0)}\wedge S_N\wedge k}\Ls f_N(L_s)\dd s\right].\]
	Clearly, $\Ls f_N(n)=\Ls f(n)$ for $n\leq N$. Thus, using \eqref{gen-neg}, we obtain \begin{align*}
	\E_n\big[f_N(L_{{T^{({n_0})}}\wedge S_N\wedge k})\big]\leq f_N(n)-C\, \E_n\big[ T^{(n_0)}\wedge S_N\wedge k\big].
	\end{align*} Hence, $$C\,\E_n\big[T^{({n_0})}\wedge S_N\wedge k \big]\leq f_N(n)-\E_n\big[f_N(L_{T^{(n_0)}\wedge S_N\wedge k})\big]\leq f_N(n),$$
	where the last inequality is a consequence of the fact that $f_N(n)\geq 0$. Since $L$ is non-explosive, $S_N\to\infty$ almost surely as $N\to\infty$ so that \[C\,\E_n\big[\,T^{(n_0)}\wedge k\,\big]\leq f(n).\] Letting $k\to\infty$ yields the result.
\end{proof}
\begin{corollary}\label{coro:positiverecurrent}
	Assume $b(\beta)<c(\Lambda)$.
	\begin{itemize}
		\item If $\Lambda\neq \delta_1$, then $L$ is positive recurrent.
		\item If $\Lambda=\delta_1$, then $n$ is positive recurrent for~$L$ if and only if there exists $k\in \N$ such that $n=n_1+\ldots+n_k$ for $n_\ell\in \{i\in \N:\beta_{i+1}> 0\}$, $\ell\in [k]$.
	\end{itemize}
	In particular, if $\beta_2>0$, then~$L$ is positive recurrent.
\end{corollary}
\begin{proof}
	Let~$L$ be the leaf process with parameter $(\beta,\Lambda)$ with $\Lambda\neq \delta_1$ and let $q(n,j)$ be the corresponding transition rates. Consider $n_0$ from Lemma~\ref{lem:FiniteExpectedReturnTime}. Define another Markov chain~$\hat{L}$ with transition rates 
	\[\hat{q}(n,j)\coloneqq\begin{cases} 0, &\text{if}\quad n<n_0,\, j\in \{n_0,\ldots,n_0+m-1\},\\
	\sum_{k=n_0}^{n_0+m-1}q(n,k),& \text{if}\quad n<n_0,\, j=n_0+m,\\
	q(n,j), &\text{otherwise.}
	\end{cases}\]
	Define $\hat{T}^{(n)}$ and $\hat{\mathcal{T}}_n$ for $\hat{L}$ as the analogue of $T^{(n)}$ and $\mathcal{T}_n$ for $L$. Since the transition rates of $L$ and $\hat{L}$ agree for $n\geq n_0$, we have $\E_n[\hat{T}^{(n_0)}]=\E_n[T^{(n_0)}]<\tilde{c}f(n)$. Let $T(n_0)\coloneqq\inf\{t\geq 0: L_t\geq n_0\}$. Analogously, define $\hat{T}(n_0)$ for $\hat{L}$. Clearly, $\E_k[T(n_0)]=\E_k[\hat{T}(n_0)]<\infty$ for all $k<n_0$. Hence,\[\E_{n_0+m}[\hat{\mathcal{T}}_{n_0+m}]\leq\E_{n_0+m}[\hat{T}^{(n_0)}]+\sum_{k=1}^{n_0-1}\E_k[\hat{T}(n_0)]<\infty,\] and so $n_0+m$ is positive recurrent for $\hat{L}$. By irreducibility (since $\Lambda\neq \delta_1$), $1$ is positive recurrent for~$\hat{L}$.
	
	Assume that $L_0=\hat{L}_0=n$ for some $n>n_0$. Clearly, there exists a coupling such that $L_t\leq \hat{L}_t$ for all $t\geq 0$. In particular, since state~$1$ is positive recurrent for~$\hat{L}$, $1$ is also positive recurrent for~$L$. By irreducibility, all states of~$L$ are positive recurrent. If $\Lambda=\delta_1$, then $\E_n[{\mathcal{T}}_1]<\infty$. Hence,~$1$ is positive recurrent for~$L$. The result follows since the communication class of~$1$ consists of all~$n$ for which there exists $k\in \N$ such that $n=n_1+\ldots+n_k$ for some $n_\ell\in \{i\in \N:\beta_{i+1}> 0\}$, $\ell\in [k]$.
\end{proof}

The next results agrees with~\citet[Lem.2.5]{foucart2013impact} in the case of genic selection.
\begin{lem}\label{lem:degreeprocesstransient}
	If $b(\beta)>c(\Lambda)$, then $L$ is transient.
\end{lem}
\begin{proof}
	Again, we mimic the proof of~\citet[Lem.~0.1]{foucart2014}. Assume that there is~$n_0\in \N$ and a bounded strictly decreasing function~$g$ that is chosen such that $\Ls g(n)<0$ for all $n\geq n_0$. The process $(g(L_{t\wedge T^{(n_0)}}):\,t\geq 0)$ started from $n\geq n_0$ is a supermartingale. By the martingale convergence theorem, $\lim_{t\to\infty} \E_n[g(L_{t\wedge T^{(n_0)}})]\leq g(n)<g(n_0)$ and so $\P_n(T^{(n_0)}<\infty)<1$. Decompose \[\P_{n_0-1} (T(n_0-1)<\infty)=\sum_{n< n_0-1}\P_{n} (T(n_0-1)<\infty)\P(L_{J_1}=n)
	+ \sum_{n\geq n_0}^{\infty}\P_{n} (T(n_0-1)<\infty)\P(L_{J_1}=n).\] Since for $n\geq n_0$, $\P_{n}(T(n_0-1)<\infty)<\P_{n}(T^{(n_0)}<\infty)<1$, it follows that $\P_{n_0-1}(T(n_0-1)<\infty)<1$, and so $L$ is transient~\citep[Thm.~4.3.2]{norris1998markov}. As we will show, the conditions are indeed satisfied for the function $g(n)\coloneqq 1/\log(n+1).$
	It is proven in~\citet[p.2]{foucart2014} that $$\Ls_{\Lambda} g(n)=\frac{1}{\log(n+1)\log(n+2)}\big(c(\Lambda)+o(1)\big).$$
	Our claim is that $\Ls_{\beta}g(n)\leq \frac{1}{\log(n+1)\log(n+2)}(-b(\beta) + o(1))$. It then follows together with the just mentioned result of~\citet{foucart2014} that \[\Ls g(n)= \frac{1}{\log(n+2)\log(n+1)}\big(c(\Lambda)-b(\beta)+o(1) \big)<0,\] for $n$ large enough. It remains to prove the claim.
	Note that, \begin{align*}
	\Ls_{\beta} g(n)&=\sum_{\ell=2}^m \beta_\ell \frac{1}{\log(n+\ell)\log(n+1)} n\log\Big(\frac{n+1}{n+\ell}\Big)\\
	&=-\sum_{\ell=2}^m\beta_\ell \frac{1}{\log(n+\ell)\log(n+1)} \big((\ell-1)+o(1)\big)\\
	&\leq -\frac{1}{\log(n+m)\log(n+1)} \big(b(\beta)+o(1)\big),
	\end{align*}
	where in the second equality, we use that $n\log((n+1)/(n+\ell))=-((\ell-1)+o(1)).$
	Finally, since $\log(n+2)/\log(n+m)=1-o(1),$ we have\[-\frac{1}{\log(n+2)\log(n+1)} \frac{\log(n+2)}{\log(n+m)}\big(b(\beta)+o(1)\big)\leq \frac{1}{\log(n+2)\log(n+1)} \big(-b(\beta)+o(1)\big),\] which proves the claim.
\end{proof}
\begin{proof}[Proof of~Theorem~\ref{thm:degreeprocessrecurrenttransient(intro)}]
	The result follows from Corollary~\ref{coro:positiverecurrent} and Lemma~\ref{lem:degreeprocesstransient}. 
\end{proof}

\subsection{Siegmund duality and Fearnhead-type recursions for the leaf process}\label{s5.2}
As mentioned in Section \ref{s2.4} the leaf process $L$ is stochastically monotone, i.e. $n\mapsto\Pb_n(L_t\geq \ell)$ is non-decreasing, see Remark~\ref{SM} below. It is well-known that this implies that $L$ admits a so-called \emph{Siegmund dual} process \cite{siegmund1976}. In this section we spell out this duality in our setting and use it as an analytical tool to derive recursions for the stationary tail-probabilities of~$L$ (in the positive recurrent case). We subsequently use these recursions to deduce that the leaf process has exponential moments if $\Lambda(\{0\})>0$.
\smallskip

Let $(\beta,\Lambda)\in \R_+^{m-1}\times\Ms_f^*([0,1])$ and consider the process $D\coloneqq(D_t:\,t\geq 0)$ on $\N\cup\{\infty\}$ with infinitesimal generator $Gg(d)\coloneqq G_{\Lambda}g(d)+G_{\beta}g(d)$, where \begin{align*}
G_{\Lambda}g(d)&\coloneqq\ \ \sum_{c\geq 1} \binom{d+c-1}{c+1}\lambda_{c+d,c+1}[g(c+d)-g(d)] + \1_{\{d\geq 2\}}\Lambda (\{1\}) [g(\infty)-g(d)],\\
G_{\beta}g(d) &\coloneqq\sum_{r=1}^{(m\wedge d)-1} \Big((d-r)\beta_{r+1} + \sum_{k=r+1}^{m-1} \beta_{k+1} \Big)[g(d-r)-g(d)],
\end{align*}
and which is absorbed the first time it reaches infinity. The latter can occur due to a jump, or an explosion of the process. The other absorbing state of~$D$ is $1$. It is well-known that in the neutral case and in the case of genic selection, the process~$D$ (shifted by $-1$) corresponds to the \emph{fixation line} \citep[Rem. 4.6]{BLW16}. Loosely speaking, the fixation line codes how a new most recent common ancestor establishes itself in the population, see \citep{achaz2018sequential, henard2015,gaiser2016,pfaffelhuber2006} for more details.

\begin{remark}[Stochastic monotonicity]\label{SM}
Consider the leaf processes $\bar{L}$ with $\bar{L}_0=\bar{n}$ and construct the process $L$ as follows. Set $L_0=n<\bar{n}$. Assume $\bar{L}_{t-}=\bar{\ell}\geq \ell=L_{t-}$. If $\bar{L}$ increases by~$j-1\in[m-1]$ at time~$t$, then set $L_t=\ell+j-1$ (resp. $L_t=\ell$) with probability $\ell/\bar{\ell}$ (resp. $1-\ell/\bar{\ell}$). If $\bar{L}$ decreases to~$\bar{\ell}-k+1$ at time~$t$ for some $k\in \, ]\bar{\ell}\,]$, then set
$L_t=\ell-K_{\ell,k}^{\bar{\ell}}+1$, where $K_{\ell,k}^{\bar{\ell}}\sim\hypdist{\bar{\ell}}{\ell}{k}$. By construction, $\bar{L}_t\geq L_t$ for all $t\geq 0$, and $L$ is a leaf process started at $n$ with the same parameters as~$\bar{L}$. This coupling shows that the leaf process is indeed stochastically monotone.
\end{remark} 

\begin{lemma}[Siegmund duality]\label{lem:siegmundDuality}
	The process $(D_t:\,t\geq 0)$ and $(L_t:\,t\geq 0)$ are Siegmund dual, i.e. \begin{equation}
	\P_\ell(d\leq L_t)=\P_d(D_t\leq \ell), \qquad  \forall\, \ell\in \N,\, d\in \N\cup\{\infty\},\, t\geq 0.
	\end{equation}
\end{lemma}
\begin{proof}
	For $d,\ell\in \N$, set $\hat{H}(d,\ell)=\1_{\{d\leq \ell\}}$. We will show that $\Ls \hat{H}(d,\cdot)(\ell) =G\hat{H}(\cdot,\ell)(d)$. The result then follows by~\citep[Prop.~1.2]{Jaku}. First, note that the result in the neutral case (see~ \citep[Eq. (4.1), Lem.~4.5]{BLW16}) implies that the processes generated by $\Ls_{\Lambda}$ and $G_{\Lambda}$ are dual with respect to~$\hat{H}$. Hence, $\Ls_{\Lambda}\hat{H}(d,\cdot)(\ell)=G_{\Lambda}\hat{H}(\cdot,\ell)(d)$. Thus, it remains to prove $\Ls_{\beta} \hat{H}(d,\cdot)(\ell) =G_{\beta}\hat{H}(\cdot,\ell)(d)$. Clearly, for $d\leq \ell$ we have $\Ls_{\beta} \hat{H}(d,\cdot)(\ell)=0 =G_{\beta}\hat{H}(\cdot,\ell)(d)$. In addition, for $d>\ell,$ a straightforward calculation yields
	\begin{align*}
	G_{\beta}\hat{H}(\cdot,\ell)(d)&-\Ls_{\beta}\hat{H}(d,\cdot)(\ell)\\
	=&\sum_{r=2}^{(d\wedge m)-1} (d-r)\beta_{r+1}\1_{\{d\leq \ell+r\}} + \sum_{r=2}^{m-1}\beta_{r+1}(\ell+d\wedge r-d)\1_{\{d<\ell+d\wedge r\}}-\ell\sum_{r=2}^{m-1}\beta_{r+1}\1_{\{d\leq \ell+r\}}. \end{align*}
	The result follows by inspecting the cases $m\leq d$ and $m>d$.
\end{proof}
For the remainder of this section, we assume~$b(\beta)<c(\lambda)$ so that~$L$ is positive recurrent (in the communication class of~$1$). Let $L_\infty$ be a random variable distributed according to the stationary distribution of $L$, and define for $n\in \Nb$, $a_n\coloneqq\Pb(L_\infty>n)$. By the Siegmund duality we deduce that $a_n=\Pb_{n+1}(D \textrm{ absorbs in } 1)$. \citep{BLW16} already exploited this relation to obtain a generalisation of the Fearnhead recursion~\citep{fearnhead2002common} for a $\Lambda$-Wright--Fisher process with genic selection. We further generalise this recursions to our setting. 

\begin{proposition}\label{prop:fearnhead}
	The tail-probabilities are the unique solution to the system of equations 
	\begin{equation}\label{eq:fearnhead-recursion}
	\sum_{c\geq 2} \binom{n+c-1}{c}\lambda_{c+n,c}[a_n-a_{c+n-1}] + \Lambda(\{1\})a_n= \sum_{r=1}^{(m-1)\wedge n} \! \! \! \! \Big((n+1-r)\beta_{r+1} + \sum_{k=r+1}^{m-1} \beta_{k+1} \Big) [a_{n-r}-a_{n}]
	\end{equation}
	with boundary conditions $a_0=1$ and $\lim_{n\to\infty} a_n=0$.
\end{proposition}

\begin{proof}
	Clearly, the boundary conditions hold because $a_n$ is a tail probability. Set $h(n)\coloneqq a_{n-1}$, $n\in\Nb$. Lemma~\ref{lem:siegmundDuality} implies that
	$h(n)=\Pb_{n}(D \textrm{ absorbs in } 1)$. Hence, $h$ is harmonic for $G$, i.e. $Gh(n)=0$, and Eq.~\eqref{eq:fearnhead-recursion} follows. For the uniqueness we follow the proof of~\citep[Thm. 2.4]{BLW16}. Denote by $a'=(a'_{n})_{n\geq 0}$ another solution of the recursion and set $g(n)\coloneqq a_{n-1}-a_{n-1}'$ for~$n\in\Nb$. Hence, $g(1)=0$, and $\lim_{n\to\infty}g(n)=0$. Since $n\mapsto a_{n-1}$ and $n\mapsto a'_{n-1}$ are both harmonic for~$G$, also~$g$ is harmonic for~$G$. In particular, $(g(D_t):\,t\geq 0)$ is a bounded martingale. Let $T_{1,k}\coloneqq\inf\{t\geq 0: D_t\in\{1,k,k+1,\ldots\}\}$. Then $T_{1,k}$ is almost surely finite for every $k\in\Nb$. If $D_0=d$, the optional stopping theorem yields $g(d)=\E_{d}[g(D_{T_{1,k}})]$ for all~$k\in\Nb$. But $g(D_{T_{1,k}})\to 0$ as $k\to\infty$. Hence, by the dominated convergence theorem, $g(d)=0$ for all~$d\in \N$, and so $a'=a$.  
\end{proof}
\begin{remark}
	In the case of genic selection, \citep{CM19} gives a general approach to solve \eqref{eq:fearnhead-recursion} and provides explicit solutions for the Kingman case, the star-shaped case, and the Bolthausen-Sznitman case.  
\end{remark}

\begin{corollary}\label{coro:exponentialMoments}
	If $\Lambda(\{0\})>0$, then $L_{\infty}$ has exponential moments of all orders, i.e. $\E[\exp(xL_{\infty})]<\infty$ for all $x\in \Rb$.
\end{corollary}
\begin{proof}
	Define $q_n\coloneqq\P(L_{\infty}=n)=a_{n-1}-a_n$. A straightforward manipulation of Eq.~\eqref{eq:fearnhead-recursion} yields that~$q_0=0$, $\sum_{n\geq 1} q_n=1$, and 
	\begin{equation}\label{eq:recursionProbabilitiesLeafprocess}
	\sum_{k\geq n+1} \! \! q_k \bigg(c_{n,k}+\frac{\Lambda(\{1\})}{n}\bigg) =\sum_{k=(n-m+2)\vee 1 }^{n} \! \! \! q_k b_{n,k},\qquad n\in\Nb,
	\end{equation}
	where \[c_{n,k}\coloneqq \frac{1}{n} \sum_{\ell\geq k} \binom{\ell}{\ell-n+1} \int_{[0,1)} r^{\ell-n-1} (1-r)^n \Lambda(\dd r), \ \ \ b_{n,k} \coloneqq \! \! \! \! \sum_{r=n-k+1}^{m-1}\! \! \frac{k\wedge(k-r+n+1)}{n} \beta_{r+1}.\] 
	First note that
	$$\sum_{k\geq n+1} \! \! q_k \bigg(c_{n,k}+\frac{\Lambda(\{1\})}{n}\bigg) \geq q_{n+1} c_{n,n+1}\geq q_{n+1}\,\frac{(n+1)}{2}\,\Lambda(\{0\}).$$
	To get an upper bound for the right-hand side of \eqref{eq:recursionProbabilitiesLeafprocess}, we set some notation. For $n\in\N$, define 
	$$\hat{q}_n\coloneqq \max \{q_k:\, k\in [n-1]\setminus [n-m]\}\quad\textrm{and}\quad r(n)\coloneqq\min\{k\in[n-1]\setminus [n-m]: q_k=\hat{q}_n\},$$ 
	with the convention $[k]=\emptyset$ for $k\leq 1$. Write $r^i$ for the $i$-th composition of $r$, i.e. $r^i(n)=r^{i-1}(r(n))$ and $r^0(n)=n$. By construction we have for $n>1$, $r(n)\leq n-1$, and hence $r^{n-1}(n)=1$. Set $$\hat{d}(n)\coloneqq\min \{i\in[n-1]:r^i(n)=1\}\leq n-1.$$ 
	Clearly,
	$$\sum_{k=(n-m+2)\vee 1 }^{n} \! \! \! q_k b_{n,k}\leq m\lVert \beta\rVert_1 q_{r(n+1)},$$
	where $\lVert \beta \rVert_1\coloneqq \sum_{\ell=2}^m \lvert \beta_\ell\rvert $. Therefore, we obtain
	\[q_{n}\, \leq \, \frac{1}{n}\, \frac{2 m\rvert \beta\lvert }{\Lambda(\{0\})} \, q_{r(n)} \, \leq \, \frac{1}{\prod_{i=0}^{\hat{d}(n)-1} r^i(n)}\left( \frac{2m\lVert \beta \rVert_1}{\Lambda(\{0\})}\right)^{\hat{d}(n)}\,q_1. \]
	We claim that $r^i(n)>\hat d(n)-i$ for all $i\in[\hat d(n)]_0$ and $n>1$. We proceed by induction. Since $\hat d(n)\leq n-1$, the claim is true for $i=0$. Now let us assume the result is true for $i-1$ and $n>1$ such that $i\in[\hat d(n)]_0$. Since $\hat d(r(n))=\hat d(n)-1$, we obtain
	$$r^i(n)=r^{i-1}(r(n))>\hat d(r(n))-i+1=\hat d(n)-i,$$
	which proves the claim. Consequently, we obtain that $\prod_{i=0}^{\hat{d}(n)-1} r^i(n)\geq \hat d(n)!$, and hence
	$$q_n\leq \frac{1}{\hat d(n)!}\left( \frac{2m\lVert \beta \rVert_1}{\Lambda(\{0\})}\right)^{\hat{d}(n)}\,q_1.$$
	Since $n-r(n)\leq m$, we deduce that $n/m-1\leq \hat{d}(n)\leq n-1$. The result follows.
	\end{proof}

\begin{remark}
	The system of equations~\eqref{eq:recursionProbabilitiesLeafprocess} complemented by $q_0=0$ and $\sum_n q_n=1$ is equivalent to the infinite system~$0=q \Ls$ in the sense that both characterise the stationary probabilities $(q_n)_{n\in\Nb_0}$. However, in many situations it is easier to deal with \eqref{eq:recursionProbabilitiesLeafprocess}. For the sake of illustration, let us consider the case where $\Lambda = \delta_0$ and $\beta = 0$, which corresponds to the Kingman coalescent. In this particular setting, the leaf process is absorbed at $1$ after a finite time so that $q_n=\1_{\{n=1\}}$. On the one hand, the typical condition $0=q \Ls$ yields $\frac{n(n+1)}{2} q_{n+1}-\frac{n(n-1)}{2} q_n=0$ for every $n\geq2$. On the other hand, \eqref{eq:recursionProbabilitiesLeafprocess} reads~$q_{n+1}=0$ for every~$n\in\Nb$, so it directly yields the solution to the infinite system.
\end{remark}

\subsection{Invariant measure of the Bernstein coefficient process}\label{subsec:bcpinvariant} In this section we prove Proposition~\ref{prop:abs}, which describes the long-term behaviour of the Bernstein coefficient process. 

\begin{proof}[Proof of Proposition~\ref{prop:abs}] It follows directly from Lemma~\ref{lem:action-coal-frag} that $V_{t}(0)$ and $V_{t}(L_t)$ are constant along time.
\smallskip

 (1) Let $\mathbb{M}_{V}$ be the set of matrices of size $n\times 2$ for some $n\geq 2$ that can be obtained as the product of a finite number of (compatible) selection and coalescence matrices (the $2\times 2$ identity matrix is seen as an empty product of such matrices). Fix $a,b\in \R$ and denote by 
		\[C_V(a,b)\coloneqq\left\{M \begin{pmatrix} a \\ b \end{pmatrix}\in \R^{\infty}: \ M\in \mathbb{M}_{V}\right\},\]
		the set of points that can be reached by $V$ starting from $(a,b)^T$. The set $C_V(a,b)$ is by definition invariant for $V$. Moreover, the restriction of $V$ to $C_V(a,b)$ is an irreducible continuous-time Markov chain. Indeed, for $w,w'\in C_V(a,b)$ to go from $w$ to $w'$, first go from $w$ to $(a, b)^T$ by successive coalescence operations. Then go from $(a, b)^T$ to $w'$ in a finite number of successive selection and coalescence operations. By Theorem~\ref{thm:degreeprocessrecurrenttransient(intro)}, the assumption $b(\beta)<c(\Lambda)$ implies that~$L$ is positive recurrent. In particular, the state $(a,b)^T$ is positive recurrent for $V$. Thus, the restriction of~$V$ to $C_V(a,b)$ is positive recurrent, and hence it admits a unique invariant distribution $\mu^{a,b}$~\citep[Thm. 3.5.2, Thm. 3.5.3]{norris1998markov}. It remains to show that $\mu^{a,b}$ is the unique stationary distribution of $V$ with support included in $\widehat{C}_V(a,b)\coloneqq\{v\in\Rb^\infty: v_0=a,\, v_{\dim(v)-1}=b\}$. This follows directly noting that, since $L$ is positive recurrent, the process $V$ starting in $V_0=v\in \widehat{C}_V(a,b)$ enters~$C_V(a,b)$ in finite time.
\smallskip

		(2) Let $V_{\infty}^{a,b}\sim \mu^{a,b}$. If $V_0=(a,b)^T$, then $V_t\xrightarrow[t\to\infty]{(d)}V_{\infty}^{a,b}$ in law by classic Markov chain theory~\citep[Thm. 3.6.2]{norris1998markov}. On the other hand, for $n\in \N$ and $v\in \R^{n+1}$ with $v_0=a$ and $v_n=b$, as remarked above, $V$ enters $C_V(a,b)$ in finite time. Hence, the result follows.
\end{proof}

\subsection{Coming down from infinity}\label{s5.3}
In this section we prove the criterion for $L$ to come down from infinity and we show that in this case $L$ is positive recurrent.
\begin{proof}[Proof of Theorem~\ref{teo:cdi-criterium(intro)}]
	Recall from \citep{Pit99} that if $\Lambda(\{1\})=0$, then the line-counting process of the $\Lambda$-coalescent either stays infinite with probability $1$, or c.d.i.with probability $1$. The same arguments that are used in the proof of Theorem 4 and Proposition 23 of \cite{Pit99} can be extended to our branching-coalescing system. In particular, the same dichotomy also holds here. Let us now show that $L$ c.d.i.~if and only if the underlying $\Lambda$-coalescent c.d.i.. Since the leaf process stochastically dominates the underlying $\Lambda$-coalescent, it suffices to show that if the underlying $\Lambda$-coalescent c.d.i., so does the leaf process. 
	The later follows from Lemma~\ref{lem:cond-cdi-coalescent} and by letting $n\to\infty$ in the inequality in Lemma~\ref{lem:FiniteExpectedReturnTime}.
\end{proof}

\begin{proposition}
	If $L$ comes down from $\infty$, then $L$ is positive recurrent.
\end{proposition}

\begin{proof}
	Assume that $L$ c.d.i.. By Theorem~\ref{teo:cdi-criterium(intro)}, 
	the underlying $\Lambda$-coalescent also c.d.i..
	By~\cite{herriger2012conditions} (see also \cite[Thm.~2.2]{foucart2013impact}), this implies 
	$\sum_{k=2}^\infty \delta(k)^{-1}<\infty$. Since $\delta(k)/k \to c(\Lambda)$, we must have 
	$c(\Lambda)=\infty$, and the proposition is then a direct application of Theorem~\ref{thm:degreeprocessrecurrenttransient(intro)}.
\end{proof}

%%%%%%%%%%%%%%%%%%%%%%%%%%%%%%%%%%%%%%%%%%%%%%%%%%%%%%%%%%%%%%%%%%%%%%%%%%%%%%%%%%%%%%%%%%%%%%%%%%%%%%%%%%%%%%%%%%%%%%%%%%%%%%%%%%%%%%%%%%%%%%%%%%%%%%%%%%%%%
%%%%%%%%%%%%%%%%%%%%%%%%%%%%%%%%%%%%%%%%%%%%%%%%%%%%%%%%%%%%%%%%Section6%%%%%%%%%%%%%%%%%%%%%%%%%%%%%%%%%%%%%%%%%%%%%%%%%%%%%%%%%%%%%%%%%%%%%%%%%%%%%%%%%%%%%
%%%%%%%%%%%%%%%%%%%%%%%%%%%%%%%%%%%%%%%%%%%%%%%%%%%%%%%%%%%%%%%%%%%%%%%%%%%%%%%%%%%%%%%%%%%%%%%%%%%%%%%%%%%%%%%%%%%%%%%%%%%%%%%%%%%%%%%%%%%%%%%%%%%%%%%%%%%%%

\section{Applications: absorption probabilities and absorption time}\label{s6}

This section is devoted to the proofs of Proposition~\ref{prop:absX} and Proposition~\ref{prop:absorptiontime}, which provide conditions for the accessibility of $0$ and $1$ by $X$, and expressions for the absorption probabilities and absorption times.  
\subsection{Absorption probabilities}\label{sec:absorptionProbability}

In order to prove Proposition~\ref{prop:absX}, we start with a useful lemma.

\begin{lemma}\label{prop:bernoulli} Let $x\in[0,1]$. The following three statements are equivalent
	\begin{enumerate}
		\item We have\begin{equation} \lim_{t\to\infty} 
		\ \E_{e_2}\big[\langle  B_{L_t}(x),V_t\rangle\big]=\lim_{t\to\infty} \E_{e_3}\big[\langle  B_{L_t}(x),V_t\rangle\big]\eqqcolon p(x)\in[0,1].\label{cond:C1}\end{equation} 
		\item If $X_0=x$, then the limit $X_{\infty}\coloneqq\lim_{t\to\infty}X_t$ exists almost surely and $X_{\infty}\sim \mathrm{Ber}(p(x))$.
		\item For all~$v\in\R^{n+1}$ with $n\in \N$, $\lim_{t\to\infty}\E_{v}[\langle  B_{L_t}(x),V_t\rangle]=(1-p(x))\,v_0+p(x)\,v_n.$
	\end{enumerate}
\end{lemma}

\begin{proof} Let $x\in [0,1]$. We first prove that (1) $\Rightarrow$ (2).  The duality in combination with~\eqref{cond:C1} yields
	\begin{equation}
	\lim_{t\to\infty}\E_{x}\big[X_t\big]=\lim_{t\to\infty}\E_{e_2}\big[\langle  B_{L_t}(x),V_t\rangle\big]=p(x)=\lim_{t\to\infty}\E_{e_3}\big[\langle  B_{L_t}(x),V_t\rangle\big]=\lim_{t\to\infty}\E_{x}\big[X_t^2\big].\label{eq:consequenceC1}\end{equation}
	Denote by~$\mu_x(t)$ the law of~$X_t$ starting from $x$. Let $t_k\nearrow\infty$ as~$k\to\infty$. Since $[0,1]$ is compact, $(\mu_x(t_k))_{k\geq 0}$ is tight. Therefore, by Prokhorov's theorem, there exists a weakly convergent subsequence~$(\mu_x(t_{k_l}))_{l\geq 0}$. Let $\mu_x$ denote its limit. Eq.~\eqref{eq:consequenceC1} yields \[0=\lim_{\ell\to\infty}\E_x[X_{t_{k_\ell}}(1-X_{t_{k_\ell}})]=\int_{[0,1]} y(1-y)\mu_x(\dd y).\] Thus, $\mu_x\sim \mathrm{Ber}(p(x))$. Since this limit is independent of the choice of the sequence $(t_k)_{k\geq 0}$, we conclude that $\lim_{t\to\infty}X_t$ exists almost surely and has distribution~$\mu_x$.\\
	Now we prove that (2) $\Rightarrow$ (3). Let $v\in \R^{n+1}$ with $n\in\N$. Assuming (2), the Bernstein duality yields \[\lim_{t\to\infty}\E_{v}\big[\langle  B_{L_t}(x),V_t\rangle\big]=\lim_{t\to\infty} \E_x\big[\langle  B_n(X_{t}),v\rangle\big]=\E_x\big[\langle  B_n(X_{\infty}),v\rangle\big]=(1-p(x))v_0+p(x)v_n,\]
	which proves (3).\\
	Finally we prove that (3) $\Rightarrow$ (1). For this we use (3) with $v=e_2$ and $v=e_3$. Since in both cases the first entry is $0$ and the last one is $1$, the result follows.
\end{proof}

\begin{proof}[Proof of Proposition~\ref{prop:absX}]
	Assume $b(\beta)<c(\Lambda)$. Using Proposition~\ref{prop:abs} with $V_0=e_2$ and $V_0=e_3$, we obtain  \begin{equation}
	\lim_{t\to \infty}\E_{e_2}[\langle  B_{L_t}(x),V_t\rangle]=\E[\langle  B_{L_\infty}(x),V_\infty\rangle]=\lim_{t\to \infty}\E_{e_3}[\langle  B_{L_t}(x),V_t\rangle]\eqqcolon p(x).\label{eq:p(x)1}	\end{equation} where $V_\infty\sim\mu^{0,1}$. In particular, using Proposition~\ref{prop:bernoulli}, we infer that $X_{\infty}\coloneqq\lim_{t\to\infty} X_{t}$ exists almost surely and $X_{\infty}\sim \mathrm{Ber}(p(x))$. Thus, $h(x)=\P_x\left(X_\infty=1\right)=p(x)$ and (1) follows from~\eqref{eq:p(x)1}.
	Furthermore, 
	\begin{equation} \label{eq:p(x)2}h(x)=\E\left[\sum_{i=0}^{L_{\infty}} V_{\infty}(i)b_{i,L_{\infty}}(x)\right]=\sum_{\ell = 1}^{\infty} \P(L_{\infty}=\ell)\sum_{i=0}^\ell d_{i,\ell}b_{i,\ell}(x),\end{equation}
	where $d_{i,\ell}\coloneqq\E[V_{\infty}(i)\mid L_{\infty}=\ell]$. Assume that $V_0=e_2$, using Lemma~\ref{lem:action-coal-frag}, we conclude that $d_{i,\ell}\in[0,1]$. Also by Lemma~\ref{lem:action-coal-frag}, $d_{0,\ell}=0$ and $d_{\ell,\ell}=1$. Hence, if $x\in (0,1)$, \begin{equation}
	0<\sum_{\ell=1}^{\infty} \P(L_{\infty}=\ell) x^\ell\leq h(x) \leq \sum_{\ell=1}^{\infty} \P(L_{\infty}=\ell) (1-(1-x)^\ell)<1,\label{eq:p(x)in01}
	\end{equation}
	which concludes the proof of (2).
\end{proof}

For $\ell\in \N$ and $v\in \R^{\ell+1}$, let $\bar{c}_{k,\ell}(v)$ be the $k$-th coefficient in the monomial basis of $\langle  B_\ell(x),v\rangle$, where $k\in [\ell]_0$. Inspired by Proposition~\ref{prop:absX}, a naive guess is that $h(x)=\sum_{k=0}^{\infty} x^k \E[\bar{c}_{k,L_{\infty}}(V_{\infty})] $. We make this precise in the next lemma.

\begin{proposition}\label{lem:absorptionprobabilityanalytic}
	Assume $b(\beta)<c(\Lambda)$. If~$L_{\infty}$ admits exponential moments of order~$\ln(3)$, then~$h$ is analytic and has series representation $h(x)=\sum_{k=1}^{\infty} c_k x^k$, where $c_k \, = \, \sum_{\ell= k}^{\infty}\E\left[ \bar{c}_{k,\ell}(V_{\infty})\1_{\{L_{\infty}=\ell\}}\right].$
	Moreover,
	\[\forall x\in[0,1], \ \ h(x) \leq \E[(1+2x)^{L_\infty}]-1\]
\end{proposition}

\begin{proof}
	First note that for $\ell\in \N$ and $v\in \R^{\ell+1}$ with $v_0=0$, a straightforward computation yields
	\[\bar{c}_{k,\ell}(v)= \sum_{i=1}^k \binom{\ell}{k}\binom{k}{i}(-1)^{k-i} v_i,\quad k\geq 1,\ \]
	complemented by $\bar{c}_{0,\ell}(v)=0$.
	In particular, $\langle  B_\ell(x),v\rangle=\sum_{k=1}^{\ell}\bar{c}_{k,\ell}(v)x^k$.
	Using~\eqref{eq:p(x)2}, provided that one can exchange the order of summation, a straightforward formal calculation yields \begin{align*}
	h(x)\ = \ \sum_{\ell=1}^{\infty}\sum_{k=1}^\ell x^k \E[\bar{c}_{k,\ell}(V_{\infty}) \1_{\{L_{\infty}=\ell\}}]\ =\ \sum_{k=1}^{\infty}x^k \sum_{\ell=k}^{\infty} \E[\bar{c}_{k,\ell}(V_{\infty}) \1_{\{L_{\infty}=\ell\}}].
	\end{align*}
	It remains to justify the interchange of the two sums in the last identity. This requires the absolute convergence of the series. Note that $\lVert V_{\infty}\rVert_{\infty}\leq 1$ (by means of Lemma~\ref{lem:action-coal-frag}). As a consequence $\lvert \bar{c}_{k,\ell}(V_\infty) \rvert \leq \sum_{i=0}^k \binom{\ell}{k}\binom{k}{i} = \binom{\ell}{k} 2^k$ so that
	\[\sum_{k=1}^{\infty}x^k \sum_{\ell= 1}^{\infty} \E[\lvert \bar{c}_{k,\ell}(V_{\infty})\rvert \1_{\{L_{\infty}=\ell\}}] \ \leq \ \sum_{\ell=1}^\infty \P(L_\infty=\ell) \sum_{k=1}^\ell \binom{\ell}{k} (2x)^k \ \leq \ \E[(1+2x)^{L_\infty}]-1.\]
	In particular the series is absolutely convergent for all~$x$ if $L_\infty$ has exponential moments of order~$\ln(1+2x)$, which is the case under our assumption.
\end{proof}
\begin{remark}
	If $\Lambda(\{0\})>0$, then $L_{\infty}$ has exponential moments of all orders (see Corollary~\ref{coro:exponentialMoments}). In this case $h$ is also harmonic for the infinitesimal generator, and hence it is then possible to derive a system of equations for the $c_k$. 
\end{remark}

\subsection{Absorption time}

Recall from Section~\ref{s2.5}, $\bar p_{i,\ell} = 1-p_{\ell-i,\ell}$. Consider the Bernstein coefficient processes $V=(V_t^n:\,t \geq 0)$ and $W=(W_t^n:\,t\geq 0)$ starting at $e_{n+1}$ and with parameters $(\beta,p,\Lambda)$ and $(\beta,\bar{p},\Lambda)$, respectively, constructed on the basis of the same leaf process~$L$.

\begin{lem}
	For every $t>0$ and $n\in\N$,
	\[\E_x\big[(1-X_t)^n\big] \ = \ \E\big[ \langle  B_{L_t^n}(1-x), W_t^n\rangle\big].\]
\end{lem}
\begin{proof}
	Let $Y_t=1-X_t$. Then $Y\coloneqq(Y_t:\,t\geq0)$ is identical in law to the solution of the SDE
	\[\dd Y_t = \bar{d}(Y_t)\, \dd t+\sqrt{\Lambda(\{0\})\, Y_t(1-Y_t)}\, \dd W_t+\! \! \! \! \int\limits_{(0,1]\times[0,1]}\!\!\!\!\!\!\!\left[\1_{\{u\leq Y_{t-}\}}r(\1- Y_{t-})-\1_{\{u> Y_{t-}\}}r Y_{t-}\right]\tilde{N}(\dd t,\dd r,\dd u),\]
	where
	\[\bar{d}(x)=-d(1-x) = \sum_{\ell=2}^m \beta_\ell \sum_{i=0}^\ell b_{\ell-i,\ell}(x)\Big(-p_{i,\ell} +\frac{i}{\ell}\Big)= \sum_{\ell=2}^m \beta_\ell \sum_{i=0}^\ell b_{i,\ell}(x)\Big(\bar p_{i,\ell} -\frac{i}{\ell}\Big).\]
	By the duality Theorem~\ref{thm:bernstein-duality}, $Y$ is dual to the Bernstein coefficient process with parameters $(\beta,\bar p,\Lambda)$, which completes the proof of the lemma.
\end{proof}

\begin{proof}[Proof of Proposition~\ref{prop:absorptiontime}]
	First,
	\begin{align}
	\E_x\big[ X_t^n + (1-X_t)^n\big] & = \E_x\big[ \1_{\{T\leq t\}} \ (X_t^n + (1-X_t)^n) \big] \ + \  \E_x\big[ \1_{\{T>t\}}  \ (  X_t^n + (1-X_t)^n )\big] \nonumber \\
	& =  \P_x(T\leq t) \ +   \E_x\big[ \1_{\{T>t\}}  \ (  X_t^n + (1-X_t)^n )\big].\label{eq:Qt}  
	\end{align}
	By the monotone convergence theorem, the second term on the right goes to $0$ as $ n\to \infty$. On the other hand, by the duality,
	we have
	\[  \E_x\big[ X_t^n\big] \ = \ \E[\langle  B_{L_t^n}(x), V_t^n\rangle], \ \   
	\E_x\big[ (1- X_t)^n] =   \E[\langle  B_{L_t^n}(1-x), W_t^n\rangle].\]
	The first identity of the proposition then follows by letting $n\to\infty$ in~\eqref{eq:Qt}. For the second identity, let $t=\tau^{(n)}$. Applying Lemma~\ref{lem:action-coal-frag} to $V^n$ and $W^n$, we obtain $V_{t}^n=W_{t}^n = e_2$. Hence,
	\[Q_t^{n}(x)\, = \, \langle  B_{L_t^n}(x), V_t^n\rangle + \langle  B_{L_t^n}(1-x), W_t^n\rangle \, = \, x +(1-x)  \, = \,  1.\] 
	Further, from the definition of the coalescence and selection matrices it follows that the latter identity extends to any $t\geq \tau^{(n)}$. Using this and $\P_x(T\leq t)=\lim_{n\to\infty}\E[Q_t^n(x)]$, we get
	$$\E_x[T]=\int_{0}^{\infty} \P_x(T\geq t)\dd t=\int_0^\infty\lim\limits_{n\to\infty}\E\left[(1-Q_t^n(x))\1_{\{t\leq \tau^{(n)}\}}\right]\dd t$$
	Note that $0\leq(1-Q_t^n(x))\1_{\{t\leq \tau^{(n)}\}}\leq \1_{\{t\leq \tau^{(\infty)}\}}$, and since~$L$ c.d.i. we have $\E[\tau^{(\infty)}]<\infty$. Hence, using dominated convergence theorem and Fubini's theorem for positive functions, we obtain
	$$\E_x[T]=\lim\limits_{n\to\infty}\E\left[\int_0^\infty (1-Q_t^n(x))\1_{\{t\leq \tau^{(n)}\}}\dd t \right]\leq \E[\tau^{(\infty)}].$$
\end{proof}

%%%%%%%%%%%%%%%%%%%%%%%%%%%%%%%%%%%%%%%%%%%%%%%%%%%%%%%%%%%%%%%%%%%%%%%%%%%%%%%%%%%%%%%%%%%%%%%%%%%%%%%%%%%%%%%%%%%%%%%%%%%%%%%%%%%%%%%%%%%%%%%%%%%%%%%%%%%%%
%%%%%%%%%%%%%%%%%%%%%%%%%%%%%%%%%%%%%%%%%%%%%%%%%%%%%%%%%%%%%%%%Section7%%%%%%%%%%%%%%%%%%%%%%%%%%%%%%%%%%%%%%%%%%%%%%%%%%%%%%%%%%%%%%%%%%%%%%%%%%%%%%%%%%%%%
%%%%%%%%%%%%%%%%%%%%%%%%%%%%%%%%%%%%%%%%%%%%%%%%%%%%%%%%%%%%%%%%%%%%%%%%%%%%%%%%%%%%%%%%%%%%%%%%%%%%%%%%%%%%%%%%%%%%%%%%%%%%%%%%%%%%%%%%%%%%%%%%%%%%%%%%%%%%%

\section{Minimal ancestral structures}\label{s7}

According to Theorem~\ref{thm:infiniteed}, there are infinitely many selection decompositions for a polynomial vanishing at the boundary of the unit interval. In view of the conditions that guarantee the existence of a unique stationary distribution for the Bernstein coefficient process and from which one can deduce the accessibility of the boundary for the~$\Lambda$-Wright--Fisher process, one would like to identify selection decompositions with minimal effective branching rate. To achieve this, we first derive a geometrical characterisation of the set of Bernstein coefficients vectors that admit a selection decomposition with effective branching rate~$\lambda$ in Section~\ref{sect:minimal-models2}. This leads to a characterisation of the set of $b$-minimal selection decompositions for a given polynomial~$d$. For the sake of illustration, we examine the case $\deg(d)=2,3$ more closely in Section~\ref{sect:minimal-models20}. We put a particular focus on the classic diploid selection model with dominance, see Eq.~\eqref{eq:wf-dominance}. In this case the faces of~$\cS_1$ admit a natural biological interpretation in terms of recessive/dominant positive/negative selection. 

\smallskip

If we want to distinguish selection decompositions in such a way that we minimise the number of superfluous branches in the ASG, the notion of graph-minimality of a selection decomposition is more natural. It is not clear whether the two notions of minimality are equivalent. In Section~\ref{sect:minimal-models} we show that any $b$-minimal selection decomposition is also graph-minimal. A proof of the converse statement for~$\deg(d)=3$ is given in Section~\ref{sec:equivalence}. We conjecture that this is true in any dimension. However, the situation is more involved in higher dimensions, and the problem of the equivalence between the two notions of minimality remains open.

\subsection{\texorpdfstring{Finding $b$-minimal selection decompositions}{Finding b-minimal selection decompositions}} 
\label{sect:minimal-models2}
In this section we derive a geometric characterisation of the set of $b$-minimal selection decompositions. The starting point is the geometrical representation of the minimal effective branching rate $b_\star(d)$ as $\inf\{\lambda>0:  \rho(d) \in \cS_\lambda \}$ in Proposition~\ref{exminimal}. Before we proceed, we recall some relevant notation. For a polynomial~$f$ vanishing at the boundary, $\rho(f)=(\rho_i(f))_{i\in [m-1]}$ is the unique vector such that $f(x)=\sum_{i=1}^{m-1}\rho_i(f) b_{i,m}(x)$ for all $x\in\Rb$. If the context is clear, we refer to $\rho(f)$ as the Bernstein coefficient vector of~$f$. The set $\cS_\lambda\subset\Rb^{m-1}$ is the set of $\lambda$-decomposable vectors, i.e. $ \Ss_\lambda= \{\rho(d_{\beta,p}):(\beta,p)\in \R_+^{m-1}\times \Ps_m, \, b(\beta) = \lambda \}$, where $d_{\beta,p}$ is defined in~\eqref{eq:driftbetap}. The next lemma provides insight into the structure of this set.

\begin{lemma}\label{prop:betapBCV}
	For every $(\beta,p)\in \R_+^{m-1}\times \Ps_m$ and for every $i\in[m-1]$,
	$$\rho_i(d_{\beta,p}) = \sum_{\ell=2}^m\beta_\ell \left(\E\left[p^{}_{K_{\ell,i},\ell}\right]-\frac{i}{m}\right),$$
	where $K_{\ell,i}\sim\hypdist{m}{\ell}{i}$.
\end{lemma}
\begin{proof}
	Recall that $$d_{\beta,p}(x)=\sum_{\ell=2}^{m} \beta_\ell\sum_{j=0}^\ell b_{j,\ell}(x) \left(p^{}_{j,\ell}-\frac{j}{\ell}\right).$$
	Since
	\begin{equation}\label{eq:degreeelevation}
		b_{j,\ell}(x)=\sum_{i=j}^{m-\ell+j}\frac{\binom{\ell}{j}\binom{m-\ell}{i-j}}{\binom{m}{i}}b_{i,m}(x),\end{equation}
	applying Fubini's theorem yields for $i\in[m-1]$,
	$$\rho_i(d_{\beta,p})=\sum_{\ell=2}^m\beta_{\ell} \sum_{j=0\vee (i+\ell-m)}^{i\wedge\ell}\frac{\binom{\ell}{j}\binom{m-\ell}{i-j}}{\binom{m}{i}} \left(p^{}_{j,\ell}-\frac{j}{\ell}\right).$$
	The result follows from classical properties of the hypergeometric distribution.
\end{proof}
For $\ell\in\, ]m]$, consider $(\beta,p)\in \R_+^{m-1}\times \Ps_m$ such that $\beta_\ell>0$ and for $i\neq \ell$, $\beta_i=0$, i.e. there are only $\ell$-interactions. The representation of Lemma~\ref{prop:betapBCV} then reads \begin{equation}
	\rho(d_{\beta,p})=\beta_\ell (\theta_\ell (p_{\cdot,\ell})-u_m),\label{eq:thetabetal}
\end{equation} where $u_m\coloneqq (i/m)_{i=1}^{m-1}$ and $$\theta_{\ell}:\{0\}\times[0,1]^{\ell-1}\times\{1\}\to[0,1]^{m-1},\qquad p\mapsto\theta_{\ell}(p)\coloneqq\Big(\E\big[p^{}_{K_{\ell,i}}\big]\Big)_{i=1}^{m-1}.$$ 
We show in the forthcoming proposition that for this special choice of $\beta$, \eqref{eq:thetabetal} leads to a representation of $\Ss_\lambda$. Before stating the result we introduce some definitions. For every $\ell\in\,]m]$, define
\[\cS_{\lambda}^{\ell} \ \coloneqq \ \big\{\rho(d_{\beta,p}):(\beta,p)\in \R_+^{m-1}\times \Ps_m, \ b(\beta)=\lambda, \ \forall i\in\,]m]\setminus\{\ell\}, \ \ \beta_i=0  \big\},\]
i.e. $\Ss_\lambda^\ell$ are the $\lambda$-decomposable vectors consisting only of $\ell$-interactions. Next, define
$\Ps_{0,1}^\ell\coloneqq\{0\}\times\{0,1\}^{\ell-1}\times\{1\}$, i.e. the set of deterministic $\ell$-colouring rules (cf. Remark~\ref{rem:determiisticvoting(intro)}). For~$\ell\in\,]m]$, we say that $(\beta,p)\in \R^{m-1}\times \Ps_m$ is \emph{$\ell$-extremal} if \begin{itemize}
	\item $\beta_{\ell} >0$, and for all $i\in\,]m]\setminus\{\ell\}, \beta_i=0$.
	\item $p_{\cdot,k}\in \Ps_{0,1}^{k}$ for every $k\in\,]m]$.
\end{itemize}
Hence, an $\ell$-extremal selection decomposition has only $\ell$-interactions and deterministic colouring rules. Note that if an $\ell$-extremal selection decomposition has effective branching rate~$\lambda$, then $\beta_\ell=\lambda/(\ell-1)$. Recall that for $K\subseteq\R^{m-1}$, $\conv{K}$ denotes the convex hull of~$K$.

\begin{proposition}[Characterisation of $\cS_\lambda$]\label{prop:description-s-lambda}
	We have
	\[\cS_{\lambda} \ = \ \conv{\left\{\frac{\lambda}{\ell-1}\big(\theta_\ell(p) -u_m \big):\ell\in\,]m], p\in \Ps_{0,1}^\ell\right\}  }. \]
	Furthermore, 
	\[\cS_{\lambda}^\ell \ = \ \conv{ \left\{\frac{\lambda}{\ell-1}\big(\theta_\ell(p) -u_m \big): p \in \Ps_{0,1}^\ell \right\}}. \]
\end{proposition}
\begin{remark}
	According to Proposition~\ref{prop:description-s-lambda}, every Bernstein coefficient vector admitting a selection decomposition with effective branching rate~$\lambda$ and consisting only of~$\ell$-interactions is a convex combination of Bernstein coefficient vectors of $\ell$-extremal selection decompositions with effective branching rate~$\lambda$. More generally, any Bernstein coefficient vector that admits a selection decomposition with effective branching rate~$\lambda$ is a convex combination of Bernstein coefficient vectors of extremal selection decompositions with effective branching rate~$\lambda$. 
\end{remark}
\begin{proof}[Proof of Proposition~\ref{prop:description-s-lambda}]
	We use the notation of Proposition~\ref{prop:s-polytope}. Recall that $G_\lambda=\{\beta\in \R_+^{m-1}:b(\beta)=\lambda\}$ and
	$\cS_\lambda$ is the image of $G_\lambda\times \Ps_m$
	under the linear-affine map $(\beta,p)\mapsto \rho(d_{\beta,p})$. Consequently, $\cS_\lambda$
	is the convex hull of the image of the extreme points of $G_\lambda\times \Ps_m$. The extreme points are the extremal selection decompositions with effective branching rate~$\lambda$. 
	By Proposition \ref{prop:betapBCV}, for an $\ell$-extremal $(\beta,p)$, $\rho(d_{\beta,p})=\lambda\left(\theta_\ell(p_{\cdot,\ell}) -u_m \right)/(\ell-1)$. This proves the first identity. The identity for $\cS_\lambda^\ell$ is proved along the same lines.
\end{proof}

Let $\ell\in\,]m]$. Fix a branching rate vector $\beta$ with effective branching rate~$\lambda>0$ and only~$\ell$-branchings, i.e. $\beta_\ell=\lambda/(\ell-1)$ and for $i\neq \ell$, $\beta_i=0$. With such~$\beta$ fixed, \eqref{eq:thetabetal} defines a map that associates to every colouring rule an element in $\R^{m-1}$. If this map were injective, we could associate to every $\bar{\rho}\in \R^{m-1}$ a unique selection decomposition consisting only of $\ell$-interactions and having effective branching rate~$\lambda$. The map in~\eqref{eq:thetabetal} is injective if and only if~$\theta_\ell$ is injective.

\begin{lemma}\label{lem:thetaInjective}
	$\theta_\ell$ is an injective map. 
\end{lemma}
\begin{proof}
	Note that $\theta_\ell$ can be extended to a linear map on $\{0\}\times \R^{\ell}$. Hence, it suffices to show that if $\theta_\ell(p)=0$, then~$p=0$. Assume $p=(p_0,\ldots,p_{\ell})\in\{0\}\times \R^{\ell}$ is such that~$\theta_\ell(p)=0$. By assumption, $p_0=0$. Since $\E[\,p_{K_{\ell,1}}]=p_1\ell!/m!=0$, it follows that $p_1=0$. We proceed by induction. Assume $p_i=0$ for all $i\leq k$ for some~$k<\ell$. By assumption and by the induction hypothesis, \[0=\E[\,p^{}_{K_{\ell,k+1}}]=\sum_{j=0\vee (\ell+k+1-m)}^{k+1} \frac{\binom{\ell}{j}\binom{m-\ell}{k+1-j}}{\binom{m}{k+1}}p_j=\frac{\binom{\ell}{k+1}}{\binom{m}{k+1}}p_{k+1}.\]
	It follows that also $p_{k+1}=0$. Altogether, $p= 0$.
\end{proof}

Hence, $\theta_\ell^{-1}$ leads to a unique selection decomposition consisting of only~$\ell$-interactions and having effective branching rate~$\lambda$. In a next step we associate to~$\bar{\rho}\in \Ss_\lambda$ a general selection decomposition. To do so we consider for each~$\ell\in\,]m]$, $v_\ell \in \Ss_\lambda^\ell$ (not necessarily extremal) such that we can write $\bar{\rho}$ as a convex combination of $\{v_\ell\}_{\ell\in\,]m]}$. We construct the general selection decomposition for $\bar{\rho}$ by combining the selection decompositions that are associated to each element in $\{v_\ell\}_{\ell\in\,]m]}$. To end up with a general selection decomposition with effective branching rate~$\lambda$, we decrease the rate of the $\ell$-interaction by the weight of $v_\ell$ in the convex combination of $\bar{\rho}$. In particular, a general selection decomposition is then parametrised by the effective branching rate~$\lambda$, a set of vectors $\{v_\ell\}_{\ell\in\,]m]}$, and their convex weights $\alpha\in \Delta_{m-2}$, where for $m\geq2$, $\Delta_{m-2}$ is the $m-2$-simplex defined via $$\Delta_{m-2}\coloneqq\left\{\alpha\coloneqq(\alpha_\ell)_{\ell=2}^m\in[0,1]^{m-1}:\sum_{\ell=2}^{m} \alpha_\ell=1\right\}.$$ This leads to the following definition.
\begin{definition}[$\lambda$-convex decompositions]\label{def:lambdaconvexdecompositions}
	Let~$\lambda>0$. The \emph{set of~$\lambda$-convex decompositions} is defined as $\Cs^\lambda\coloneqq\left(\prod_{\ell=2}^m \cS_{\lambda}^\ell \right)\times \Delta_{m-2} $. For $\bar{\rho}\in \R^{m-1}$, we say that $(\vec{v},\alpha)\in \Cs^\lambda$ is a \emph{$\lambda$-convex decomposition} of~$\bar{\rho}$ if $\bar{\rho}=\sum_{\ell=2}^{m}\alpha_\ell v_\ell$. Define \begin{equation}
		\Cs \coloneqq\left\{(\lambda,\vec{v},\alpha)\in \R_+\times \left(\prod_{\ell=2}^m \Rb^{m-1}\right)\times\Delta_{m-2} : (\vec{v},\alpha)\in \Cs^\lambda \right\}.
	\end{equation}
	Furthermore, define $\varphi:\Cs \to \R_+^{m-1}\times \Ps_m$ as $\varphi(\lambda,\vec{v},\alpha)=(\beta^{(\lambda,\vec{v},\alpha)},p^{(\lambda,\vec{v},\alpha)}),$
	where  $$\beta^{(\lambda,\vec{v},\alpha)}_{\ell}\coloneqq\frac{\lambda}{\ell-1}\alpha_{\ell}\qquad \text{and}\qquad p_{\cdot,\ell}^{(\lambda,\vec{v},\alpha)}\coloneqq\theta^{-1}_{\ell}\left(\frac{\ell-1}{\lambda}v_\ell+u_m\right),\qquad \ell\in\,]m].$$
\end{definition}

The next result states that the reasoning preceding Definition~\ref{def:lambdaconvexdecompositions} is correct, i.e we can indeed associate to each convex combination of $\bar{\rho}\in \Ss_\lambda$ with elements in $\Ss_\lambda^\ell$ (for $\ell\in\,]m]$) a unique selection decomposition.
\begin{proposition}[Embedding]\label{lem:varphiBijection}
	$\varphi$ is a bijection from $\Cs$ to $\R_+^{m-1}\times \Ps_m$ 
\end{proposition}
\begin{proof}
	The injectivity of~$\varphi$ follows from the injectivity of $\theta_\ell$. 
	For the surjectivity, consider~$(\beta,p)\in \R_+^{m-1}\times \Ps_m$. Set $\lambda\coloneqq b(\beta)$, $\alpha_\ell\coloneqq\beta_\ell(\ell-1)/\lambda$, and $v_\ell\coloneqq\lambda (\theta_\ell(p_{\cdot,\ell})-u_m)/(\ell-1)$. Clearly, $\lambda\in \R_+$ and $\alpha\in \Delta_{m-1}$. We claim that $\vec{v}\coloneqq(v_\ell)_{\ell=2}^m\in \prod_{\ell=2}^m \cS_\lambda^{\ell}$. To see this, first note that $p_{\cdot,\ell}$ can be written as a convex combination of elements in the extreme set $\Ps_{0,1}^\ell$. Since~$(\beta,p)\mapsto \rho(d_{\beta,p})$ is affine in the second argument (see proof of Proposition~\ref{prop:s-polytope}), the claim follows by the characterisation of~$\cS_\lambda^{\ell}$ given in the second part of Proposition~\ref{prop:description-s-lambda}. Finally, note that indeed $\varphi(\lambda,\vec{v},\alpha)=(\beta,p)$.
\end{proof}
Ultimately, we search for all selection decompositions of the drift polynomial in~\eqref{eq:SDEoriginal02}. For a polynomial~$d$ with $\deg(d)\leq m$ and $d(0)=d(1)=0$ define 
\[\cC_d \coloneqq \left\{ (\lambda,\vec{v}, \alpha) \in \cC \ : \ (\vec{v},\alpha)
\ \text{is a}\ \lambda\text{-convex decomposition of}\ \rho(d) \right\},\]
i.e. $\cC_d$ is the set of convex decompositions of~$\rho(d)$. The next result relates~$\cC_d$ to the set of selection decompositions of~$d$.
\begin{corollary}[Geometric characterisation of $\ccS_d$]\label{coro:geometriccharacterisationced}
	$\ccS_d$ coincides with $\varphi(\cC_d)$.
\end{corollary}
\begin{proof}
	$(\beta,p)\in\ccS_d$ if and only if $\rho(d)=\rho(d_{\beta,p})$. Moreover, for $(\lambda,\vec{v},\alpha)\in \cC_d$, \[\rho(d_{\beta^{(\lambda,\vec{v},\alpha)},p^{(\lambda,\vec{v},\alpha)}})=\sum_{\ell=2}^{m} \alpha_\ell v_\ell=\rho(d),\]
	where the last equality holds, since $(\vec{v},\alpha)$ is a~$\lambda$-convex decomposition of~$\rho(d)$. Since $\varphi$ is bijective, the result follows.
\end{proof}
\begin{figure}[t!]
	\begin{minipage}{0.4\textwidth}
		\begin{center}
			\scalebox{0.4}{
				\begin{tikzpicture}
				%\Ss^3
				%\fill [gray,opacity=0.2] (-2.5,-5) rectangle (5,2.5);
				%\fill [gray,opacity=0.2] (-1.66,-3.33) rectangle (3.33,1.66);
				\fill [gray,opacity=0.4] (-3.5,-7) rectangle (7,3.5);
				%-1.02x+0.16=y
				%-2/5+x=y
				%-147/55x+391/550=y
				%Box

				%\Ss^2
				%\draw[line width=.5mm, opacity=0.2] (-3.33,-3.33) -- (3.33,3.33);
				%\draw[line width=.5mm, opacity=0.2] (-5,-5) -- (5,5);
				\draw[line width=.5mm, opacity=0.4] (-7,-7) -- (7,7);

				%\draw[line width=.2mm, gray, dotted] (-5,5) -- (8.5,-8.5);

				\node[above] at (2,5.5) {\scalebox{2.6}{$\Ss_\lambda$}};
				
				% \Ss_\lambdas
				\draw[line width=.5mm] (-7,-7) -- (-3.5 ,3.5) -- (7,7) -- (7,-7) -- (-7,-7);
				
				\node at (3.33,-1.8) {\scalebox{2.6}{$\times$}};
				\node[right] at (3.6,-1.5) {\scalebox{2.6}{$\rho(d)$}};
				\node at (.55,1.2) {\scalebox{2.6}{$v_2$}};
				\node at (5.6,-3.7) {\scalebox{2.6}{$v_3$}};
				\node at (2,-1) {\scalebox{2.6}{$l_2$}};
				\node at (4.1,-3) {\scalebox{2.6}{$l_3$}};
				\node[left] at (-7,-7) {\scalebox{2.6}{\color{red}}};
				%\node[left] at (1.7,-.75) {\scalebox{2.6}{$(0,0)$}};
				\node[right] at (7,7) {\scalebox{2.6}{\color{red} }};
				\node[left] at (-3.5,3.5) {\scalebox{2.6}{}};
				\node[right] at (7,3.5) {\scalebox{2.6}{}};
				\node[right] at (7,-7) {\scalebox{2.6}{}};
				\node[below] at (-3.5,-7) {\scalebox{2.6}{}};
				
				%\node at (2.4,1.9) {\scalebox{1.5}{$v_2'$}};
				%\node at (4,-4.45) {\scalebox{1.5}{$v_3'$}};
				%\node[right] at (3.38,-1.7) {\scalebox{1.5}{$v^{i,\ell}$}};
				
				\draw[line width=.5mm, black] (5,-3.5) -- (3.33,-1.8) -- (.79,.79);
				%\draw[line width=.5mm, black] (4.25,-4.25) -- (3.33,-1.8) -- (1.935,1.935);
				\node[left,opacity=0] at (0,-7) {\scalebox{2.6}{$(-\frac{1}{3},-\frac{1}{3})$}};
				\end{tikzpicture}}
		\end{center}
	\end{minipage}~~ ~~\begin{minipage}{0.45\textwidth}
		\begin{center}
			\scalebox{0.4}{
				\begin{tikzpicture}
				%\Ss^3
				%\fill [gray,opacity=0.2] (-2.5,-5) rectangle (5,2.5);
				%\fill [gray,opacity=0.2] (-1.66,-3.33) rectangle (3.33,1.66);
				\fill [gray,opacity=0.4] (-3.5,-7) rectangle (7,3.5);
				%-1.02x+0.16=y
				%-2/5+x=y
				%-147/55x+391/550=y
				%Box
				
				\draw[line width=.5mm, opacity=0.4] (-7,-7) -- (7,7);
				%\Ss^2

				\node[above] at (2,5.5) {\scalebox{2.6}{$\Ss_1$}};
				
				% \Ss_\lambdas
				\draw[line width=.5mm] (-7,-7) -- (-3.5 ,3.5) -- (7,7) -- (7,-7) -- (-7,-7);
				\node at (0,0) {\scalebox{2.8}{$\times$}};
				\node[left] at (-7,-7) {\scalebox{2.6}{$(-\frac{1}{3},-\frac{1}{3})$}};
				\node[left] at (1.9,-.75) {\scalebox{2.6}{$(0,0)$}};
				\node[right] at (7,7) {\scalebox{2.6}{}};
				\node[left] at (-3.5,3.5) {\scalebox{2.6}{$(-\frac{1}{6},\frac{1}{6})$}};
				\node[right] at (7,3.5) {\scalebox{2.6}{}};
				\node[right] at (7,-7) {\scalebox{2.6}{}};
				\node[below] at (-3.5,-7) {\scalebox{2.6}{}};
				
				\draw[line width=.4mm, gray, dotted] (0,0) -- (-10,-3);
				\node at (-10,-3) {\scalebox{2.8}{$\times$}};
				\node at (-5.2,-1.5) {\scalebox{2.8}{$\times$}};
				\node at (-5.8,-0.5) {\scalebox{2.6}{$\rho(\bar{d})$}};
				\node[left] at (-9.7,-2.3) {\scalebox{2.6}{$\rho(d)$}};

				\end{tikzpicture}}
		\end{center}
	\end{minipage} ~~ ~~ ~~ 
	\caption{Left: Representation of some~$(\lambda,\vec{v},\alpha)\in \Cs_d$ for~$m=3$. Here, $l_2=\lVert v_2-\rho\rVert_2$ and~$l_3=\lVert v_3-\rho\rVert_2$. $\alpha_2=l_3/(l_2+l_3)$ and $\alpha_3=l_2/(l_2+l_3)$. Right: $\mathbb{L}_d$ and $\cS_1$ intersect on the west face of the polygon at the point $ \rho(\bar{d})$. Since $1$-convex decompositions must involve a point in $\cS_1^2$ (the diagonal segment inside $\Ss_1$), the unique $1$-convex decomposition of $ \rho(\bar{d})$ is obtained by taking the $1$-convex decomposition of $\rho(d)$ involving the two extreme points of the west face.}
	\label{fig:exampleCsdm200}
\end{figure}

\begin{remark}\label{rem:interpretationconvexm3} We now explain the case~$m=3$ in more detail. The objects of Definition~\ref{def:lambdaconvexdecompositions} admit a clear graphical interpretation. Consider a selection decomposition of $(\beta,p)\in\ccS_d$. We recover $(\lambda,\vec{v},\alpha)=\phi^{-1}(\beta,p)$ directly from Figs.~\ref{fig:exampleCsdm200} and~\ref{fig:exampleCsdm2}. The first entry~$\lambda$ corresponds to the effective branching rate and it fixes~$\Ss_\lambda^2$ (Fig.~\ref{fig:exampleCsdm2}, red diagonal) and~$\Ss_\lambda^3$ (Fig.~\ref{fig:exampleCsdm2}, grey square). Next, $\vec{v}=(v_2,v_3)$ where $v_2$ and $v_3$ are Bernstein coefficient vectors given by
	\[v_\ell = \beta_\ell\left(\theta_\ell(p_{\cdot,\ell})-u_3\right), \ \ \ell=2,3,\]
	correspond to points in~$\Ss_\lambda^2$ and $\Ss_\lambda^3$, respectively. To read of the convex weight, we set $l_2=\lVert v_2-\rho(d_{\beta,p})\rVert_2$ and~$l_3=\lVert v_3-\rho(d_{\beta,p})\rVert_2$. Then $\alpha_2=l_3/(l_2+l_3)$ and $\alpha_3=l_2/(l_2+l_3)$ are the convex weights in the combination of $v_2$ and $v_3$ leading to $\rho(d_{\beta,p})$. 
\end{remark}

For a given polynomial~$d$ vanishing at the boundary, it remains to determine the set of selection decompositions that have a minimal effective branching rate. We start with the following observation. Denote by~$O$ the origin in $\Rb^{m-1}$. Let $\mathbb{L}_d\coloneqq\{v\in \R^{m-1}: v=\lambda \rho(d)\ \text{for some} \ \lambda>0\}$, i.e. the half line passing through the origin and $\rho(d)$, with extremity~$O$. Since $\cS_1$ contains~$O$, $\mathbb{L}_d$ intersects with $\cS_1$ in a unique point; for an illustration if $m=3$ see Fig.~\ref{fig:exampleCsdm200} (right). Let $ \rho(\bar{d})$ be this point and let $\bar{d}$ be the polynomial, i.e. $\bar d(x)  =  \langle  B_m(x), \rho(\bar{d}) \rangle.$ Note that $\rho(\bar{d})$ lies on the boundary of $\Ss_1$, and hence $b_\star(\bar d)=1$.
\begin{proposition}[Scaling]\label{prop:scaling-polu}
	~ \begin{enumerate}
		\item[(i)] Let $i\in\,]m-1]$ such that $\rho_i(\bar{d})\neq 0$. Then
		\[b_{\star}(d) = \frac{\rho_i(d)}{\rho_i(\bar{d})}.\]
		\item[(ii)] $(\beta,p)\in \ccS_{\bar{d}}$ if and only if $(b_{\star}(d) \beta,p)\in \ccS_d.$
	\end{enumerate}
\end{proposition}
\begin{proof}
	This easily follows from the previous result and the scaling relation $\cS_\lambda = \lambda \cS_1$.
\end{proof}

Finally, we identify the $b$-minimal selection decompositions for~$d$. By the above discussion and by Proposition~\ref{prop:scaling-polu}, it is enough to determine the $b$-minimal selection decomposition of $\bar d$. More precisely, $\bar{d}$ is determined by the point of intersection~$\rho(\bar{d})$ of the line~$\mathbb{L}_d$ and the polygon~$\cS_1$. Finding a point of intersection of a line and a polygon is a classic problem in computational geometry, see e.g.~\cite{foley1996computer,cyrus1978generalized,liang1984new}. We now provide an algorithm to find a $b$-minimal selection decomposition taking into account the following points. The effective branching rate of the $b$-minimal selection decompositions of~$d$ can be recovered using Proposition~\ref{prop:scaling-polu}{-(i)}. The $b$-minimal selection decompositions of $\bar{d}$ correspond to the image under~$\varphi$ of the convex combinations of $\rho(\bar{d})$ in~$\Ss_1$. The $b$-minimal selection decompositions of~$d$ are recovered from the $b$-minimal selection decompositions of $\bar{d}$ by scaling with $b_\star(d)$.
\begin{algorithm}[Finding $b$-minimal selection decompositions]\label{teo:algorithm}
	\hspace{2em}
	\begin{enumerate}
		\item[Step 1.] Compute the point of intersection $\rho(\bar{d})$.
		\item[Step 2.] Compute $b_{\star}(d)$ via Proposition~\ref{prop:scaling-polu}-(i).
		\item[Step 3.] Determine $\cC_{\bar{d}}^{\min}\coloneqq\{(\lambda,\vec{v},\alpha)\in \cC_{\bar{d}}:\lambda=1\}$ and set $\ccS_{\bar d}^{\min} \coloneqq \varphi(\cC_{\bar d}^{\min})$. 
		\item[Step 4.] Finally $\ccS_d^{\min}\coloneqq\{(b_{\star}(d)\beta,p): (\beta,p)\in \ccS_{\bar d}^{\min} \}$. 
	\end{enumerate}
\end{algorithm}
Then $\ccS_d^{\min}$ is the set of $b$-minimal selection decompositions.

\subsection{\texorpdfstring{Minimal selection decompositions if \boldmath{$m=2,3$}}{Minimal selection decompositions if m=2,3}}\label{sec:minimalSDm3}
\label{sect:minimal-models20}

\begin{figure}[t!]
	\begin{center}
		\scalebox{0.4}{
			\begin{tikzpicture}
			%\Ss^3
			%\fill [gray,opacity=0.2] (-2.5,-5) rectangle (5,2.5);
			%\fill [gray,opacity=0.2] (-1.66,-3.33) rectangle (3.33,1.66);
			\fill [gray,opacity=0.4] (-3.5,-7) rectangle (7,3.5);
			%-1.02x+0.16=y
			%-2/5+x=y
			%-147/55x+391/550=y
			%Box

			%\Ss^2
			%\draw[line width=.5mm, opacity=0.2] (-3.33,-3.33) -- (3.33,3.33);
			%\draw[line width=.5mm, opacity=0.2] (-5,-5) -- (5,5);
			\draw[line width=.5mm, opacity=0.4, red] (-7,-7) -- (7,7);

			%\draw[line width=.2mm, gray, dotted] (-5,5) -- (8.5,-8.5);

			\node[above] at (2,5.5) {\scalebox{2.6}{$\Ss_\lambda$}};
			
			% \Ss_\lambdas
			\draw[line width=.5mm] (-7,-7) -- (-3.5 ,3.5) -- (7,7) -- (7,-7) -- (-7,-7);

			\node[left] at (-7.4,-7) {\scalebox{2.6}{\color{red}$\lambda(-\frac{1}{3},-\frac{1}{3})$}};
			\node[right] at (7.4,7) {\scalebox{2.6}{\color{red} $\lambda(\frac{1}{3},\frac{1}{3})$}};
			\node[left] at (-3.9,3.5) {\scalebox{2.6}{$\lambda(-\frac{1}{6},\frac{1}{6})$}};
			\node[right] at (7.4,3.5) {\scalebox{2.6}{$\lambda(\frac{1}{3},\frac{1}{6})$}};
			\node[right] at (7.4,-7.4) {\scalebox{2.6}{$\lambda(\frac{1}{3},-\frac{1}{3})$}};
			\node[below] at (-3.5,-7.4) {\scalebox{2.6}{$\lambda(-\frac{1}{6},-\frac{1}{3})$}};

			\draw (-7,-7) circle (1.5 mm)  [fill=red];
			\draw (7,7) circle (1.5 mm)  [fill=red];
			\draw (-3.5,-7) circle (1.5 mm)  [fill=black!100];
			\draw (-3.5,3.5) circle (1.5 mm)  [fill=black!100];
			\draw (7,3.5) circle (1.5 mm)  [fill=black!100];
			\draw (7,-7) circle (1.5 mm)  [fill=black!100];
			\end{tikzpicture}}
	\end{center}
	\caption{Points $(\rho_1,\rho_2)$ with red coordinates (resp. black coordinates) are the extreme points $(\rho_1,\rho_2)^T$ of $\cS_\lambda^2$ (resp. $\cS_\lambda^3$), and they correspond to $2$-extremal selection decompositions (resp. $3$-extremal). $\cS_\lambda^2$ is the red line, $\cS_\lambda^3$ is the grey square, and $\cS_\lambda$ is their convex hull.}
	\label{fig:exampleCsdm2}
\end{figure}

\begin{figure}[b]
	\begin{center}
		\scalebox{0.4}{
			\begin{tikzpicture}
			\draw[pattern=fivepointed stars ,draw=none,pattern color=red!60] plot[tension=1] coordinates{(0,0) (-6.5,6.5) (-8.5,-8.5) (0,0)};
			\draw[pattern=fivepointed stars,draw=none,,pattern color=green!60] plot[tension=4] coordinates{(0,0) (8,8) (8.5,-8.5) (0,0)};
			\draw[pattern=dots ,draw=none,pattern color=green!60] plot[tension=4] coordinates{(0,0) (-6.5,6.5) (8,8) (0,0)};
			\draw[pattern=dots,draw=none,pattern color=red!60] plot[tension=4] coordinates{(0,0) (8.5,-8.5) (-8.5,-8.5) (0,0)};

			\fill [gray,opacity=0.4] (-3.5,-7) rectangle (7,3.5);
			%-1.02x+0.16=y
			%-2/5+x=y
			%-147/55x+391/550=y
			%Box

			%\draw[line width=.2mm, gray, dotted] (-5,5) -- (8.5,-8.5);
			
			\node[right] at (-10,0) {\scalebox{2.8}{\color{black} $\dns$}};
			\node[right] at (+10,0) {\scalebox{2.8}{\color{black} $\dps$}};
			\node[right] at (0,-9) {\scalebox{2.8}{\color{black} $\rns$}};
			\node[right] at (0,+8) {\scalebox{2.8}{\color{black} $\rps$}};
			
			%\node[left] at (8,-8.5) {\scalebox{2}{$\Ds_3$}};
			
			\node[above] at (2,5.5) {\scalebox{2.6}{$\Ss_1$}};
			
			% \Ss_\lambdas
			\draw[line width=.5mm] (-7,-7) -- (-3.5 ,3.5) -- (7,7) -- (7,-7) -- (-7,-7);
			
			\node[left] at (-7.5,-7) {\scalebox{2.6}{$v^{2,2}$}};
			\node[above] at (0,.5) {\scalebox{2.6}{$O$}};
			\node[right] at (7.5,7) {\scalebox{2.6}{$v^{1,2}$}};
			\node[left] at (-4,3.5) {\scalebox{2.6}{$v^{1,3}$}};
			\node[right] at (7.5,3.5) {\scalebox{2.6}{$v^{4,3}$}};
			\node[right] at (7.5,-7) {\scalebox{2.6}{$v^{2,3}$}};
			\node[below] at (-3.5,-7.5) {\scalebox{2.6}{$v^{3,3}$}};

			\draw[line width=.4mm, blue] (-6.5,6.5) -- (8.5,-8.5);
			\draw[line width=.4mm, blue] (-8.5,-8.5) -- (8,8);
			\node[right] at (-8,-8.5) {\scalebox{3}{$\color{blue}\Ds_2$}};
			\node[left] at (8,-8.5) {\scalebox{3}{$\color{blue}\Ds_3$}};
			\end{tikzpicture}}
	\end{center}
	\caption{The lines $\cD_{2}=\{a v^{1,2}:\, a\in \R\}$ and $\cD_3=\{a v^{1,3}:\, a\in \R\}$ delimit the plane into $4$ regions. These regions have a natural biological interpretation in terms of dominant($\star$)/recessive($\cdot$) positive(green)/negative(red) selection in a diploid model.}
	\label{fig:exampleCsdm22}
\end{figure}
For the sake of illustration we examine here the case~$m=3$ more closely. The polynomial $d$~is of the form 
\[d(x) = x(1-x)(Ax+B). \]
Our goal is to explicitly construct from such a polynomial the set of its $b$-minimal selection decompositions. The basic idea is to follow Algorithm~\ref{teo:algorithm}. We first determine the extremal selection decompositions. In this setting the deterministic colouring rules are 
\begin{center}
	\begin{minipage}{0.3\textwidth}
		\begin{align*}
			p^{1,2}&\coloneqq\left(0,1,1\right), \\
			p^{2,2}&\coloneqq\left(0,0,1\right),
	\end{align*}\end{minipage} \begin{minipage}{0.15\textwidth} and
	\end{minipage}  \begin{minipage}{0.4\textwidth} \begin{align*}
			p^{1,3}&\coloneqq\left(0,0,1,1\right),  \qquad &\ p^{3,3}\coloneqq\left(0,0,0,1\right),  \\
			p^{2,3}&\coloneqq\left(0,1,0,1\right), \qquad &p^{4,3}\coloneqq \left(0,1,1,1\right).
		\end{align*}
	\end{minipage}
\end{center} Proposition~\ref{prop:description-s-lambda} with~$\lambda=1$ yields the Bernstein coefficient vectors of the extremal selection decompositions. More precisely, for $\ell\in\,]3]$ and $i\in[2^{\ell-1}]$, $v^{i,\ell}\coloneqq (\theta_\ell(p^{i,\ell})-u_m)/(\ell-1)$ so that

\begin{center}
	\begin{minipage}{0.3\textwidth}
		\begin{align*}
			v^{1,2}&\coloneqq\left(\ \ \frac{1}{3},\ \ \frac{1}{3}\right), \\
			v^{2,2}&\coloneqq\left(-\frac{1}{3},-\frac{1}{3}\right),
	\end{align*}\end{minipage} \begin{minipage}{0.15\textwidth} and
	\end{minipage}  \begin{minipage}{0.4\textwidth} \begin{align*}
			v^{1,3}&\coloneqq\left(-\frac{1}{6},\ \ \frac{1}{6}\right),  \qquad &\ v^{3,3}\coloneqq\left(-\frac{1}{6},-\frac{1}{3}\right),  \\
			v^{2,3}&\coloneqq\left(\ \ \frac{1}{3},-\frac{1}{3}\right), \qquad &v^{4,3}\coloneqq \left(\ \ \frac{1}{3},\ \ \frac{1}{6}\right).
		\end{align*}
	\end{minipage}
\end{center}
Hence, by Proposition~\ref{prop:description-s-lambda},
\begin{align*}
	\cS_1^2=\conv{v^{1,2},v^{2,2}},\qquad \cS_1^3=\conv{v^{1,3},v^{2,3},v^{3,3},v^{4,3}},\quad \text{and }\quad \cS_1=\conv{v^{1,2},v^{2,2},v^{1,3},v^{2,3} }.
\end{align*}
See also Fig.~\ref{fig:exampleCsdm2} for an illustration of~$\Ss_\lambda$, and the just-defined vectors.

\smallskip

To derive the $b$-minimal selection decompositions, it is convenient to split the analysis into the following parameter regions. 
\begin{align*}
	\rps&\coloneqq\{(x,y)\in \R^2: y\geq \lvert x\rvert \},\qquad  && \rns\coloneqq\{(x,y)\in \R^2: y\leq -\lvert x\rvert \},\\
	\dps&\coloneqq\{(x,y)\in \R^2: \lvert y \rvert\leq x \}, \qquad  &&\dns\coloneqq\{(x,y)\in \R^2:\lvert y\rvert\leq -x\}.
\end{align*}
See also Fig.~\ref{fig:exampleCsdm22} for an illustration of these sets. We refer to the regions as \emph{recessive-positive}~($\rps$), \emph{recessive-negative}~($\rns$),
\emph{dominant-positive}~($\dps$), and
\emph{dominant-negative}~($\dns$), respectively. The terminology will be justified in Section~\ref{subsec:classicexamplesregions}, where we relate the regions to diploid Wright--Fisher models. A straightforward computation yields for the Bernstein coefficient vector \[\rho(d)=(a,b)\quad \text{with}\quad a= \frac{B}{3}\quad \text{and}\quad b = \frac{A+B}{3}.\]
We now distinguish different cases depending on the region that contains $\rho(d)$.

\subsubsection{Region $\dns$, $\rps$}\label{sec:uniqueregions}  Assume $(a,b)\in \dns$. Then $\mathbb{L}_d\subset \dns$, and $\mathbb{L}_d$ intersects the face of $\Ss_1$ in $\dns$. In particular, $\rho(\bar{d})$ belongs to this face. More specifically, a straightforward calculation yields \[\rho(\bar{d}) = \frac{-2}{9a-3b}(a,b).\]
By Proposition~\ref{prop:scaling-polu} (alternatively Step $1$ and Step $2$ of Algorithm \ref{teo:algorithm}), the minimal effective branching rate is
\[b_{\star}(d) \ = \ \frac{3}{2}(b-3a).\]
Let us now proceed to Step 3. The only $1$-convex decomposition 
$(\vec{v},\alpha)$ such that 
\[\rho(\bar{d}) = \alpha_2 v_2 + (1-\alpha_2) v_3,\]
for $v_\ell\in \cS_1^\ell$, $\ell=2,3$, arises if $v_\ell$ is $\ell$-extremal in the face of $\Ss_1$ in $\dns$, i.e. $v_2=v^{2,2}$ and $v_3=v^{1,3}$, see Fig.~\ref{fig:exampleCsdm200}. Using this and the expression for~$\rho(\bar{d})$ yields
\[\alpha_2 = \frac{a+b}{3a-b}, \ \ \alpha_3=1-\alpha_2.\]
Hence, $\ccS_{\bar{d}}^{\min}=\{(\beta^{(1,\vec{v},\alpha)},p^{(1,\vec{v},\alpha)})\}$, where \[\beta^{(1,\vec{v},\alpha)}=\left(\alpha_2,\frac{\alpha_3}{2}\right)\quad \text{and}\quad  p^{(1,\vec{v},\alpha)}=(p^{2,2},p^{1,3}).\]
Finally, by Algorithm \ref{teo:algorithm} we have $\ccS_d^{\min}=\{(\beta,p)\}$ with 
\[ \beta_\ell = b_{\star}(d)\left(\alpha_2,\frac{\alpha_3 }{2}\right) \quad \text{and}\quad  p=(p^{2,2},p^{1,3}).\]
If~$(a,b)$ belongs to $\rps$, a symmetry argument exposes that $b_{\star}(d)$ and $\alpha$ are obtained from the case when the point belongs to~$\dns$ by the transformation $(a,b)\mapsto(-b,-a)$ and by using $v^{1,2}$ and $v^{1,3}$ for the convex combination of $\rho(\bar{d})$.

\subsubsection{Region $\rns$, $\dps$}\label{sec:nonuniqueregions}  Assume $(a,b)\in \rns$. Then $\mathbb{L}_d\subset \rns$ and $\rho(\bar{d})$ belongs to the face of $\Ss_1$ in $\rns$. More precisely, 
\[\rho(\bar{d}) = -\frac{1}{3} \ \left(\frac{a}{b},1\right).\]
We use again Proposition~\ref{prop:scaling-polu} and obtain the minimal effective branching rate as
\[b_{\star}(d) \ = \ -\frac{b}{3}.\]
In contrast to $\rho(d)\in \dns\cup\rps$, here there are multiple $1$-convex decompositions of $\rho(\bar{d})$.
The set of $1$-convex decompositions of $\rho(\bar{d})$ is of the form
\[v_2 = v^{2,2}, \ \ v_3 \ = \ \left(x_0,-\frac{1}{3}\right) \quad \text{with} \quad x_0\in \left[\left(-\frac{1}{6}\vee -\frac{b}{3a}\right), \frac{1}{3}\right], \]
and 
\[\alpha_2=\frac{1}{a}\frac{b + 3 ax_0}{1+ 3 x_0}.\]
For fixed $x_0$, using $\theta_3^{-1}$ we obtain $p^{x_0}=(0,2 x_0+\frac{1}{3},0,1)$. Consequently, the corresponding selection decomposition is given by
\[ \beta_\ell = b_{\star}(d)\left( \alpha_2, \frac{\alpha_3 }{2}\right)\quad \text{and} \quad p=\left(p^{1,2},p^{x_0}\right). \] 
For $(a,b)\in \dps$, one easily recovers the $b$-minimal selection decomposition by symmetry.

\subsubsection{Classic examples for m=2,3}\label{subsec:classicexamplesregions}
Let us consider several examples from the literature that clarify the names of the various regions.  

\bigskip

The \emph{haploid Wright--Fisher diffusion with genic selection} arises from \eqref{eq:SDEoriginal02} if $\Lambda=\delta_0$ and \[d(x)  = -\sigma x(1-x),\] with $\sigma>0$. Here, type~$a$ (resp. type~$A$) is called unfit (resp. fit); and the (absolute) fitness difference is given by~$\sigma$. Note that $\rho(d)= \sigma v^{2,2}$ belongs to the~$1$-dimensional subspace ${\dps}\cap \, \rps$. Following Algorithm~\ref{teo:algorithm}, we obtain $\rho(\bar{d})=v^{2,2}$ and $b_{\star}(d)=\sigma$. In particular, $\rho(\bar{d})=v_2\alpha_2$ with
	$v_2=v^{2,2}$ and $\alpha_2=1$. Hence, the $b$-minimal selection decomposition is unique and given by $(\beta,p)\in \R_+\times \Ps_2$ with $\beta_2=\sigma$ and $p_2=(0,0,1)$. This is consistent with the classical ASG of~\citeauthor{KroNe97}~\citep{KroNe97,NeKro97}.

\bigskip

The \emph{diploid Wright--Fisher diffusion} arises from \eqref{eq:SDEoriginal02} if $\Lambda=\delta_0$ and the drift term is
	\begin{equation}\label{eq:wf-dominance}
		d(x) = \sigma x(1-x)(x+h(1-2x)),
	\end{equation}
	where~$\sigma,h\in \R$. In particular, $\rho(d)=(0,\sigma h,\sigma(1-h),0)/3$. This process approximates the frequency of type~$a$ in the large population limit of a diploid Wright--Fisher model with weak selection, where the pairs $aa$, $aA$ (and $Aa$), and $AA$ have relative fitness~$\sigma$, $\sigma h$, and $0$, respectively, see e.g.~\citep[Ch.~5]{ewens2004mathematical}. If $\sigma>0$, then type~$a$ is fitter than~$A$, and hence type~$a$ is subject to positive selection. On the other hand, if $\sigma<0$, type~$a$ is less fit than type~$A$ and hence is subject to negative selection. $\lvert \sigma \rvert$ yields the selection strength, and~$h$ is the dominance parameter, i.e. $h$ quantifies the contribution of $a$ to the fitness of an heterozygote. When $h=1/2$, selection is said to be additive (and agrees with the drift in the case of genic selection). This is sometimes also referred to as the case with \emph{no dominance}, because none of the two alleles dominates the other one. If $h<1/2$, then $a$ is \emph{recessive}. (It is completely (resp. partially) recessive if $h=0$ (resp. $h\in (0,1/2)$.) If $h>1/2$, then $a$ is \emph{dominant}. (It is completely (resp. partially) dominant if $h=1$ (resp. $h\in (1/2,1)$.) A direct calculation shows that when $\sigma<0$, $\rho(d)$ belongs to region $\rns$ if $h<1/2$, and to region $\dns$ if $h>1/2$. When $\sigma>0$, $\rho(d)$ belongs to region $\rps$ if $h<1/2$ and it belongs to region $\dps$ if $h>1/2$. In conclusion, each face of the polytope~$\Ss_1$ corresponds to a $b$-minimal selection decomposition of a model with recessive/dominant positive/negative selection. Note that the analysis in Sections~\ref{sec:uniqueregions} and~\ref{sec:nonuniqueregions} implies that in the regime of recessive positive, and dominant negative selection, there is a unique $b$-minimal selection decomposition; whereas there are infinitely many possible choices in the regime of dominant positive and recessive negative selection. It would be interesting to investigate if there is any biological meaning of this multiplicity.

\bigskip

A special case of the diplod model is the \emph{Wright--Fisher diffusion with balancing selection}. It has the drift term
	\[d(x)=x(1-x)(1-2x).\]
	Here, $\rho(d)=v^{2,3}$ so that $v_3=v^{2,3}$ and $\alpha_3=1$. This leads to $\beta=(0,1)$ and $p_3=(0,1,0,1)$. This is consistent with the duality obtained in \citet{Ne99}. The particular form of the colouring rule~$p$ is also called the minority rule. \cite{Ne99} shows that the ASG generated from this selection decomposition has a natural interpretation in terms of the genealogy of a diploid population model.
\begin{remark}
	If $m=3$, there is always a $b$-minimal selection decomposition with a deterministic colouring rule. In regions $\dns$ and $\rps$, this is automatic since $v_2$ and $v_3$ are Bernstein coefficient vectors of selection decompositions with deterministic colouring rules. In region $\rns$ (resp. $\dps$), this corresponds to choose as $v_3$ one of the corners of $\Ss_1^3$, i.e. either $v_3=v^{2,3}$ or $v_3=v^{3,3}$ (resp. $v_3=v^{4,3}$) (if the latter is permitted, i.e. when $\rho(\bar{d})$ lies in between $v^{2,2}$ and $v^{3,3}$). Then we have a deterministic colouring rule, i.e. the colour of the roots are a deterministic function of the ASG and the colouring of the leaves. It is tempting to conjecture that also in higher dimensions there is always a $b$-minimal selection decomposition with deterministic colouring rules. According to Step 3 of our algorithm, this is equivalent to say that each face of $\Ss_1$ contains at most one vertex on each $\Ss_1^\ell$, $\ell\in\,]m]$ (because each extreme point of a given face represents a deterministic $\ell$-colouring rule, with a distinct $\ell$). However, numerical simulations for $m=4$ suggest that this is not the case, see Fig.~\ref{fig:convexhull-m=4}.
\end{remark}
\begin{figure}[t]
	\hspace{-3cm}
	\begin{minipage}{0.45\textwidth}
		\begin{center} \scalebox{.55}{\includegraphics{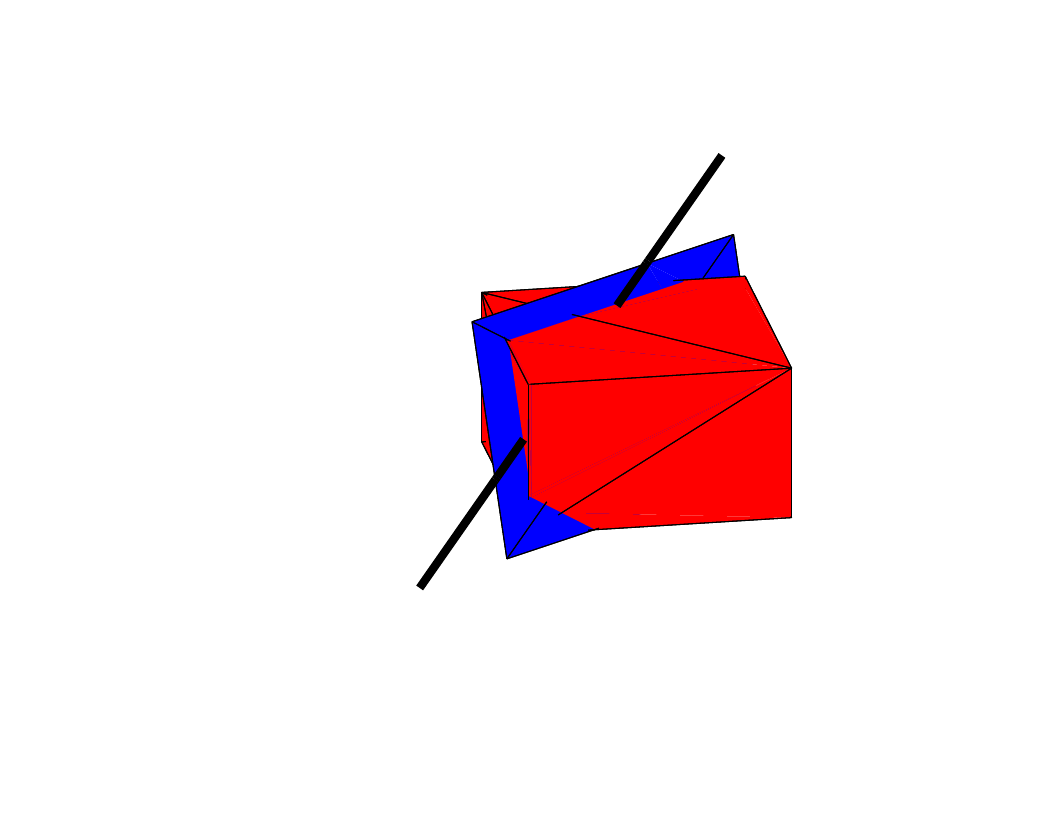}}\end{center}
	\end{minipage}\hspace{1.5cm}
	\begin{minipage}{0.5\textwidth}
		\begin{center} \scalebox{.55}{\includegraphics{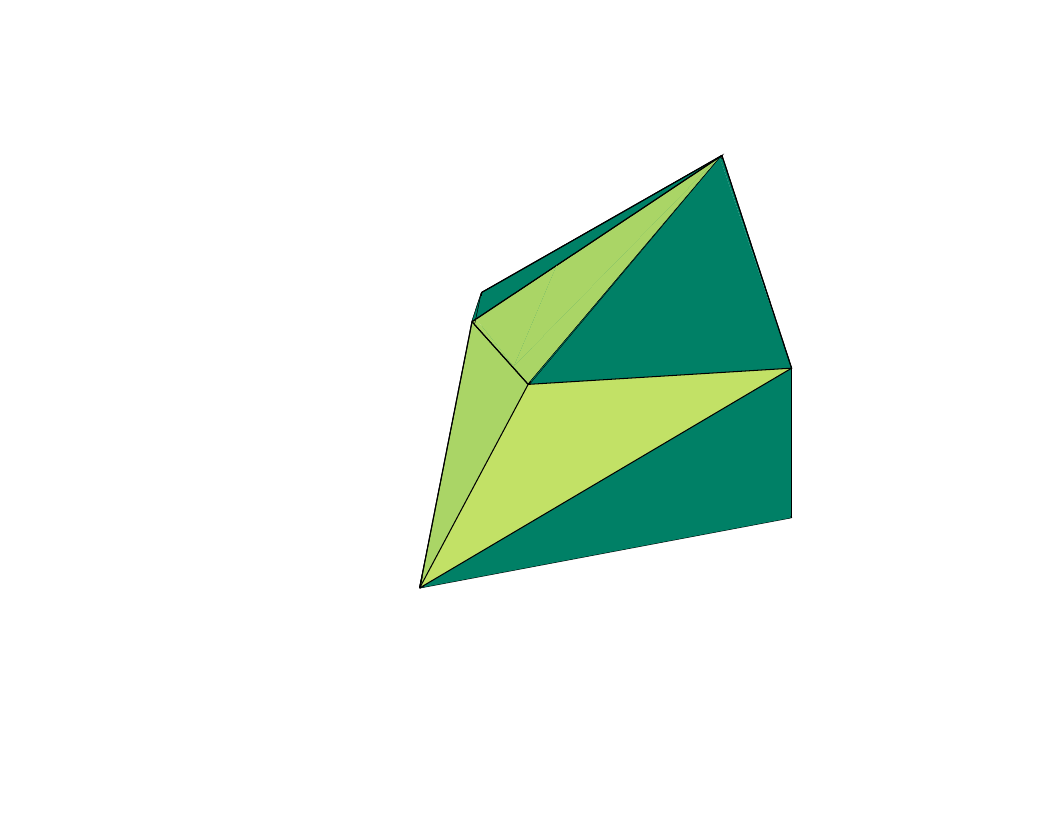}}\end{center}%\hspace{6cm}
	\end{minipage} 
	\vspace{-2.5cm}
	
	\caption{Left: $\cS_1^2, \cS^3_1, \cS^4_1$ when $m=4$.  Right: $\cS_1$ when $m=4$. The front triangular face (bottom, dark green) has two vertices corresponding to a deterministic colouring rule of order $4$, i.e. in $\Ps_{0,1}^4$. As a consequence, if $\rho(d)$ belongs to the interior of this face, there is no $b$-minimal selection decomposition with deterministic colouring rules.}
	\label{fig:convexhull-m=4}
\end{figure}
\subsection{Graph-minimal selection decompositions}
\label{sect:minimal-models}
A possible order of ASGs arises from an order of the underlying branching-coalescing systems. At least intuitively, one way to order the systems is by saying that a system is smaller than another one if it arises by removing lines from the larger one. In this section we formalise this intuition and prove that such thinnings induce an order of branching-coalescing systems. We also prove that this order induces a partial order of the corresponding effective branching rates. Throughout this subsection $\Lambda\in \Ms_f^*([0,1])$ is fixed.

\smallskip

Recall that a thinning mechanism acts on a branching rate vector~$\beta\in\R_+^{m-1}$ via
\begin{equation}\label{eq:linear-thinning}
	(\beta)_{\ell=2}^m\mapsto (\Tb \beta)_{\ell=2}^m,\quad\textrm{with}\quad (\Tb \beta)_\ell= 
	\sum_{k=\ell}^m  \beta_k \Tb_{k,\ell},
\end{equation}
where $\Tb=\{\Tb_{{k,i}}\}_{k,i=1}^m$ is a lower-triangular stochastic matrix, see Definition \ref{def:thinning}. The following definition formalises the idea of removing lines in a branching-coalescing system.
\begin{definition}[$\Tb$-thinning of a branching-coalescing system]
	Let $\Gs=(\Gs_t: \, t\geq 0)$ be the branching-coalescing system with parameters $(\beta,\Lambda)$, see Section~\ref{s4}, and let $\Tb\coloneqq\{\Tb_{{k,i}}\}_{k,i=1}^m$ be a thinning mechanism. Conditional on~$\Gs$, define the branching-coalescing system $\Tb \Gs\coloneqq(\Tb \Gs_t:\, t\geq 0)$ dynamically according to the following random procedure. Independently at every $k$-branching in $\Gs$, for $i\in[k]$ with probability $\Tb_{k,i}$, remove $k-i$ particles chosen uniformly at random among the $k-1$ new ones. The marked particle giving rise to the branching event is never removed and the remaining particles are kept. We call~$\Tb\Gs$ the \emph{$\Tb$-thinned} version of~$\Gs$. 
\end{definition}
It follows from the consistency of the rates of the $\Lambda$-coalescent that~$\Tb \Gs$ is distributed as a branching-coalescing system with (unchanged) coalescence mechanism $\Lambda$ and branching mechanism $(\Tb\beta_\ell)_{\ell=2}^m$. The next result formalises this.
\begin{proposition}\label{prop:graphminimal}
	Let $(\Gs_t:\, t\geq0)$ and $(\Gs_t^\star:t\geq0)$ be the branching-coalescing particle systems with parameters $(\beta,\Lambda)$ and $(\Tb\beta,\Lambda)$, respectively. If both branching-coalescing systems start with the same number of particles, then 
	\[ \forall t\geq0, \ \ \cG_t^\star\, \overset{(d)}{=}\,\Tb \cG_t \subseteq \cG_t.\]
\end{proposition}
Using Proposition~\ref{prop:graphminimal}, the notion of graph-minimality of selection decompositions given in Definition~\ref{def:graphminimal} can be expressed at the level of the branching-coalescing systems.
In particular, the selection decomposition $(\beta,p)\in\ccS_d$ is graph-minimal if and only if there are no superfluous (dummy) branches in the corresponding ASG.  

\smallskip 

The next proposition states that thinnings of branching-coalescing systems are closely related to a particular partial order $\preceq$ on $\R_+^{m-1}$. A selection decomposition is smaller than another one if the branching-coalescing system generated by the smaller selection decomposition can be obtained by a $\Tb$-thinning from the system corresponding to the larger selection decomposition. More precisely, for any $\beta, \beta' \in \R^{m-1}_+$, write 
$\beta'  \preceq \beta$ if and only if for all $k\in \, ]m]$ 
\begin{equation} \label{eq:ineq-betas}
	\ \  \sum_{j=k}^m \beta_j'  \ \leq \  \sum_{j=k}^m \beta_j.
\end{equation}
Moreover, we write $\beta'\prec \beta$ if and only if $\beta'\preceq \beta$ and $\beta'\neq \beta$. Loosely speaking, for two population models with interaction rates $\beta'$ and $\beta$, respectively, $\beta'\prec \beta$ means that in the $\beta'$-population model for every $k\in\,]m]$, individuals participate in an interactive event that includes~$k$ or more individuals less frequently than in a $\beta$-population model. The next proposition relates the partial order with thinnings.
\begin{proposition}\label{prop:order-beta-order-graph}
	Let $\beta',\beta \in \R^{m-1}_+$, $\Lambda\in\Ms_f^*([0,1])$, and~$n\in \N$. Let $\Gs'$ and~$\Gs$ be the branching-coalescing particle systems constructed from the pairs $(\beta',\Lambda)$
	and $(\beta,\Lambda)$, both system starting with $n$ particles. Then
	$\beta' \prec \beta$ if and only if there exists a thinning mechanism $\Tb$ different from the identity such that 
	$\Gs' \ \overset{(d)}{=} \ \Tb \Gs.$
\end{proposition}

\begin{proof}
	It is immediate to check from \eqref{eq:linear-thinning} that $\Tb \beta\prec\beta$ if $\Tb$ is not the identity. Let us now show the converse, i.e. assuming that 
	$\beta'\prec \beta$, we construct a thinning (different from the identity) of $\Gs$ of the branching-coalescing system generated by~$(\beta,\Lambda)$ such that it is distributed as the system $\Gs'$ generated by $(\beta',\Lambda)$ for~$\Lambda\in \Ms_f^*([0,1])$.
	Fix $\beta',\beta\in \R_+^{m-1}$ with $\beta'\prec \beta$. We construct a sequence of thinning mechanisms $(\Tb^{(k)})_{k=0}^{m-1}$, and a sequence of selection rate vectors $(\beta^{(k)})_{k=0}^{m-1}$ with $\beta^{(0)}\coloneqq \beta$ and $\Tb^{(0)}$ the identity such that for $k\in [m-1]$, (1) $\beta^{(k)}=\Tb^{(k)}\beta^{(k-1)}$, (2) $\Tb^{(k)}$ acts only on $(m-k+1)$-branchings, and (3) $\beta_i^{(k)}=\beta_i'$ for $i\geq m-k+1$. Once we finished the construction, the result follows since $\beta'=\Tb\beta$ with $\Tb\coloneqq\Tb^{(m-1)}\cdots \Tb^{(0)}$.
	We define the sequences recursively as follows. For $k\in[m-1]$, set
	\begin{align*}
	\Tb^{(k)}_{m-k+1,m-k+1}  \coloneqq  \frac{\beta_{m-k+1}'}{\beta_{m-k+1}^{(k-1)}},\quad \Tb^{(k)}_{m-k+1,m-k}  \coloneqq  \bigg(1-\frac{\beta_{m-k+1}'}{\beta_{m-k+1}^{(k-1)}}\bigg),\quad \Tb^{(k)}_{j,j}=1 \text{ if }j\neq m-k+1, 
	\end{align*}
	and all other entries of $\Tb^{(k)}$ are zero. Define $\beta^{(k)}$ according to (1). $\Tb^{(k)}$ thins at most one branch in every $m-k+1$ branching, thus increasing the rate of~$m-k$ branchings. Hence, (2) is satisfied by construction. It remains to verify (3) and that $\Tb^{(k)}$ is a thinning mechanism. We verify this via induction by proving that for all $k\in [m-1]$, for $i>m-k+1$, $\beta_i'=\beta_i^{(k)}$ and $\beta_{m-k+1}'\leq \beta_{m-k+1}^{(k-1)}$. For $k=1$, we have by definition, $\beta_m'=\beta_m^{(1)}$, and by \eqref{eq:ineq-betas}, $\beta_m'\leq \beta_m$. Assume the claim holds for $k<m-1$. First note that, for $i>m-k$, $\beta_i^{(k+1)}=\beta_i^{(k)}=\beta_i'$. Furthermore, from the definition of $\Tb^{(k+1)}$, we have  $$\beta_{m-k}^{(k+1)}=(\Tb^{(k+1)}\beta^{(k)})_{m-k}=\beta_{m-k}'.$$ It remains to prove that $\beta_{m-k}'\leq \beta_{m-k}^{(k)}$. By the induction hypothesis, $\Tb^{(k)}$ is a thinning. Hence, $\beta^{(k+1)}\preceq \beta^{(k)}$. In particular, $$\sum_{j=m-k}^m \beta_j'=\sum_{j=m-k}^m \beta_j^{(k+1)}\leq \sum_{j=m-k}^m \beta_j^{(k)}=\beta^{(k)}_{m-k}+\sum_{j=m-k+1}^m \beta_j',$$
	which proves the claim.
\end{proof}
\begin{proof}[Proof of Theorem \ref{prop:graph:minimal}]
	By Proposition \ref{prop:order-beta-order-graph}, it suffices to show that $\beta'\prec \beta$ implies $b(\beta')<b(\beta)$. 
	This is easily seen by summing the inequalities in \eqref{eq:ineq-betas}. 
\end{proof}

\subsection{\texorpdfstring{Equivalence of minimality if m=3}{Equivalence of minimality if m=3}}\label{sec:equivalence}
\label{sect:minimal-models3}
Theorem \ref{prop:graph:minimal} states that every $b$-minimal selection decomposition is graph-minimal. In dimension $m=2$, i.e. when $\deg(d)=2$, it is straightforward to see that both notions of minimality agree. In higher dimensions the question of equivalence of the two notions is more involved. In this section we prove the equivalence in dimension $m=3$. Throughout, let $d$ be a polynomial vanishing at the boundary of the unit interval.

\smallskip

We begin with an observation that is applicable in any dimension $m\geq 2$. Consider $\lambda,\lambda^{\star}\in\R_+$ with $\lambda^{\star}<\lambda$ and $\vec{v}=(v_2,\ldots, v_m)\in\prod_{\ell=2}^m \Ss_{\lambda^{\star}}^\ell$. By Proposition~\ref{prop:s-polytope} we also have $\vec{v}\in\prod_{\ell=2}^m \Ss_{\lambda}^\ell$. In a convex combination of~$\rho\in \Ss_{\lambda^{\star}}$ encoded by $(\lambda,\vec{v},\alpha)\in \Cs$, i.e. $\rho=\sum_{\ell=2}^m \alpha_\ell v_\ell$, the extremal points of $\Ss_\lambda^{\ell}$ act as the reference points. This leads to a (unique) selection decomposition with effective branching rate~$\lambda$. We can represent $\rho$ also as a convex combination using the same vectors $\vec{v}$, but with the extremal points of the (smaller) sets $\Ss_{\lambda^\star}^\ell$ acting as reference points. The latter convex combination is encoded by $(\lambda^{\star},\vec{v},\alpha)\in \Cs$, leading to a selection decomposition with effective branching rate~$\lambda^{\star}$. By shrinking the reference frame, we decrease the effective branching rate and at the same time obtain a thinner ancestral structure. See Fig.~\ref{fig:shrikingMoving} (left) for an illustration of this shrinking operation. The next result formalises this idea.

\begin{proposition}[Shrinking the polygon] \label{prop:shrinkinglamd}
	Consider $(\lambda,\vec{v},\alpha)\in\Cs_d$. Set 
	$$\lambda^{\star}(\vec{v})\coloneqq\inf\{\gamma\geq 0: v_\ell \in \Ss_{\gamma}^{\ell},\,\,\forall \ell\in \, ]m]\}.$$ Then $(\lambda^{\star}(\vec{v}),\vec{v},\alpha)\in \Cs_d$ and $\beta^{(\lambda^{\star}(\vec{v}),\vec{v},\alpha)}\preceq \beta^{(\lambda,\vec{v},\alpha)}$.
\end{proposition}
\begin{proof}
	Assume $\lambda\neq \lambda^{\star}(\vec{v})$. Since, $v_\ell\in \Ss_{\lambda^{\star}(\vec{v})}^{\ell}$ and $\rho=\sum_{\ell=2}^{m}\alpha_\ell v_{\ell}$ the first statement follows. The second statement follows by the definition of $\beta^{(\lambda,\vec{v},\alpha)}$.
\end{proof}

For the remaining part, assume $m=3$. We now describe a technique that allows to identify for a given selection decomposition, a smaller one (with respect to $\prec$) if $\deg(d)=3$. We continue to use the notation of Section~\ref{sec:minimalSDm3}. In particular, $v^{1,2},v^{2,2}\in \Ss_1^2$ and $v^{1,3},v^{2,3}\in \Ss_1^3$ are the extremal points of~$\Ss_1$. Denote by 
$$\Ds_2\coloneqq\{av^{1,2}:\,a\in \Rb\}\quad\textrm{and}\quad \Ds_3\coloneqq\{av^{1,3}:\,a\in \Rb\},$$
see also Fig.~\ref{fig:regions}. 
Consider $\rho=\rho(d)\in \Ss_1$ and $\vec{v}=(v_2,v_3)\in \Ss_{\lambda^\star}^2\times \Ss_{\lambda^\star}^3$ for some $\lambda^\star>1$ such that $(\lambda,\vec{v},\alpha)\in \Cs_d$. 
Recall from Remark~\ref{rem:interpretationconvexm3} that the weight~$\alpha_2$ of a convex combination of $\rho$ encoded by $(\lambda,\vec{v},\alpha)$ can be read off from our figures (Figs.~\ref{fig:exampleCsdm200}
and~\ref{fig:shrikingMoving}) as the relative distance between $v_3$ and $\rho$ compared to the distance between $v_3$ and $v_2$. If we shift $v_3$ along the parallel of~$\Ds_2$ that passes through~$v_3$, we either hit~$\Ds_3$ or $\Ss_1$. Denote by $v_3'$ the first hitting point of one of these two sets. We can move $v_2$ along $\Ds_2$ to a new location $v_2'$ such that the relative distance between $v_2'$ and $\rho$ compared to the distance between~$v_2'$ and~$v_3'$ is again $\alpha_2$, i.e. $\rho=\alpha_2v_2+\alpha_3v_3=\alpha_2v_2'+\alpha_3v_3'$. We can shift the vectors in such a way that allows us to change the reference frame of the convex combination to a frame with smaller effective branching rate. More precisely, we can shift the points such that the shrinking procedure of Proposition~\ref{prop:shrinkinglamd} becomes applicable. For an illustration of this shifting operation, see also Fig.~\ref{fig:shrikingMoving} (right). The next result formalises this idea in the case $v_3$ is not a corner point of $\Ss_{\lambda^\star}$. \begin{proposition}[Shifting~$\vec{v}$]\label{prop:shiftingVec}
	Consider $(\lambda,\vec{v},\alpha)\in\Cs_d$. Assume that $b_{\star}(d)=1$ and that $\lambda^\star\coloneqq\lambda^\star(\vec{v})>1$. If $v_3\notin\{\lambda^{\star}v^{1,3},\lambda^{\star}v^{2,3}\}$, then there exists~$\lambda'<\lambda^{\star}$ and~$\vec{v}'$ such that
	\begin{itemize}
		\item $(\lambda',\vec{v}',\alpha)\in \Cs_d,$
		\item $\beta^{(\lambda',\vec{v}',\alpha)}\preceq \beta^{(\lambda^{\star},\vec{v},\alpha)}$,
		\item either $\vec{v}'\in \{\lambda' v^{1,3},\lambda' v^{2,3}\}$ or $v_i'\in F_\rho$ for $i\in\,]3]$,
	\end{itemize}
	where $F_\rho$ is the face of $\Ss_1$ containing $\rho(d)$. 
	
\end{proposition}
\begin{remark}
	$v_3\notin\{\lambda^{\star}v^{1,3},\lambda^{\star}v^{2,3}\}$ means that $v_3\in \cS_{\lambda^{\star}}^3$ is not a corner point $\cS_{\lambda^{\star}}$.
\end{remark}
\begin{proof}[Proof of Proposition~\ref{prop:shiftingVec}]
	Since $\lambda^\star>1$, it follows that $\alpha_2\in(0,1)$. For $a\in\Rb$, define $\vec{v}(a)\coloneqq(v_2(a),v_3(a))$ via
	$$v_3(a)\coloneqq v_3+a v^{1,2}\quad\textrm{and}\quad v_2(a)\coloneqq v_2-\frac{\alpha_3 }{\alpha_2} av^{1,2}.$$
	Note that $\alpha_2 v_2(a)+\alpha_3 v_3(a)=\rho(d)$. Hence,  $(\lambda^\star,\vec{v}(a),\alpha)\in\Cs_d$ if and only if $\vec{v}(a)\in\Ss_{\lambda^\star}^2\times\Ss_{\lambda^\star}^3$. Since $\Ds_2\perp\Ds_3$ and $v_3\notin\{\lambda^{\star}v^{1,3},\lambda^{\star}v^{2,3}\}$, there is a unique $a_0\in\Rb\setminus\{0\}$ such that $v_3(a_0)\in\Ds_3\cap \mathrm{int}(\Ss_{\lambda^\star}^3)$. We now split the analysis in two cases: (i) $v_3(a_0)\notin \Ss_1^3$, and (ii) $v_3(a_0)\in \Ss_1^3$.
	In case (i), since any line from $\Ds_3\cap (\Ss_1^3)^C$ to $(\Ss_1^2)^C$ lies outside of $\Ss_1$, and $\rho(d)\in \Ss_1$ is a convex combination of $v_2(a_0)$ and $v_3(a_0)$, we infer that $v_2(a_0)\in \Ss_1^2$.
	In particular, $\vec{v}(a_0)\in\Ss_{\lambda^\star}^2\times\Ss_{\lambda^\star}^3$. Hence, $(\lambda^\star,\vec{v}(a_0),\alpha)\in\Cs_d$ and $\beta^{(\lambda^\star,\vec{v}(a_0),\alpha)}=\beta^{(\lambda^\star,\vec{v},\alpha)}$. Moreover, $v_i(a_0)\in\mathrm{int}(\Ss_{\lambda^\star}^i)$ for $i\in\,]3]$. Therefore,  $\lambda'\coloneqq\lambda^\star(\vec{v}(a_0))<\lambda^\star$, and then setting $\vec{v}'\coloneqq\vec{v}(a_0)$, the result follows from Proposition \ref{prop:shrinkinglamd}. In case (ii), there is a unique $a_1\in\Rb$ such that $v_3(a_1)\in F_\rho$. Since $F_{\rho}$ is the face that contains~$\rho(d)$, we conclude that also~$v_2(a_1)\in F_{\rho}$. As before, we obtain that $(\lambda^\star,\vec{v}(a_1),\alpha)\in\Cs_d$ and $\beta^{(\lambda^\star,\vec{v}(a_1),\alpha)}=\beta^{(\lambda^\star,\vec{v},\alpha)}$. The result follows from Proposition \ref{prop:shrinkinglamd} by setting $\lambda'\coloneqq\lambda^\star(\vec{v}(a_1))<\lambda^\star$ and $\vec{v}'\coloneqq\vec{v}(a_1)$.
\end{proof}

\begin{figure}[t]
	\begin{minipage}{0.45\textwidth}
		\begin{center}
			\scalebox{0.4}{
				\begin{tikzpicture}
				%\Ss^3
				\fill [gray,opacity=0.2] (-2.5,-5) rectangle (5,2.5);
				\fill [gray,opacity=0.2] (-1.66,-3.33) rectangle (3.33,1.66);
				\fill [gray,opacity=0.2] (-3.5,-7) rectangle (7,3.5);
				%-1.02x+0.16=y
				%-2/5+x=y
				%-147/55x+391/550=y
				%Box

				%\Ss^2
				\draw[line width=.5mm, opacity=0.2] (-3.33,-3.33) -- (3.33,3.33);
				\draw[line width=.5mm, opacity=0.2] (-5,-5) -- (5,5);
				\draw[line width=.5mm, opacity=0.2] (-7,-7) -- (7,7);

				\draw[line width=.4mm, gray, dotted] (-6.5,6.5) -- (8.5,-8.5);
				\draw[line width=.4mm, gray, dotted] (-8.5,-8.5) -- (8,8);
				\node[right] at (-8,-8.5) {\scalebox{2.6}{$\Ds_2$}};
				\node[left] at (8,-8.5) {\scalebox{2.6}{$\Ds_3$}};
				
				\node[above] at (2,5.5) {\scalebox{2.6}{$\Ss_\lambda$}};
				\node[above] at (2.04,4.1) {\scalebox{2.6}{$\Ss_{\lambda^{\star}}$}};
				\node[above] at (1.9,2.9) {\scalebox{2.6}{$\Ss_{1}$}};
				
				% \Ss_\lambdas
				\draw[line width=.5mm] (-7,-7) -- (-3.5 ,3.5) -- (7,7) -- (7,-7) -- (-7,-7);
				\draw[line width=.5mm] (-5,-5) -- (-2.5 ,2.5) -- (5,5) -- (5,-5) -- (-5,-5);
				\draw[line width=.5mm] (-3.33,-3.33) -- (-1.66,1.66) -- (3.33,3.33) -- (3.33,-3.33) -- (-3.33,-3.33);
				
				\node at (3.33,-1.8) {\scalebox{2.8}{$\times$}};
				\node[right] at (3.43,-1.6) {\scalebox{2.3}{$\rho(d)$}};
				\node at (.25,1) {\scalebox{2.6}{$v_2$}};
				\node at (5.6,-3.5) {\scalebox{2.6}{$v_3$}};
				
				\draw[line width=.5mm, black] (5,-3.5) -- (3.33,-1.8) -- (.79,.79);
				%\draw[line width=.5mm, black] (4.25,-4.25) -- (3.33,-1.8) -- (1.935,1.935);

				\draw [line width=.2mm,-{Latex[length=5mm,width=4mm]}] (-7,-7) -- (-5,-5);
				\draw [line width=.2mm,-{Latex[length=5mm,width=4mm]}] (7,7) -- (5,5);
				\draw [line width=.2mm,-{Latex[length=5mm,width=4mm]}] (7,-7) -- (5,-5);
				\draw [line width=.2mm,-{Latex[length=5mm,width=4mm]}] (-3.5,3.5) -- (-2.5,2.5);
				\end{tikzpicture}}
		\end{center}
	\end{minipage}
	\begin{minipage}{0.45\textwidth}
		\begin{center}
			\scalebox{0.4}{
				\begin{tikzpicture}
				\fill [gray,opacity=0.2] (-2.5,-5) rectangle (5,2.5);
				\fill [gray,opacity=0.2] (-2.5,-5) rectangle (5,2.5);
				\fill [gray,opacity=0.2] (-1.66,-3.33) rectangle (3.33,1.66);
				
				%-1.02x+0.16=y
				%-2/5+x=y
				%-147/55x+391/550=y
				%Box
				
				\draw[line width=.5mm, opacity=0.2] (-3.33,-3.33) -- (3.33,3.33);
				\draw[line width=.5mm, opacity=0.2] (-5,-5) -- (5,5);
				
				\draw[line width=.4mm, gray, dotted] (-6.5,6.5) -- (8.5,-8.5);
				\draw[line width=.4mm, gray, dotted] (-8.5,-8.5) -- (8,8);
				\node[right] at (-8,-8.5) {\scalebox{2.6}{$\Ds_2$}};
				\node[left] at (8,-8.5) {\scalebox{2.6}{$\Ds_3$}};
				
				\draw[opacity=0] (-5,-8.5) -- (-5,7);
				
				% fsm with rate 1
				\draw[line width=.5mm] (-5,-5) -- (-2.5 ,2.5) -- (5,5) -- (5,-5) -- (-5,-5);
				\draw[line width=.5mm] (-3.33,-3.33) -- (-1.66,1.66) -- (3.33,3.33) -- (3.33,-3.33) -- (-3.33,-3.33);
				
				\node at (3.33,-1.8) {\scalebox{2.8}{$\times$}};
				\node[right] at (3.42,-1.6) {\scalebox{2.3}{$\rho(d)$}};
				\node at (0.2,1) {\scalebox{2.6}{$v_2$}};
				\node at (5.55,-3.5) {\scalebox{2.6}{$v_3$}};
				\node at (2.6,1.9) {\scalebox{2.6}{$v_2'$}};
				\node at (3.7,-4.45) {\scalebox{2.6}{$v_3'$}};
				%\node[right] at (3.38,-1.7) {\scalebox{1.5}{$v^{i,\ell}$}};
				
				\draw[line width=.5mm, black] (5,-3.5) -- (3.33,-1.8) -- (.79,.79);
				\draw[line width=.5mm, black] (4.25,-4.25) -- (3.33,-1.8) -- (1.935,1.935);
				
				\node[above] at (2.05,4.1) {\scalebox{2.6}{$\Ss_{\lambda^{\star}}$}};
				\node[above] at (1.9,2.9) {\scalebox{2.6}{$\Ss_{1}$}};
				
				\draw [line width=.2mm,-{Latex[length=5mm,width=4mm]}] (5,-3.5) -- (4.25,-4.25);
				\draw [line width=.2mm,-{Latex[length=5mm,width=4mm]}] (0.79,0.79) -- (1.935,1.935);
				
				\end{tikzpicture}}
		\end{center}
	\end{minipage}
	\caption{Left: Shrinking the Polygon~$\Ss_\lambda$ to $\Ss_{\lambda^{\star}}$ while keeping $\vec{v}$ constant. Right: Moving~$\vec{v}$ into $\mathrm{int}(\Ss_{\lambda^{\star} } )$ along the direction of~$\Ds_2$. By the intercept theorem, this keeps the relative distances to~$\rho(d)$ constant.}
	\label{fig:shrikingMoving}
\end{figure} 

We now combine all possible ways to obtain a smaller selection decomposition (with respect to $\prec$). Furthermore, we generalise Proposition~\ref{prop:shiftingVec} to the case when $v_3$ is a corner point of $\Ss_{\lambda^\star}$.
\begin{proposition}\label{prop:equivalenceminimality1}
	Assume $b_{\star}(d)=1$. Let $(\beta,p)\in \ccS_d$ be a selection decomposition with $b(\beta)=\lambda>1$. Then there exists $(\beta',p')\in \ccS_d$ such that $\beta'\preceq\beta$ and $b(\beta')=1$.
\end{proposition}
\begin{proof}
	Lemma~\ref{lem:varphiBijection} allows us to identify $(\beta,p)$ with $(\lambda,\vec{v},\alpha)\in \Cs_d$. 
	We make the following case distinctions: (1) $v_\ell \in \textrm{int}(\Ss_\lambda^{\ell})$ for all $\ell\in \, ]3]$, (2) there is $\ell$ such that $v_\ell\in \partial\Ss_\lambda^\ell$, but $v_3\notin\{\lambda v^{1,3},\lambda v^{2,3}\}$, and (3) $v_3\in\{\lambda v^{1,3},\lambda v^{2,3}\}$. 
	In case (1), apply the shrinking operation of Proposition~\ref{prop:shrinkinglamd}. This leads to a new triple $(\lambda^{\star},\vec{v},\alpha)$ with $\beta^{(\lambda^{\star},\vec{v},\alpha)}\preceq \beta^{(\lambda,\vec{v},\alpha)}$. In particular, $(\lambda^{\star},\vec{v},\alpha)$ falls now into case (2) or (3). In case (2), apply the shift-operation of Proposition~\ref{prop:shiftingVec}. This leads to a new triple~$(\lambda',\vec{v}',\alpha)$ with $\beta^{(\lambda',\vec{v}',\alpha)}\preceq\beta^{(\lambda^{\star},\vec{v},\alpha)}.$ Moreover, either $\vec{v}'\in \{\lambda' v^{1,3},\lambda' v^{2,3}\}$ or $v_i'\in F_\rho$ for $i=2,3$. In the first case we get into (3), and in the second case the result directly follows by setting $(\beta',p')=\phi(\lambda',\vec{v}',\alpha)$. Hence, it only remains to prove~(3). Assume $v_3\in \{\lambda v^{1,3},\lambda v^{2,3}\}$. In particular, $u_3\coloneqq v_3/\lambda$ has to be an extremal point of the face $F_\rho$, and $v_2\in \Ss_1^2$ (otherwise the connection $v_3$ to~$\cD_2$ does not intersect~$F_{\rho}$). Similarly, if $u_2\in \Ss_1^2$ is the other extremal point of $F_\rho$, then $v_2=\gamma u_2$ for some $\gamma\in[0,1]$. Therefore, $\rho(d)=\alpha_2\gamma u_2+\alpha_3\lambda u_3$. Since $u_2$ and $u_3$ are the extremal points of $F_\rho$, there is $\hat{\alpha}\in\Delta_1$ such that $\rho(d)=\hat\alpha_2 u_2+\hat\alpha_3 u_3$, i.e. $(1,(u_2,u_3),\hat{\alpha})\in \Cs_d$. Since $u_2$ and $u_3$ are linearly independent, we conclude that $\alpha_2 \gamma=\hat \alpha_2$ and $\alpha_3\lambda=\hat\alpha_3$. It follows that $\beta^{(1,(u_2,u_3),\hat{\alpha})}\preceq \beta^{(\lambda,\vec{v},\alpha)}$, which proves the result.
	
\end{proof}

Finally, we prove that for every selection decomposition that is not $b$-minimal, there is a thinning mechanism $\Tb$ different from the identity and a colouring rule $p'$ such that $(\Tb\beta,p')\in\ccS_d$.

\begin{proof}[Proof of Proposition~\ref{prop:equivalenceminimality(intro)}]
	The proof follows by Proposition~\ref{prop:equivalenceminimality1} together with the scaling property described in Proposition~\ref{prop:scaling-polu}.
\end{proof}

There is a third notion of minimality we have not considered so far, namely, the one induced by the component-wise ordering, i.e. $\beta\preceq_{\textrm{cw}}\beta'$ if and only if $\beta_\ell\leq\beta_{\ell}'$ for all $\ell\in\,]m]$. Clearly,  $$\beta\preceq_{\textrm{cw}}\beta'\Rightarrow \beta\preceq\beta'\Rightarrow b(\beta)\leq b(\beta').$$
The three notions are equivalent for~$m=2$. For $m\geq 3$, this is not any more the case. However, by inspection of the proofs of Propositions~\ref{prop:shrinkinglamd}, \ref{prop:shiftingVec}, and~\ref{prop:equivalenceminimality1}, we see that for $m=3$, we have proved the following stronger version of Proposition~\ref{prop:equivalenceminimality(intro)}.  
\begin{proposition}
	Assume that $\deg(d)=3$. For any $(\beta,p)\in \ccS_d$ with $b(\beta)>b_\star(d)$, there is $(\beta',p')\in\ccS_d$ such that $b(\beta')=b_\star(d)$ and $\beta'\preceq_{\textrm{cw}}\beta$. 
	
\end{proposition}

\addtocontents{toc}{\protect\setcounter{tocdepth}{0}}

\section*{Acknowledgements}
We are grateful to A.~Gonz{\'a}lez Casanova for many interesting discussions, and to Ellen Baake for clarifying to us some aspects of general diploid selection models. Thanks to two anonymous reviewers for their helpful suggestions to improve the manuscript. F. Cordero and S. Hummel received financial support from Deutsche Forschungsgemeinschaft~(CRC 1283 ``Taming Uncertainty'', Project~C1).

\addtocontents{toc}{\protect\setcounter{tocdepth}{2}}
\bibliographystyle{abbrvnat}
\bibliography{Literature}

\end{document}